\documentclass[reqno,11pt]{amsart}
\usepackage[top=2.0cm,bottom=2.0cm,left=2cm,right=2cm]{geometry}
\usepackage{amsthm,amsmath,amssymb,dsfont}
\usepackage{mathrsfs,amsfonts,functan,extarrows,mathtools}

\usepackage[colorlinks,linkcolor=black,citecolor=black]{hyperref}

\usepackage{xcolor}
\usepackage{cite}
\usepackage{indentfirst, latexsym, amssymb, enumerate,amsmath,graphicx}
\usepackage{float}
\usepackage{relsize}
\usepackage{marginnote}
\usepackage{stmaryrd}
\usepackage{esint}
\usepackage{graphicx}
\usepackage{bm}
\usepackage{caption}
\usepackage{subfigure}
\usepackage{cite}
\usepackage{color}
\usepackage{graphicx}
\usepackage{subfigure}
\usepackage{float}
\usepackage{paralist}
\usepackage{indentfirst}
\usepackage{cite}
\allowdisplaybreaks 
\theoremstyle{plain}
\newtheorem{thm}{Theorem}[section]

\newtheorem{lem}[thm]{Lemma}

\newtheorem{rem}{Remark}[section]

\numberwithin{equation}{section}

\usepackage{appendix}
\usepackage{xcolor}

\begin{document}

\title[a compressible fluid-particle interaction model]
{Global strong solutions to a compressible fluid-particle interaction model with density-dependent friction 
force}

\author[F.  Li]{Fucai Li}
\address{School of Mathematics, Nanjing University, Nanjing
 210093, P. R. China}
\email{fli@nju.edu.cn}

\author[J. Ni]{Jinkai Ni$^*$}  \thanks{$^*$\! Corresponding author}
\address{School  of Mathematics, Nanjing University, Nanjing 
 210093, P. R. China}
\email{jinkaini123@gmail.com}

\author[M. Wu]{Man Wu}
\address{Department of Mathematics, Nanjing Audit University, Nanjing
 211815, P. R. China}
\email{manwu@nau.edu.cn}

\begin{abstract}
We investigate the Cauchy problem for a   fluid-particle interaction model
in  $\mathbb{R}^3$.   This model consists of the  compressible  barotropic  Navier–Stokes equations and the Vlasov-Fokker-Planck  equation coupled together via  the density-dependent friction force. Due to the strong coupling caused by the friction force, it is  a challenging problem to construct the global existence and  optimal decay rates of  strong solutions.
 In this paper,  by assuming that the $H^2$-norm of the initial data 
is sufficiently small, we establish the global well-posedness of strong solutions. 
Furthermore, if  the $L^1$-norm of initial data
is bounded, then we achieve the optimal decay rates of strong solutions  and their gradients in $L^2$-norm. 
The proofs  rely on developing  refined energy estimates and  exploiting the  frequency decomposition
method. In addition, for the periodic domain case, our global strong solutions decay exponentially. 

\end{abstract}



\maketitle

\setcounter{equation}{0}
 \indent \allowdisplaybreaks

\section{Introduction}\label{Sec:intro-resul}
\subsection{Fluid-particle models}
Fluid-particle models, where dispersive particles are suspended within a dense fluid, have been extensively applied in various fields, such as    
 dynamics of sprays \cite{BBBDLLT-irma-2005}, combustion theory \cite{Wfa-1958,Wfa-1985}, chemical engineering \cite{CP-siam-1983}, and biomedical sprays \cite{BBJM-esaim-2005}. From a statistical viewpoint,
the behavior of particles is typically characterized by a probability density distribution
function $F=F(t,x,v)\geq 0$ for the time $t
>0$, the position $x\in\Omega\subset \mathbb{R}^3$ and the velocity of the particles $v\in\mathbb{R}^3$,  which  typically adheres to  the  Vlasov–Fokker–Planck (VFP) equation (\!\!\cite{MV-MMMAS-2007}):
\begin{align}\label{B1}
\partial_{t}F+v\cdot\nabla F+\mathrm{div}_{v}[F_dF-\nabla_{v}F]=0.  
\end{align}
Here $F_d$ denotes  the friction force (also known as  drag force)  
\begin{align}\label{fo-aaaa}
   F_d=F_0(u-v) 
\end{align} 
with some  constant $F_0>0$ (usually being chosen as one for simplicity), and 
the unknown function $u=(u_1,u_2,u_3)^{\top}\in\mathbb{R}^3$ represents
the velocity of the fluid.

The formula \eqref{fo-aaaa}  is widely used in  incompressible  homogeneous fluids \cite{CDM-krm-2011,GJV-IUMJ-2004,GJV-IUMJ-2004-2,CKL-JDE-2011,
GHMZ-siamj-2009,Hankwan-2022-PMP,HMM-ARMA-2020},    incompressible inhomogeneous fluids \cite{EHM-nonlinearity-2021,Danchin-kato,SWYZ-24,SWYZ-23,SY-20,JLN-2025-JMP,LNW-SAPM-2025}, and compressible fluids \cite{CKL-JHDE-2013,DL-KRM-2013,Ww-CMS-2024,Mu-Wang-20,MV-CMP-2008,MV-MMMAS-2007,Li-Ni-Wu}.
As Mellet and Vasseur  pointed out in \cite{MV-CMP-2008}, the choice of \eqref{fo-aaaa}
may not be the most
precise from a physical perspective, especially when the density of the fluid is varying, 
for example, incompressible inhomogeneous fluids   and compressible fluids. Thus, it is more appropriate
to assume that $F_d$ depends on the fluid density $n=n(t,x)$, say,  
\begin{align}\label{fo-aa}
 F_d=F_0n(u-v),  
\end{align} 
 which carries greater physical significance.  However, due to the fluid density $n$ involved in \eqref{fo-aa},
 some trilinear terms in fluid-particle models appear. Hence  it is more challenging  to obtain   global solutions 
 of  the  fluid-particle models 
in this situation and only a few rigorous mathematical results are available, see  \cite{Wang-Yu2015,BD-JHDE-2006,Choi-Kwon-15,LMW-SIAM-2017}. 
Wang and Yu \cite{Wang-Yu2015} obtained the global existence of weak solutions to the following system
\begin{equation}\label{wang}
\left\{\begin{aligned}
&\partial_{t}F+v\cdot\nabla F+n\mathrm{div}_{v}[(u-v)F]=0,\\
& \partial_{t} n+\mathrm{div}(nu)=0,  \quad \mathrm{div}u=0,  \\
&\partial_{t}(nu)+\mathrm{div}(nu\otimes u)+\nabla \pi =-n\int_{\mathbb{R}^{3}}(u-v)F {\rm d}v {,} 
 \end{aligned}
 \right.
\end{equation}
 with large data in a bounded domain with the reflection boundary condition. Choi and Kwon \cite{Choi-Kwon-15} obtained the   strong solutions on $[0,T]$ to \eqref{wang}
in  $\mathbb{T}^3$ and $\mathbb{R}^3$ for arbitrary $T>0$ when the initial data
is sufficiently small and regular. Moreover, they also showed that   the velocities of particles and fluid 
tend to the  aligned  values exponentially  provided that the local density
of the particles satisfies a certain integrability condition.
 Baranger and  Desvillettes \cite{BD-JHDE-2006} first obtained the   local-in-time  classical solutions
to the following model
\begin{equation}\label{B2}
\left\{\begin{aligned}
&\partial_{t}F+v\cdot\nabla F+n\mathrm{div}_{v}[(u-v)F]=0,\\
&\partial_{t} n+\mathrm{div}(nu)=0,  \\
&\partial_{t}(nu)+\mathrm{div}(nu\otimes u)+\nabla P =-n\int_{\mathbb{R}^{3}}(u-v)F {\rm d}v {,}   
 \end{aligned}
 \right.
\end{equation}
in $\mathbb{R}^N_x$ when the initial data are sufficiently smooth and their support has suitable properties. Carrillo,  Goudon, and  Lafitte  \cite{CGL-JCP-2008} introduced a modified version  of \eqref{B2}
by adding a viscous term in \eqref{B2}$_1$, i.e. changing $n\mathrm{div}_{v}[(u-v)F]$
to $ n\mathrm{div}_{v}[(u-v)F-\nabla_v F]$. However, no mathematical results on this modified model were presented in \cite{CGL-JCP-2008}. By adding a viscous term $-\mu \Delta u$ in the fluid equation \eqref{B2}$_3$ to the model introduced in \cite{CGL-JCP-2008},
Li, Mu, and Wang \cite{LMW-SIAM-2017} considered the following 
model:
\begin{equation}\label{I-1}
\left\{\begin{aligned}
&\partial_{t}F+v\cdot\nabla F+n\mathrm{div}_{v}[(u-v)F-\nabla_{v}F]=0,\\
&\partial_{t} n+\mathrm{div}(nu)=0,  \\
&\partial_{t}(nu)+\mathrm{div}(nu\otimes u)+\nabla P-\mu\Delta u =-n\int_{\mathbb{R}^{3}}(u-v)F {\rm d}v {,}   
 \end{aligned}
 \right.
\end{equation}
in the whole space $\mathbb{R}^3$ and a torus $\mathbb{T}^3$, respectively.  The pressure $P=P(t,x)$ depends only on
$n$ with $P^{\prime}(\cdot)>0$ (Recall the prototype   $P(n)=c_0 n^{\gamma}$ with $\gamma >1$ and $c_0 >0$
for the isentropic gas case). Additionally, the parameter $\mu>0$ is the viscous coefficient. They  addressed the global well-posedness of classical solutions in the $H^4$ framework when the initial data is a small perturbation of 
the global Maxwellian.   However, no discussions on optimal time-decay rates of 
classical solutions were included in \cite{LMW-SIAM-2017}. In fact, they showed that the solutions   have a   decay rate of  $(1+t)^{-1/2}$ in $L^\infty$-norm in $\mathbb{R}^3$ and decay rate in the   natural $L^2$-norm is missing.

In this paper, we revisit the model \eqref{I-1} in $\mathbb{R}^3$. Our purpose includes  two aspects: (1) diminishing the regularity of  initial data  from $H^4(\mathbb{R}^3)$ to $H^2(\mathbb{R}^3)$; (2) obtaining optimal time-decay rates of strong solutions   and    their gradients in the $L^2$-norm, and 
even in  $L^p$ norm with $p\in [2,6]$.

\subsection{Briefly review of related results}

As mentioned before, most studies on fluid-particle models focus on the  friction force with the form 
\eqref{fo-aaaa}. Below we briefly review some results on both the incompressible  and compressible 
models.

First, 
we consider the following incompressible   model:
\begin{equation}\label{B5}
\left\{\begin{aligned}
&\partial_{t}F+v\cdot\nabla F+\mathrm{div}_{v}[(u-v)F-\nabla_{v}F]=0,\\
&{\rm div}u=0,  \\
&\partial_{t}u+u\cdot\nabla u+\nabla \pi-\Delta u =-\int_{\mathbb{R}^{3}}(u-v)F {\rm d}v.  
 \end{aligned}
 \right.
\end{equation}
 {This model} \eqref{B5} is known as the incompressible
homogeneous Navier-Stokes-Vlasov-Fokker-Planck system. 
Goudon, et al.  \cite{GHMZ-siamj-2009} established the global existence of classical solutions in Sobolev space $H^s$ ($s\geq 2)$ to \eqref{B5} near an equilibrium on a torus $\mathbb{T}^3$.
Chae, Kang, and Lee \cite{CKL-JDE-2011} proved the global existence of weak solutions to \eqref{B5} in $\mathbb{R}^d$ (where $d = 2, 3$). Carrillo, Duan and Moussa \cite{CDM-krm-2011} achieved the time-decay rate of smooth solutions to the inviscid case of \eqref{B5} in $L^2$-norm.
If the $\mathrm{div}_{v}[(u-v)F-\nabla_{v}F]$ in \eqref{B5}$_1$ is changed to $\mathrm{div}_{v}[(u-v)F]$, i.e. the diffusion term $-\Delta_v F$ is neglected, the system \eqref{B5} becomes the so-called 
 Vlasov–Navier–Stokes system and there are  substantial  mathematical works  on it. 
Hamdache \cite{Hk-jjiam-1998}  first established the global existence and  long-time behavior of solutions for the Vlasov-Stokes system in a bounded domain.
Boudin, et al. \cite{BD-jneq-2009} obtained
the global existence of weak solutions to the 
  Vlasov–Navier–Stokes
system in a torus $\mathbb{T}^3$.
Recently, Han-Kwan, Moussa and Moyano \cite{HMM-ARMA-2020} studied large time behavior of the Vlasov-Navier–Stokes system on the torus, assuming that the initial  modulated energy is sufficiently small. Besides, large time behavior of small-data solutions to the Vlasov-Navier–Stokes system on the whole space $\mathbb{R}^3$ was studied in \cite{Hankwan-2022-PMP}. Ertzbischoff,  Han-Kwan, and  Moussa  \cite{EHM-nonlinearity-2021} 
investigated small smooth solutions for the
Vlasov-Navier–Stokes system in a bounded domain $\Omega\subset \mathbb{R}^3$.
Danchin \cite{Danchin-kato} obtained the global existence of Fujita-Kato type solutions and demonstrated that    if the   $L^1$-norm of initial velocity is sufficiently small, then the total energy of the system  decays to $0$ at the optimal rate of $t^{- {3}/{4}}$ in $L^2$-norm.
Goudon,  Jabin,  and  Vasseur \cite{GJV-IUMJ-2004,GJV-IUMJ-2004-2}, and  Han-Kwan and  Michel \cite{HK-M-2024}  studied the  
hydrodynamic limit to  the Vlasov-Navier-Stokes system in different regimes, see also  \cite{SWYZ-23,SWYZ-24}
on the recent progress on the  inhomogeneous incompressible Navier-Stokes-Vlasov(-Fokker-Planck) system.

Next,  we turn to the compressible model:
\begin{equation}\label{B6}
\left\{\begin{aligned}
&\partial_{t}F+v\cdot\nabla F+\mathrm{div}_{v}[(u-v)F-\nabla_{v}F]=0,\\
&\partial_tn+{\rm div}(nu)=0,  \\
&\partial_{t}(nu)+\mathrm{div}(nu\otimes u)+\nabla P-\Delta u =-\int_{\mathbb{R}^{3}}(u-v)F {\rm d}v.  
 \end{aligned}
 \right.
\end{equation}
Mellet and Vasseur \cite{MV-MMMAS-2007} established the global existence of weak solutions of \eqref{B6} in a bounded domain $\Omega\subset\mathbb{R}^3$.
By utilizing relative entropy method, they also conducted the asymptotic analysis of solutions in \cite{MV-CMP-2008}.
Additionally,
the global existence  and the
exponential decay of the classical solutions   to \eqref{B6} were obtained in
\cite{CKL-JHDE-2013}. Based
on the classical energy estimates,
Duan and Liu 
\cite{DL-KRM-2013} investigated the global well-posedness of classical solutions to the inviscid version  of \eqref{B6} (ignoring the term  $-\Delta u$ in \eqref{B6}$_3$)   in the framework of small perturbations.
Very recently, Wang \cite{Ww-CMS-2024} established the global strong solutions in the $H^2$ framework and derived the optimal decay rates for all derivatives of the solutions to \eqref{B6},  including  a term $\nabla{\rm div}u$ in the momentum equation  \eqref{B6}$_3$.
For more results regarding the compressible fluid-particle models, interested  readers can refer
\cite{CG-2006,Li-Ni-Wu,LLY22,Mu-Wang-20} and   references cited therein.

\subsection{Our  results and the strategies in the proofs}

The objective of this paper is to establish the global existence and optimal
time decay rate of strong solutions to the system \eqref{I-1}  near the equilibrium state $(F, n, u)\equiv(M,1,0)$,
where $$  M=M(v)={(2\pi)^{\frac{3}{2}}}e^{-\frac{|v|^{2}}{2}}$$
represents the global Maxwellian.  
We supplement the system \eqref{I-1} with  the initial data
\begin{equation}\label{I-2}
(F,n,u)|_{t=0}=(F_{0}(x,v),n_{0}(x),u_{0}(x)),\quad    (x, v)\in \mathbb{R}^3\times \mathbb{R}^3,
\end{equation}


By setting the standard
perturbation  $n=1+\rho$, $F=M+\sqrt{M}f$, we can rewrite the system \eqref{I-1} and \eqref{I-2} as
\begin{align}\label{I-3}
& \partial_{t}f+v\cdot\nabla_x f+u\cdot\nabla_{v}f-{\frac{1}{2}}u\cdot vf-u\cdot v\sqrt{M}\nonumber\\
&\qquad= \mathcal{L}f+\rho\Big(\mathcal{L}f-u\cdot\nabla_v f+\frac{1}{2}u\cdot vf+u\cdot v\sqrt{M}\Big),\\ \label{I-4}
&\partial_{t}\rho+(\rho+1){\rm div} u+\nabla\rho\cdot u=0,\\ \label{I-5}
&\partial_{t}u+u\cdot\nabla u+{\frac{P^{\prime}(1+\rho)}{1+\rho}}\nabla\rho-\frac{\mu\Delta u}{1+\rho}={b-u-au},
\end{align}
with initial data
\begin{equation}\label{I-6}
(f,\rho,u)|_{t=0}=(f_0(x,v),\rho_0(x),u_0(x))=\bigg(\frac{F_0(x,v)-M}{\sqrt{M}},n_0(x)-1,u_0(x)\bigg).
\end{equation}
Here, $\mathcal{L}$ is the linearized Fokker-Planck operator defined by
\begin{align}\label{A1.7}
\mathcal{L}f=\frac{1}{\sqrt{M}}\mathrm{div}_{v}
\left[M\nabla_{v}\left(\frac{f}{\sqrt{M}}\right)\right]=\Delta_{v}f-\frac{|v|^{2}}{4}f+\frac{3}{2}f,
\end{align}
and $a= a^f, b= b^f$ are corresponding moments of $f$ given by
\begin{align*}
a^{f}(t,x)=\int_{\mathbb{R}^{3}}\sqrt{M}f(t,x,v){\rm d}v,
\quad b^{f}(t,x)=\int_{\mathbb{R}^{3}}v\sqrt{M}f(t,x,v){\rm d}v.
\end{align*}

The main results of this paper can be stated as follows.
\begin{thm}\label{T1.1}
Assume that  the initial data  $(f_0, \rho_0, u_0)$
satisfy   that $\|f_0\|_{H_{x, v}^2}+\|(\rho_0, u_0)\|_{H^2} $  is sufficiently small and $F_0=M+$ $\sqrt{M} f_0 \geq 0$. Then the Cauchy problem \eqref{I-3}--\eqref{I-6} admits a unique global strong solution 
  $(f, \rho, u)$ satisfying $F=M+\sqrt{M} f \geq$ 0 and
\begin{gather*}
   f \in C([0, \infty) ; H_{x,v}^2), \quad\rho, u\in C([0, \infty) ; H^2), \\
  \sup _{t \geq 0}\big\{\|f(t)\|_{H_{x, v}^2}+\|(\rho, u)(t)\|_{H^2}\big\} \leq C\big(\|f_0\|_{H_{x, v}^2}+\|(\rho_0, u_0)\|_{H^2}\big),
\end{gather*}
for some constant $C>0$.    
\end{thm}

\begin{rem}
When addressing the energy estimate of higher-order derivatives of solutions, as per \cite[Proposition 4.2]{LMW-SIAM-2017}, the energy structure of density is $\frac{\sqrt{P^\prime(1+\rho)}}{1+\rho}\|\partial^\alpha\rho\|^2$. This results in the estimate of $\partial_t \rho$. Consequently, a fourth-order regularity of solutions is necessary (refer to the estimate of $I_{14}$ in \cite[Proposition 4.2]{LMW-SIAM-2017}). Unlike \cite{LMW-SIAM-2017}, the energy we have chosen in Lemma \ref{L3.2} is $P^{\prime}(1)\|\partial^\alpha\rho\|^2$. By means of a refined energy method, we can restrict our estimates in the $H^2$ framework.
\end{rem}

\begin{thm}\label{T1.2}
Under the assumption of Theorem \ref{T1.1}, if we further assume that
$\|(f_0, \rho_0, u_0)\|_{\mathcal{Z}_1}$ is bounded, then
\begin{align}
\|f\|_{L^2_v(H^2)}+\|(\rho,u)\|_{H^2} 
&\leq C(1+t)^{-\frac{3}{4}},\label{G1.8}\\
\|\nabla f\|_{L^2_v(H^1)}+\|\nabla(\rho,u)\|_{H^1}  &\leq C(1+t)^{-\frac{5}{4}},\label{G1.9}\\
 \|\nabla^2 f\|_{L^2_v(L^2)}+\|\nabla^2(\rho,u)\|_{L^2} &\leq C(1+t)^{-\frac{5}{4}},\label{G1.10}
\end{align} 
for some constant $C>0$ and any $t \geq 0$. Moreover, 
\begin{align}
\|f\|_{L^2_v(L^p)}+\|(\rho,u)\|_{L^p}&\leq C(1+t)^{-\frac{3}{2}(1-\frac{1}{p})},\label{G1.11}
 \end{align}
 holds for  $2\leq p\leq 6$, and while
 \begin{align}
 \|f\|_{L^2_v(L^p)}+\|(\rho,u)\|_{L^p}&\leq C(1+t)^{-\frac{5}{4}},\label{G1.12}
 \end{align}
 holds for $6\leq p\leq\infty$.
 \end{thm}

\medskip 

In the proof of Theorem \ref{T1.1}, the local existence part   is  standard in some sense and can be established 
by applying  linearization method and Banach's fixed point theorem, while   the argument for global 
existence part relies on establishing newly uniform-in-time  a priori energy estimates.
Compared to the approach developed in \cite{LMW-SIAM-2017} where 
the $H^4$ regularity assumption of smooth solutions simplifies 
the energy estimate process, our constraints on $H^2$ regularity of solutions bring more challenges and 
require us develop new ideas to   enclose the energy estimates, see Section 3 for more details. 

For the proof of Theorem \ref{T1.2},  the nonlinear term $\rho\mathcal{L}f$ in the  equation \eqref{I-3} brings about significant difficulties in estimating the decay rates of the solutions.
For the friction force \eqref{fo-aaaa} considered in \cite{CDM-krm-2011,Ww-CMS-2024},  the nonlinear terms  {in the perturbation of the particle equation } $f$ do not include the term $v^2f$ and  their structures  are relatively simpler than our case \eqref{fo-aa}. Consequently, when conducting energy estimates,   the a priori estimates can be
enclosed without performing energy estimates on the mixed space-velocity
derivatives of solutions, see \cite{CDM-krm-2011,Ww-CMS-2024} for more details.  In contrast, our friction force  \eqref{fo-aa} involves the density $\rho$, and the equation \eqref{I-3} includes the nonlinear terms  {$\rho v^2f$ and  $\rho\mathcal{L}f$, necessitating the energy estimates incorporating the mixed space-velocity derivatives of solutions to manage our nonlinear term $K_5$ (see \eqref{G4.27}). 
 By decomposing $ \rho\mathcal{L} f $  into $\rho\mathcal{L}\{\mathbf{I}-\mathbf{P}\}f$
and $\rho\mathcal{L}\mathbf{P}f$ in Section 4,  
and leveraging the results of Theorem \ref{T4.1}, 
we successfully derive the decay rates of strong solutions
to \eqref{I-3}--\eqref{I-6} in $L^2$-norm under the assumption that the $L^1$-norm of the initial data
is {sufficiently} small.  Since the nonlinear term contains the second derivative terms such as $\Delta u$ and $\Delta_v f$, it is difficult to  get decay rates of the higher derivative
of $(f,\rho, u)$ by applying  the classical  energy method.
To overcome these difficulties, in Section 5, we adopt a more refined low-high frequency decomposition method and an improved energy method,  thereby relaxing  the requirement for the initial data of $L^1$-norm, i.e. only  boundedness 
of $L^1$-norm is needed. 
A crucial step in the subsequent proof of Theorem \ref{T1.2} is
to establish the following inequality (see \eqref{NJKKK5.24}):
\begin{align}\label{NJK}
\frac{\mathrm{d}}{\mathrm{d}t}\mathcal{E}_1(t)+\lambda_{16}\mathcal{E}_1(t)\leq C
\|\nabla^2(a^L,b^L,\rho^L,u^L)\|^2,
\end{align}
where the functional satisfies
\begin{align*}
\mathcal{E}_1(t)\backsim \|\nabla^2 f\|_{L_v^2(L^2)}^2+\|\nabla^2(\rho,u)\|^2.    
\end{align*}
Here we develop a refined low-frequency 
estimate in Theorem \ref{T5.4}, based on the Duhamel principle, 
which exhibits  more information than that in   Theorem \ref{T4.1}. 
By combining Theorem \ref{T5.4} with the inequality \eqref{NJK}, 
we derive the optimal time decay rates of strong
solutions and their gradients. In addition, by utilizing the interpolation inequality, 
we also determine the optimal decay rate $(1+t)^{- {3}/{2}(1- {1}/{p})}$ of solutions
in $L^p$-norm   for $2 \leq p \leq 6 $, and a decay rate  $(1+t)^{-5/4}$ in $L^p$-norm   for $6 \leq p \leq \infty$.
In stark contrast to the density-independent friction force \eqref{fo-aaaa} studied in 
\cite{Ww-CMS-2024,Li-Ni-Wu}, for our  density-dependent friction force case \eqref{fo-aa}, due to the fact that we cannot  fully transform 
the term $\rho\mathcal{L}f$ into an absolute energy term $\mathcal{E}(t)$, 
and consequently derive the corresponding dissipation $\mathcal{D}(t)$,
we are unable to establish the optimal time-decay rates
for the  second-order  spatial derivatives of strong solutions. 

The rest of this paper is structured as follows.
In Section 2, we present some preliminary notations 
and properties that will be used  throughout the paper.
In Section 3, we establish a priori estimates 
for the problem \eqref{I-3}--\eqref{I-6} and prove 
the global existence of strong solutions for the Cauchy 
problem. Time-decay estimates of the solutions are derived
in Section 4 by the classical energy method and in Section 5 by using low-frequency and
high-frequency decomposition techniques along with Fourier analysis.
In Section 6, we demonstrate that when particles lie
in a torus $\mathbb{T}^3$, the time-decay rate of convergence 
of strong solutions is exponential. 
Finally, in the Appendix, we present several useful lemmas.

\section{Preliminaries}
In this section, we   introduce   notations and list
main properties of 
the Fokker-Planck operator $\mathcal{L}$, defined in \eqref{A1.7}, which will 
be used frequently in the subsequent discussions.

\subsection{Notations}

The letter $C$ denotes a generic  positive constant which may change from line to line. The symbol $A\backsim B$ denotes the relation $\frac{A}{C}\leq B\leq CA$ for some general constant $C$.
 We set $\|(g,h)\|_{X}:=\|g\|_{X}+\|h\|_{X}$ for the Banach
space $X$, where the functions $g $ and $h $ belong to $X$.

We denote the weight function $\nu(v):=1+|v|^2$ and   the norm $|\cdot|_{\nu}$ as
\begin{align*}
|g|_{\nu}^{2}:&=\int_{\mathbb{R}^3}(|\nabla_vg(v)|^2+\nu(v)|g(v)|^2){\rm d}v, \quad
g=g(v).
\end{align*}
For the corresponding norm concerning with the dissipation of 
the Fokker-Planck operator, we denote by
\begin{align*}
\|g\|_{\nu}^{2}:&=\int_{\mathbb{R}^3\times\mathbb{R}^3} (|\nabla_{v}g(x,v)|^2
+\nu(v)|g(x,v)|^2){\rm d}x{\rm d}v.    
\end{align*}
For simplicity of notation, we use $\|\cdot\|$ to denote  the norm $\|\cdot\|_{L_x^2}$ or $\|\cdot\|_{L^2_{x,v}}$, and $\langle\cdot,\cdot\rangle$   the 
inner product over space $L_v^2$, i.e.,
\begin{align*}
\langle g,h\rangle=\int_{\mathbb{R}^3}g(v)h(v)\mathrm{d}v,\quad g,h\in L_v^2.   
\end{align*}
For $q\geq 1$, the standard spatial-velocity mixed Lebesgue space
$Z_q=L_v^2(L_x^q)=L^2(\mathbb{R}_v^3;L^q(\mathbb{R}^3_x))$ is defined by
\begin{align*}
\|g\|_{Z_q}^2:=\int_{\mathbb{R}^3}\Big(\int_{\mathbb{R}^3}|g(x,v)|^q\mathrm{d}x\Big)^{\frac{2}{q}}\mathrm{d}v.   
\end{align*}
For any multi-indices $\alpha=(\alpha_{1},\alpha_{2},\alpha_{3})$,
 $\beta=(\beta_{1},\beta_{2},\beta_{3})$, we always denote $\partial_{x}^{\alpha}\partial_{v}^{\beta}
=\partial_{x_{1}}^{\alpha_{1}}\partial_{x_{2}}^{\alpha_{2}}
\partial_{x_{3}}^{\alpha_{3}}\partial^{\beta_{1}}_{v_{1}}\partial^{\beta_{2}}_{v_{2}}\partial^{\beta_{3}}_{v_{3}}$,
the partial derivatives with regard to $x=\left(x_1, x_2, x_3\right)$ and $v=\left(v_1, v_2, v_3\right)$. The lengths of $\alpha$ and $\beta$ are defined as $|\alpha|=\alpha_1+\alpha_2+\alpha_3$ and $|\beta|=\beta_1+\beta_2+\beta_3$. Define
\begin{align*}
\|g\|_{H^s}=\sum_{|\alpha| \leq s}\|\partial^\alpha g\|, \quad\|g\|_{H_{x, v}^s}=\sum_{|\alpha|+|\beta| \leq s}\|\partial_\beta^\alpha g\| .
\end{align*}

For an integrable function $g:\mathbb{R}^{3}\rightarrow\mathbb{R}$, its Fourier transform
 is given by
\begin{equation*}
\widehat{g}(k)=\mathcal{F}g(k)=\int_{\mathbb{R}^{3}}e^{-ix\cdot k}g(x){\rm d}x, 
\quad x\cdot k=\sum\limits_{j=1}^{3}x_{j}k_{j}, 
\end{equation*}
for $k\in\mathbb{R}^{3}$.
Here, $i=\sqrt{-1}\in\mathbb{C}$ is the imaginary unit. 
For two complex functions $f$ and $g$,  
 $(g|h):=g\cdot\overline{h}$ 
represents the dot product of $g$ with the complex conjugate of $h$.

Finally, we denote norms $\|\cdot\|_{\mathcal{Z}_{q}}$, $\|\cdot\|_{\mathcal{H}^{m}}$ with the integer $m\geq 0$ and $q\geq 1$  by
\begin{align*}
\|(f,\rho, u)\|_{\mathcal{Z}_{q}}&=\|f\|_{Z_{q}}+\|\rho\|_{L^{q}}+\|u\|_{L^{q}}, \\
\|(f,\rho, u)\|_{\mathcal{H}^{m}}&=\|f\|_{L^{2}_{v}(H^m)}+\|\rho\|_{H^{m}}+\|u\|_{H^{m}}.
\end{align*}

\subsection{Fokker-Planck operator}
Motivated by \cite{DFT-2010-CMP,GY-iumj-2004}, we introduce the  macro-micro decomposition to $f(t,x,v)$, which plays a crucial role in the  following estimates.
%
Below we first introduce some properties of the velocity
 orthogonal projection $\mathbf{P}$:
\begin{equation*}
\mathbf{P}:L_{v}^{2}(\mathbb{R}^3)\rightarrow \mathrm{Span}\{\sqrt{M},v_{j}\sqrt{M}\},
\end{equation*}
for $j=1,2,3$,
and
\begin{align*}
\mathbf{P}:=\mathbf{P}_{0}\oplus \mathbf{P}_{1},
\end{align*}
where $\mathbf{P}_{0}f
:=a\sqrt M$ and $\mathbf{P}_{1}f:=b\cdot v\sqrt M$.

Moreover, $f(t,x,v)$ can be decomposed by 
\begin{align}\label{G2.1}
f(t,x,v)=\mathbf{P}f+\{\mathbf{I}-\mathbf{P}\}f.
\end{align}
Here,  $\mathbf{P}f$, $\{\mathbf{I}-\mathbf{P}\}f$ denote the macro part and the micro part respectively. 
Meanwhile, $\mathcal{L}f$ has the following decomposition
\begin{align}\label{G2.2}
\mathcal{L}f=\mathcal{L}\{\mathbf{I}-\mathbf{P}\}f+\mathcal{L}\mathbf{P}f
=\mathcal{L}\{\mathbf{I}-\mathbf{P}\}f-\mathbf{P}_{1}f.
\end{align}
Noticing that $\mathcal{L}$ is self-adjoint,   there exists a positive constant $\lambda_0>0$ such that
(cf.  \cite{LMW-SIAM-2017,DFT-2010-CMP,CDM-krm-2011}) 
\begin{align}\label{G2.3}
-\langle \mathcal{L}f, f\rangle
&\geq\lambda_{0}|\{\mathbf{I}-\mathbf{P}_0\}f|_{\nu}^{2},\nonumber\\
-\langle \mathcal{L}\{\mathbf{I}-\mathbf{P}\}f,f\rangle
&\geq\lambda_{0}|\{\mathbf{I}-\mathbf{P}\}f|_{\nu}^{2}, \nonumber\\
-\langle \mathcal{L}f, f\rangle
&\geq\lambda_{0}|\{\mathbf{I}-\mathbf{P}\}f|_{\nu}^{2}+|b|^{2}.
\end{align}

\section{Global existence of strong solutions} 

In this section, we shall establish the global existence of strong solutions to the Cauchy problem 
\eqref{I-3}--\eqref{I-6}. 
We first  establish the uniform a priori estimates of $f$, $\rho$ and $u$.

\subsection{A priori estimates}
For this, we shall suppose that \eqref{I-3}--\eqref{I-6} admits
a strong solution $(f,\rho,u)$ to the Cauchy problem 
 on $0\leq t\leq T$ with some $T>0$. And  the solution  $(f,\rho, u)$ satisfies 
\begin{align}\label{G3.1}
\sup_{0\leq t\leq T}\|(f,\rho,u)\|_{\mathcal{H}^{2}}\leq\delta,
\end{align}
where $0<\delta<1$ is a sufficiently small generic constant. 
\begin{lem}\label{L3.1}
For the strong solution $(f,\rho, u)$ to the Cauchy problem \eqref{I-3}--\eqref{I-6}, there exists a positive constant $\lambda_1$ such that
{
\begin{align}\label{G3.2}
&\frac{1}{2}\frac{\rm d}{{\rm d}t}\big(\|f\|^{2}+P^{\prime}(1)\|\rho\|^{2}+\|u\|^{2}\big)+\lambda_1(\|b-u\|^{2}
+\|\nabla u\|^2+\|\{\mathbf{I}-\mathbf{P}\}f\|_{\nu}^{2})
\leq C\delta\|\nabla(a,b,\rho)\|^2,
\end{align}}
for any $0 \leq t \leq T$. 
\end{lem}
\begin{proof}
Multiplying \eqref{I-3}--\eqref{I-5} by $f$, $P^{\prime}(1)\rho$ and $u$ respectively and then taking integration and summing   up, we arrive at 
\begin{align}\label{G3.3}
&\frac{1}{2}\frac{\rm d}{{\rm d}t}\big(\|f\|^{2}+P^{\prime}(1)\|\rho\|^2+\|u\|^{2}\big)+\|b-u\|^{2}
+\int_{\mathbb{R}^3}\frac{\mu|\nabla u|^2}{1+\rho}  \mathrm{d}x\nonumber\\
&+\int_{\mathbb{R}^3}(\rho+1)\langle-\mathcal{L}\{\mathbf{I}-\mathbf{P}\}f,f\rangle {\rm d}x \nonumber\\
=\,&-\int_{\mathbb{R}^3}(u\cdot\nabla u)\cdot u{\rm d}x+\int_{\mathbb{R}^3} u\cdot
\Big\langle\frac{1}{2}vf, f\Big\rangle {\rm d}x-\int_{\mathbb{R}^3}a|u|^2{\rm d}x-\frac{P^{\prime}(1)}{2}\int_{\mathbb{R}^3}\rho^2{\rm div}u\mathrm{d}x\nonumber\\
&\,-\int_{\mathbb{R}^3}\Big(\frac{P^{\prime}(\rho+1) }{\rho+1}-P^{\prime}(1)  \Big)\nabla\rho\cdot u \mathrm{d}x-\mu\int_{\mathbb{R}^3}\nabla\Big(\frac{1}{1+\rho}\Big)\cdot\nabla u\cdot u   \mathrm{d}x\nonumber\\
&\,+\int_{\mathbb{R}^3}\rho\langle \mathcal{L}\mathbf{P}f-u\cdot\nabla_v f+\frac{1}{2}u\cdot vf+u\cdot v\sqrt{M},f\rangle{\rm d}x\nonumber\\
\equiv:& \sum_{j=1}^7 I_j.
 \end{align}
For the term $I_1$, using \eqref{G3.1}, Lemma \ref{LA.2}, and  H\"{o}lder’s inequality, 
 we have
\begin{align}\label{G3.4}
I_1\leq C\|u\|_{L^3}\|u\|_{L^6} \|\nabla u\|\leq C\|u\|_{H^1}\|\nabla u\|^2
\leq C\delta \|\nabla u\|^2.
\end{align}
For the terms $I_2$ and $I_3$, it follows from \eqref{G2.1} that
\begin{align}\label{G3.5}
I_2+I_3=\,&\frac{1}{2}\int_{\mathbb{R}^3}u\cdot\langle v\{\mathbf{I}-\mathbf{P}\}f,\{\mathbf{I}-\mathbf{P}\}f\rangle\mathrm{d}x+\int_{\mathbb{R}^3}u\cdot\langle v\mathbf{P}f,\{\mathbf{I}-\mathbf{P}\}f\rangle\mathrm{d}x\nonumber\\
&+\frac{1}{2}\int_{\mathbb{R}^3}u\cdot\langle v\mathbf{P}f,\mathbf{P}f\rangle\mathrm{d}x-\int_{\mathbb{R}^3}a|u|^2\mathrm{d}x\nonumber\\
=\,&\int_{\mathbb{R}^3}au\cdot(b-u)\mathrm{d}x+\frac{1}{2}\int_{\mathbb{R}^3}u\cdot\langle v\{\mathbf{I}-\mathbf{P}\}f,\{\mathbf{I}-\mathbf{P}\}f\rangle\mathrm{d}x+\int_{\mathbb{R}^3}u\cdot\langle v\mathbf{P}f,\{\mathbf{I}-\mathbf{P}\}f\rangle\mathrm{d}x\nonumber\\
\leq\,&C\|b-u\|\|u\|_{L^3}\|a\|_{L^6}+C\|(a,b)\|_{L^6}\|\{\mathbf{I}-\mathbf{P}\}f\|_{\nu}\|u\|_{L^3}+C\|u\|_{L^{\infty}}\|\{\mathbf{I}-\mathbf{P}\}f\|_{\nu}^2\nonumber\\
\leq\,&C\delta\big(\|b-u\|^2+\|\nabla(a,b)\|^2+ \|\{\mathbf{I}-\mathbf{P}\}f\|_{\nu}^2        \big),
\end{align}
where we have used  H\"{o}lder’s and Young's inequalities, \eqref{G3.1} and Lemmas \ref{LA.1}--\ref{LA.2}.
Notice that the fact that $H^2\hookrightarrow L^{\infty}$ and the assumption \eqref{G3.1}, 
we have 
\begin{align}\label{rho-1}
 \frac{1}{2}\leq\rho+1\leq\frac{3}{2}.  
\end{align} Thus, similar to \eqref{G3.4}--\eqref{G3.5}, the terms 
$I_4$, $I_5$, and $I_6$ can be estimated as 
\begin{align}\label{G3.6}
I_4\leq\,& C \|\rho\|_{L^3}\|\rho\|_{L^6}\|\nabla u\|\leq C\|\rho\|_{H^1}\|\nabla(\rho,u)\|^2 \leq C\delta \|\nabla(\rho,u)\|^2,\\\label{G3.7}
I_5\leq\,& C \|\rho\|_{L^3}\| u\|_{L^6}\|\nabla\rho\|\leq C\|\rho\|_{H^1}\|\nabla(\rho,u)\|^2 \leq C\delta \|\nabla(\rho,u)\|^2,\\\label{G3.8}
I_6\leq\,&C\|\nabla\rho\|_{L^3}\|u\|_{L^6}\|\nabla u\|_{L^2}\leq 
C\|\rho\|_{H^2}\|\nabla u\|^2\leq C\delta \|\nabla u\|^2.
\end{align}
For the last term $I_7$, by leveraging the decomposition  \eqref{G2.2}, and applying 
H\"{o}lder’s, Sobolev's and Young's inequalities and Lemmas \ref{LA.1}--\ref{LA.2} again,
we derive that
\begin{align}\label{G3.9}
I_7=\,&\int_{\mathbb{R}^3}\rho(u-b)\cdot b\mathrm{d}x+\frac{1}{2}\int_{\mathbb{R}^3}\rho u\cdot\langle vf,f\rangle\mathrm{d}x\nonumber\\
\leq\,& C\|\rho\|_{L^3}\|u-b\|\|b\|_{L^6}+\frac{1}{2}\int_{\mathbb{R}^3}\rho u\cdot\langle v\{\mathbf{I}-\mathbf{P}\}f,\{\mathbf{I}-\mathbf{P}\}f\rangle\mathrm{d}x\nonumber\\
&+\int_{\mathbb{R}^3}\langle v\mathbf{P}f,\{\mathbf{I}-\mathbf{P}\}f\rangle\mathrm{d}x+\int_{\mathbb{R}^3}\rho u \cdot ab\mathrm{d}x\nonumber\\
\leq\,& C\|\rho\|_{H^1}\|u-b\|\|\nabla b\|+C\|\rho\|_{H^2}\|u\|_{H^2}\|\{\mathbf{I}-\mathbf{P}\}\|_{\nu}^2\nonumber\\
&+C\|\rho\|_{H^1}\|u\|_{H^1}\|(a,b)\|_{L^6}\|\{\mathbf{I}-\mathbf{P}\}\|_{\nu}+C\|\rho\|_{H^1}\|u\|_{H^1}\|\nabla(a,b)\|\nonumber\\
\leq\,&C\delta( \|u-b\|^2+\|\nabla(a,b)\|^2+\|\{\mathbf{I}-\mathbf{P}\}\|_{\nu}^2).
\end{align}
 Then, by
substituting the estimates \eqref{G3.4}--\eqref{G3.9} into \eqref{G3.3} and using
\eqref{G2.3},
we ultimately obtain  the desired estimate \eqref{G3.2}.
\end{proof}

\begin{lem}\label{L3.2}
For the strong solution $(f,\rho, u)$ to the Cauchy problem \eqref{I-3}--\eqref{I-6}, there exists a positive constant $\lambda_2$ such that
\begin{align}\label{G3.10}
&\frac{1}{2}\frac{\rm d}{{\rm d}t}\sum_{1\leq |\alpha|\leq 2}
\big(\|\partial^\alpha f\|^2+P^{\prime}(1)\|\partial^\alpha \rho\|^2+\|\partial^\alpha u\|^2\big)\nonumber\\
&\quad +\lambda_2\sum_{1\leq|\alpha|\leq 2}\big(\|\partial^\alpha(b-u)\|^2+\|\nabla\partial^\alpha u\|+\|\partial^\alpha\{\mathbf{I}-\mathbf{P}\}f\|_{\nu}^2\big) 
 \leq C\delta\|\nabla(a,b,\rho,u)\|_{H^1}^2,
\end{align}
for any $0 \leq t \leq T$. 
\end{lem}

\begin{proof}
Applying $\partial^{\alpha}$ with $1\leq |\alpha|\leq2$ to \eqref{I-3}--\eqref{I-5},
we find that 
\begin{equation}\label{G3.11}
\left\{\begin{aligned}
&\partial_{t}(\partial^\alpha f)+v\cdot\nabla (\partial^\alpha f)+u\cdot\nabla_v (\partial^{\alpha}f)-\partial^{\alpha}u\cdot v\sqrt{M}-\mathcal{L}(\partial^\alpha f)\\
&\quad =[-\partial^\alpha,u\cdot\nabla_v]f+\frac{1}{2}\partial^\alpha\big((1+\rho)u\cdot vf\big)+\partial^{\alpha}\big(\rho(\mathcal{L}f-u\nabla_v f+u\cdot v\sqrt{M})\big),\\
&\partial_{t} (\partial^\alpha\rho)+u\cdot\nabla\partial^\alpha\rho+(1+\rho){\rm div}\partial^\alpha u 
  =[-\partial^\alpha,\rho{\rm div}]u+[-\partial^\alpha,u\cdot\nabla]\rho,  \\
&\partial_{t}(\partial^\alpha u)+u\cdot\nabla(\partial^\alpha u)+\partial^\alpha\Big(\Big(\frac{P^{\prime}(1+\rho)}{1+\rho}-{P^{\prime}(1)}\Big)\nabla\rho\Big)
-\partial^\alpha\Big(\frac{\mu}{1+\rho}\Delta u\Big)\\
&\quad =[-\partial^\alpha,u\cdot\nabla]u+\partial^\alpha(b-u-au).
 \end{aligned}
 \right.
\end{equation}
Here, $[A, B]=AB-BA$ denotes the commutator of two operators $A$ and $B$.

Multiplying the equations in \eqref{G3.11} by $\partial^\alpha f$, $\partial^\alpha \rho$ and
$\partial^\alpha u$ respectively, then taking integration and summation,
we arrive at
\begin{align}\label{G3.12}
&\frac12\frac {\rm d}{{\rm d}t}\big(\|\partial^{\alpha}f\|^{2}+P^{\prime}(1)\|\partial^{\alpha}\rho\|^{2}+\|\partial^{\alpha}u\|^{2}\big)
+\|\partial^{\alpha}(b-u)\|^{2}
 \nonumber\\
&+\int_{\mathbb{R}^3}\left\langle-\mathcal{L}\{\mathbf{I}-\mathbf{P}\}\partial^{\alpha}f,
\partial^{\alpha}f\right\rangle {\rm d}x+\, \int_{\mathbb{R}^3}\frac{\mu}{1+\rho}|\nabla\partial^\alpha u|^2  \mathrm{d}x                  \nonumber\\
=\,&\int_{\mathbb{R}^3}\langle[-\partial^{\alpha},u\cdot\nabla_{v}]f,\partial^{\alpha}f\rangle {\rm d}x
+\frac{1}{2}\int_{\mathbb{R}^3}
\langle\partial^{\alpha}\big((1+\rho)u\cdot vf\big),\partial^{\alpha}f\rangle {\rm d}x\nonumber\\
&+\int_{\mathbb{R}^3}
\langle\partial^{\alpha}\big(\rho(\mathcal{L}f-u\nabla_v f+u\cdot v\sqrt{M})\big),\partial^{\alpha}f\rangle {\rm d}x+\frac{1}{2}P^{\prime}(1)\int_{\mathbb{R}^3}{\rm div}u|\partial^\alpha\rho|^2\mathrm{d}x\nonumber\\
&-P^{\prime}(1)\int_{\mathbb{R}^3}\rho\partial^\alpha \rho {\rm div}\partial^{\alpha} u\mathrm{d}x+P^{\prime}(1)\int_{\mathbb{R}^3}\partial^{\alpha}\rho[-\partial^\alpha,\rho{\rm div}]u\mathrm{d}x+\frac{1}{2}\int_{\mathbb{R}^3}|\partial^{\alpha} u|^2{\rm div}u\mathrm{d}x\nonumber\\
&-\int_{\mathbb{R}^3}\partial^\alpha
\Big(\Big(\frac{P^{\prime}(1+\rho)}{1+\rho}-{P^{\prime}(1)}\Big)\nabla\rho\Big)\cdot\partial^\alpha u\mathrm{d}x\nonumber\\
&+\sum_{0\leq|\beta|<|\alpha|}C_{\alpha,\beta}
\int_{\mathbb{R}^3}\partial^{\alpha-\beta}\Big(\frac{\mu}{1+\rho}\Big)\partial^\beta\Delta u\cdot\partial^\alpha u \mathrm{d}x\nonumber\\
&+\int_{\mathbb{R}^3}\partial^{\alpha}(au)\cdot\partial^{\alpha}u\mathrm{d}x
-\int_{\mathbb{R}^3}\nabla\Big(\frac{\mu}{1+\rho}\Big)\nabla\partial^\alpha u\cdot \partial^{\alpha} u   \mathrm{d}x\nonumber\\
\equiv:\,&\sum_{j=8}^{18} I_j,
\end{align}
where $C_{\alpha,\beta}$ denote  constants depending only on $\alpha $ and $\beta$.
According to \eqref{G3.1}, Lemmas \ref{LA.1}--\ref{LLA.3},  H\"{o}lder’s and Young's inequalities,
we can estimate $I_8$ and $I_9$ as   
\begin{align}\label{G3.13}
I_8=&\,\int_{\mathbb{R}^3}\langle\partial^\alpha(u f),\nabla_v\partial^\alpha f\rangle\mathrm{d}x\nonumber\\
\leq&\, C\|\nabla_v
\partial^\alpha f\|\|\nabla u\|_{H^1}\|\nabla f\|_{L_v^2(H_x^1)}\nonumber\\ 
\leq&\,C\|u\|_{H^2}\bigg(\sum_{1\leq|\alpha|\leq 2}\|\partial^\alpha\{\mathbf{I}-\mathbf{P}\}f\|_{\nu}^2+\|\nabla(a,b)\|_{H^1}^2\bigg)\nonumber\\
\leq&\,C\delta\bigg(\sum_{1\leq|\alpha|\leq 2}\|\partial^\alpha\{\mathbf{I}-\mathbf{P}\}f\|_{\nu}^2+\|\nabla(a,b)\|_{H^1}^2\bigg),\\\label{G3.14}
I_9\leq&\, C\|\nabla\big((1+\rho)u\big)\|_{H^1}\|\nabla f\|_{L_v^2(H^1)}\|v\partial^\alpha f\|\nonumber\\
\leq&\, C(\|\nabla u\|_{H^1}+\|\nabla\rho\|_{H^1}\|\nabla u\|_{H^1})
\bigg(\sum_{1\leq|\alpha|\leq 2}\|\partial^\alpha\{\mathbf{I}-\mathbf{P}\}f\|_{\nu}^2+\|\nabla(a,b)\|_{H^1}^2\bigg)\nonumber\\
\leq&\,C\delta\bigg(\sum_{1\leq|\alpha|\leq 2}\|\partial^\alpha\{\mathbf{I}-\mathbf{P}\}f\|_{\nu}^2+\|\nabla(a,b)\|_{H^1}^2\bigg).
\end{align}
For the term $I_{10}$, it follows from \eqref{A1.7}, \eqref{G3.1} and Lemmas
\ref{LA.1}--\ref{LA.2} that
\begin{align}\label{G3.15}
I_{10}\leq &\,C\Big\|\partial^\alpha\Big(\rho\sqrt{M}\nabla_v\Big(\frac{f}{\sqrt{M}}\Big)\Big)\Big\|
\Big\|\partial^\alpha\Big(\sqrt{M}\nabla_v\Big(\frac{f}{\sqrt{M}}\Big)\Big)\Big\|\nonumber\\
&\,+C\|\partial^\alpha(\rho u f)\|\|\nabla_v \partial^\alpha f\|+C\|\partial^\alpha(\rho u)\|\|\partial^\alpha b\|\nonumber\\
\leq&\,C(\|\nabla\rho\|_{H^1}+\|\nabla u\|_{H^1}\|\nabla\rho\|_{H^1})
\bigg(\sum_{1\leq|\alpha|\leq 2}\|\partial^\alpha\{\mathbf{I}-\mathbf{P}\}f\|_{\nu}^2+\|\nabla(a,b,u)\|_{H^1}^2\bigg)\nonumber\\
\leq&\,C\delta\bigg(\sum_{1\leq|\alpha|\leq 2}\|\partial^\alpha\{\mathbf{I}-\mathbf{P}\}f\|_{\nu}^2+\|\nabla(a,b,u)\|_{H^1}^2\bigg).
\end{align}
By utilizing \eqref{G3.1},  \eqref{rho-1}, Lemmas \ref{LA.1}--\ref{LLA.3}, and H\"{o}lder’s and Young's inequalities, the terms $I_j (j=11,\dots, 18)$ can be estimated as  
\begin{align}
I_{11}\leq\,& C\|{\rm div}u\|_{L^{\infty}}\|\partial^\alpha\rho\|^2\leq C\|\nabla u\|_{H^2}\|\nabla \rho\|_{H^1}^2\leq C\delta(\|\nabla\rho\|_{H^1}^2+\|\nabla u\|_{H^2}^2),\label{G3.16-1}\\
I_{12}\leq\,& C\|\rho\|_{L^{\infty}}\|\partial^\alpha\rho\|\|\nabla\partial^\alpha u\|
\leq C\|\rho\|_{H^2}\|\nabla \rho\|_{H^1}\|\nabla u\|_{H^2}\leq C\delta(\|\nabla\rho\|_{H^1}^2+\|\nabla u\|_{H^2}^2),\label{G3.16-2}\\
I_{13}\leq\,&C\|\partial^\alpha\rho\|(\|\nabla\rho\|_{L^3}\|\partial^\alpha u\|_{L^6}+\|\nabla u\|_{L^{\infty}}\|\partial^\alpha \rho\|)\leq C\|\rho\|_{H^2}(\|\nabla\rho\|_{H^1}^2+\|\nabla u\|_{H^2}^2)\nonumber\\
\leq\,& C\delta (\|\nabla\rho\|_{H^1}^2+\|\nabla u\|_{H^2}^2),\label{G3.16-3}\\
I_{14}\leq\,&C\|\partial^\alpha u\|^2\|\nabla u\|_{L^{\infty}}\leq 
C\|u\|_{H^2}\|\nabla u\|_{H^2}^2\leq C\delta \|\nabla u\|_{H^2}^2,\label{G3.16-4}\\
I_{15}=\,&\int_{\mathbb{R}^3}\partial^{\alpha-1}
\Big(\Big(\frac{P^{\prime}(1+\rho)}{1+\rho}-{P^{\prime}(1)}\Big)\nabla\rho\Big)\cdot\partial^{\alpha+1} u\mathrm{d}x\nonumber\\
\leq\,&C\|\partial^{\alpha+1}u\|(\| \rho\|_{L^{\infty}}\|\nabla\partial^2\rho\|+\|\rho\|_{H^2}\|\nabla\rho\|_{L^6})\leq C\|\rho\|_{H^2}\|\nabla \rho\|_{H^1}\|\nabla u\|_{H^2}\nonumber\\
\leq\,& C\delta(\|\nabla\rho\|_{H^1}^2+\|\nabla u\|_{H^2}^2),\label{G3.16-5}\\
I_{16}\leq\,&C\Big\|\partial^\alpha\Big(\frac{1}{1+\rho}   \Big)\Big\|\|\Delta u\|_{L^3}\|\partial^\alpha u\|_{L^6}+C\Big\|\partial^{\alpha-1}\Big(\frac{1}{1+\rho} \Big)\Big\|_{L^3}\|\nabla\Delta u\|\|\partial^\alpha u\|_{L^6}\nonumber\\
\leq\,&C\|\rho\|_{H^2}\|\nabla u\|_{H^2}^2\leq C\delta \|\nabla u\|_{H^2}^2,\nonumber\\
I_{17}\leq\,& C\|\partial^\alpha(au)\|\|\partial^\alpha u\|\leq C\|\nabla a\|_{H^1}\|\nabla u\|_{H^1}\|\nabla u\|_{H^1}\leq C\delta \|\nabla u\|_{H^2}^2,\label{G3.16-6}\\ 
I_{18}\leq\, & C\Big\|\nabla \Big(\frac{1}{1+\rho} \Big)\Big\|_{L^3}\|\nabla \partial^\alpha u\|\|\partial^\alpha u\|_{L^6}\leq C\|\rho\|_{H^2}\|\nabla u\|_{H^2}^2\leq C\delta \|\nabla u\|_{H^2}^2.\label{G3.16}
\end{align}
Putting estimates \eqref{G3.13}--\eqref{G3.16} into \eqref{G3.12},
 {summing them up}  over $1\leq|\alpha|\leq 2$, then
using \eqref{G2.3} again, we get the desired \eqref{G3.10}.
\end{proof}

Motivated by \cite{CDM-krm-2011,LMW-SIAM-2017}, in order to get the estimates of the energy dissipation rate $\|\nabla(a,b)\|_{H^1}$, we consider  the following equations of $a$ and $b$:
\begin{equation}\label{G3.17}
\left\{\begin{aligned}
&\partial_{t}a+\nabla\cdot b=0,\\
&\partial_{t} b_i+\partial_{i} a+\sum_{j=1}^3\partial_j\Gamma_{ij}(\{\mathbf{I}-\mathbf{P}\}f)=(1+\rho)(u_i-b_i)+(1+\rho)u_ia,  \\
&\partial_{i}b_j+\partial_j b_i-(1+\rho)(u_ib_j+u_jb_i)=-\partial_t \Gamma_{ij}(\{\mathbf{I}-\mathbf{P}\}f)+\Gamma_{ij}(\mathfrak{l}+\mathfrak{r}+\mathfrak{s}),  
 \end{aligned}
 \right.
\end{equation}
for $1\leq i,j\leq 3$,
where $\mathfrak{l},\mathfrak{r}$ ,and $\mathfrak{s}$ are given by
\begin{align}
\mathfrak{l}& :=\mathcal{L}\{\mathbf{I}-\mathbf{P}\}f-v\cdot\nabla\{\mathbf{I}-\mathbf{P}\}f, \label{lrs-1}\\
\mathfrak{r}& :=-u\cdot\nabla_v\{\mathbf{I}-\mathbf{P}\}f+\frac{1}{2}u\cdot v\{\mathbf{I}-\mathbf{P}\}f,\label{lrs-2}\\
\mathfrak{s}& :=\frac{\rho}{\sqrt{M}}\nabla_v\cdot\big(\nabla_v\big(\sqrt{M}\{\mathbf{I}-\mathbf{P}\}f\big)
+v\sqrt{M}\{\mathbf{I}-\mathbf{P}\}f-u\sqrt{M}\{\mathbf{I}-\mathbf{P}\}f             \big), \label{lrs-3}
\end{align}
and 
\begin{align}\label{gammaa}
\Gamma_{ij}(\cdot):=\big\langle (v_iv_j-1)\sqrt{M},\cdot   \big\rangle 
\end{align}
is the moment functional.
It is obvious that \eqref{G3.17}$_1$--\eqref{G3.17}$_2$ can be obtained
by multiplying by$\sqrt{M}$, $v_i\sqrt{M}(1\leq i \leq3)$ respectively
and then taking  integration over $\mathbb{R}_v^3$. In order to
derive \eqref{G3.17}$_3$,  we take \eqref{I-3}  in the form of
\begin{align}\label{G3.18}
&\partial_t\mathbf{P}f+v\cdot\nabla_v \mathbf{P}f+(1+\rho)u\cdot\nabla_v \mathbf{P}f\nonumber\\
&\quad -\frac{1}{2}(1+\rho) u\cdot v\mathbf{P}f+(1+\rho)(b\cdot v\sqrt{M}-u\cdot v\sqrt{M}) 
=-\partial_t \{\mathbf{I}-\mathbf{P}\}f+\mathfrak{l}+\mathfrak{r}+\mathfrak{s}. 
\end{align}
Then, by applying $\Gamma_{ij}$ to \eqref{G3.18}, we directly get \eqref{G3.17}$_3$.

Inspired by \cite{Ww-CMS-2024,CDM-krm-2011} , we denote the temporal functional $\mathcal{E}_0(t)$ as
\begin{align*}
\mathcal{E}_0(t):=\sum_{|\alpha|\leq 1}\sum_{i,j=1}^3\int_{\mathbb{R}^3}\partial^\alpha(\partial_ib_j
+\partial_jbi)\partial^{\alpha}\Gamma_{ij}(\{\mathbf{I}-\mathbf{P}\}f)\mathrm{d}x-\sum_{|\alpha|\leq 1}\partial^\alpha a\partial^{\alpha}\nabla\cdot b\mathrm{d}x. 
\end{align*}

We have  
\begin{lem}\label{L3.3}
For the strong solution $(f,\rho, u)$ to the Cauchy problem \eqref{I-3}--\eqref{I-6}, there exists a positive constant $\lambda_3$ such that
\begin{align}\label{G3.19}
\frac{\rm d}{{\rm d}t}\mathcal{E}_{0}(t)+\lambda_3\|\nabla(a,b)\|_{H^{1}}^{2}
\leq&\, C\big(\|\{\mathbf{I}-\mathbf{P}\}f\|_{L_{v}^{2}(H^{2})}^{2}+\|b-u\|_{H^{1}}^{2}\big), 
\end{align} 
for any $0 \leq t \leq T$.  
\end{lem}

\begin{proof}
On  {the} one hand, we have
\begin{align}\label{G3.20}
\sum_{i,j=1}^3\|\partial^\alpha(\partial_ib_j+\partial_jb_i)\|^2=2\|\nabla\partial^\alpha b\|^2+2\|\nabla\cdot\partial^\alpha b\|^2,   
\end{align}
for any $|\alpha|\leq 1$.
On the other hand, by using \eqref{G3.17}$_3$, it gives
\begin{align}\label{G3.21}
&\sum_{i,j=1}^3\|\partial^\alpha(\partial_ib_j+\partial_jb_i)\|^2\nonumber\\
=\,&\sum_{i,j=1}^3\int_{\mathbb{R}^3}\partial^\alpha(\partial_ib_j+\partial_jb_i)
\partial^\alpha\big((1+\rho)(u_ib_j+u_jb_i)-\partial_t \Gamma_{ij}(\{\mathbf{I}-\mathbf{P}\}f)+\Gamma_{ij}(\mathfrak{l}+\mathfrak{r}+\mathfrak{s})\big)\mathrm{d}x        \nonumber\\
=\,&-\frac{{\rm d}}{{\rm d}t}\sum_{i,j=1}^3\int_{\mathbb{R}^3}\partial^\alpha(\partial_ib_j+\partial_jb_i)\partial^\alpha\Gamma_{ij}(\{\mathbf{I}-\mathbf{P}\}f)\mathrm{d}x\nonumber\\
&+\sum_{i,j=1}^3\int_{\mathbb{R}^3}\partial^\alpha(\partial_i\partial_t b_j+\partial_j\partial_t b_i)\partial^\alpha\Gamma_{ij}(\{\mathbf{I}-\mathbf{P}\}f)\mathrm{d}x\nonumber\\
&+\sum_{i,j=1}^3\int_{\mathbb{R}^3}\partial^\alpha(\partial_ib_j+\partial_jb_i)\partial^\alpha
\big((1+\rho)(u_ib_j+u_jb_i)+\Gamma_{ij}(\mathfrak{l}+\mathfrak{r}+\mathfrak{s})   \big)\mathrm{d}x.
\end{align}
It follows from \eqref{G3.1}, \eqref{G3.17}$_2$, Lemmas \ref{LA.1}--\ref{LA.2} and Young's inequality that
\begin{align}\label{G3.22}
&\sum_{i,j=1}^3\int_{\mathbb{R}^3}\partial^\alpha(\partial_i\partial_t b_j+\partial_j\partial_t b_i)\partial^\alpha\Gamma_{ij}(\{\mathbf{I}-\mathbf{P}\}f)\mathrm{d}x \nonumber\\ 
=\,& -2\sum_{i,j=1}^3\int_{\mathbb{R}^3}\partial^\alpha\partial_t b_i
\partial^\alpha\partial_j\Gamma_{ij}(\{\mathbf{I}-\mathbf{P}\}f)\mathrm{d}x\nonumber\\
=\,& 2\sum_{i,j=1}^3\int_{\mathbb{R}^3}\partial^\alpha\bigg(\partial_i a+\sum_{m=1}^3\partial_m\Gamma_{im}(\{\mathbf{I}-\mathbf{P}\}f)       \bigg)\partial^\alpha\partial_j\Gamma_{ij}(\{\mathbf{I}-\mathbf{P}\}f)\mathrm{d}x\nonumber\\
&-2\sum_{i,j=1}^3\int_{\mathbb{R}^3}\partial^\alpha\big((1+\rho)(u_i-b_i)+(1+\rho)u_i a  \big)\partial^\alpha\partial_j\Gamma_{ij}(\{\mathbf{I}-\mathbf{P}\}f)\mathrm{d}x\nonumber\\
\leq\,& \frac{1}{4}\|\nabla a\|_{H^1}^2+C\|\nabla \{\mathbf{I}-\mathbf{P}\}f \|_{L_v^2(H^1)}^2+C(1+\|\rho\|_{H^{2}})^2\|u-b\|_{H^1}^2\nonumber\\
&+C(1+\|\rho\|_{H^2})^2\|u\|_{H^1}^2\|\nabla a\|_{H^1}^2\nonumber\\
\leq\,&\Big(\frac{1}{4}+C\delta\Big)\|\nabla a\|_{H^1}^2+C\|\nabla \{\mathbf{I}-\mathbf{P}\}f \|_{L_v^2(H^1)}^2+C\|u-b\|_{H^1}^2.
\end{align}

Next, we handle the last term on the right hand-side of \eqref{G3.21}.
By making use of Young's inequality, it yields
\begin{align}\label{G3.23}
&\sum_{i,j=1}^3\int_{\mathbb{R}^3}\partial^\alpha(\partial_ib_j+\partial_jb_i)\partial^\alpha\big((1+\rho)(u_ib_j+u_jb_i)+\Gamma_{ij}(\mathfrak{l}+\mathfrak{r}+\mathfrak{s})      \big)\mathrm{d}x\nonumber\\
 &\quad \leq\frac{1}{2}\sum_{i,j=1}^3\|\partial^\alpha(\partial_ib_j+\partial_jb_i)\|^2+C\sum_{i,j=1}^3
\big\|\partial^\alpha\big((1+\rho)(u_ib_j+u_jb_i)\big)\big\|^2\nonumber\\
&\qquad +C\sum_{i,j=1}^3\big(\|\partial^\alpha\Gamma_{ij}(\mathfrak{l})\|^2
+\|\partial^\alpha\Gamma_{ij}(\mathfrak{r})\|^2+\|\partial^\alpha\Gamma_{ij}(\mathfrak{s})\|^2                  \big),
\end{align}
By  utilizing  \eqref{G3.1}, Lemmas \ref{LA.1}--\ref{LA.2} and
 the property that $\Gamma_{ij}(\cdot)$ can absorb any velocity derivative
and velocity weight, the last two terms on the above inequality can  be estimated as 
\begin{align}\label{G3.24}
&\sum_{i,j=1}^3
\big\|\partial^\alpha\big((1+\rho)(u_ib_j+u_jb_i)\big)\big\|^2+\sum_{i,j=1}^3\big(\|\partial^\alpha\Gamma_{ij}(\mathfrak{l})\|^2+\|\partial^\alpha\Gamma_{ij}(\mathfrak{r})\|^2+\|\partial^\alpha\Gamma_{ij}(\mathfrak{s})\|^2                  \big)\nonumber\\
\leq\,&C\|(1+\|\nabla\rho\|_{H^1})^2\|u\otimes b\|_{H^1}^2+C\| \{\mathbf{I}-\mathbf{P}\}f \|_{L_v^2(H^2)}^2\nonumber\\
&+C\|u\|_{H^1}^2\|\nabla  \{\mathbf{I}-\mathbf{P}\}f \|_{L_v^2(H^1)}^2+C(1+\|\rho\|_{H^2})^2\|\nabla u\|_{H^1}^2\|\nabla  \{\mathbf{I}-\mathbf{P}\}f \|_{L_v^2(H^1)}^2\nonumber\\
\leq\,&C\delta\|\nabla b\|_{H^1}^2+C\| \{\mathbf{I}-\mathbf{P}\}f \|_{L_v^2(H^2)}^2.
\end{align}

 Combining \eqref{G3.23} and \eqref{G3.24}, we  obtain
\begin{align}\label{NJK3.36}
&\sum_{i,j=1}^3\int_{\mathbb{R}^3}\partial^\alpha(\partial_ib_j+\partial_jb_i)\partial^\alpha\big((1+\rho)(u_ib_j+u_jb_i)+\Gamma_{ij}(\mathfrak{l}+\mathfrak{r}+\mathfrak{s})      \big)\mathrm{d}x\nonumber\\  
&\quad \leq\frac{1}{2}\sum_{i,j=1}^3\|\partial^\alpha(\partial_ib_j+\partial_jb_i)\|^2+C\delta\|\nabla b\|_{H^1}^2+C\| \{\mathbf{I}-\mathbf{P}\}f \|_{L_v^2(H^2)}^2.
\end{align}

Inserting \eqref{G3.22}, \eqref{NJK3.36} into \eqref{G3.21}, then leveraging
\eqref{G3.20}, we conclude that
\begin{align}\label{G3.25}
&\frac{{\rm d}}{{\rm d}t}\sum_{|\alpha|\leq 1}\sum_{i,j=1}^3\int_{\mathbb{R}^3}\partial^\alpha(\partial_ib_j+\partial_jb_i)\partial^\alpha\Gamma_{ij}(\{\mathbf{I}-\mathbf{P}\}f)\mathrm{d}x+\|\nabla b\|_{H^1}^2+\|\nabla\cdot b\|_{H^1}^2\nonumber\\
&\,\quad \leq \Big(\frac{1}{4}+C\delta\Big)\|\nabla a\|_{H^1}^2+C\|u-b\|_{H^1}^2+C\| \{\mathbf{I}-\mathbf{P}\}f \|_{L_v^2(H^2)}^2.
\end{align}
From \eqref{G3.17}$_1$--\eqref{G3.17}$_2$, one has
\begin{align}\label{G3.26}
 \| \partial^\alpha\nabla a\|^2
=\,&\sum_{i=1}^3\int_{\mathbb{R}^3}\partial^\alpha\partial_i a\partial^\alpha\partial_i a\mathrm{d}x\nonumber\\
=\,&\sum_{i=1}^3\int_{\mathbb{R}^3}\partial^\alpha\partial_i a\partial^\alpha\bigg(-\partial_{t} b_i-\sum_{j=1}^3\partial_j\Gamma_{ij}(\{\mathbf{I}-\mathbf{P}\}f) 
+(1+\rho)(u_i-b_i)+(1+\rho)u_ia        \bigg)\mathrm{d}x\nonumber\\
=\,&-\frac{{\rm d}}{{\rm d}t}\sum_{i=1}^3\int_{\mathbb{R}^3}\partial^\alpha\partial_i a\partial^\alpha b_i\mathrm{d}x+\sum_{i=1}^3\int_{\mathbb{R}^3}\partial^\alpha\partial_i\partial_t a
\partial^\alpha b_i\mathrm{d}x\nonumber\\
&+\sum_{i=1}^3\int_{\mathbb{R}^3}\partial^\alpha\partial_i a\partial^\alpha\bigg(-\sum_{j=1}^3\partial_j\Gamma_{ij}(\{\mathbf{I}-\mathbf{P}\}f) 
+(1+\rho)(u_i-b_i)+(1+\rho)u_ia        \bigg)\mathrm{d}x.
\end{align}
Notice that from \eqref{G3.17}$_1$,
\begin{align}\label{G3.27}
\sum_{i=1}^3\int_{\mathbb{R}^3}\partial^\alpha\partial_i\partial_t a
\partial^\alpha b_i\mathrm{d}x=-\int_{\mathbb{R}^3}\partial^\alpha\partial_t a
\partial^\alpha\nabla\cdot b\mathrm{d}x =\|\partial^\alpha\nabla\cdot b\|^2.  
\end{align}
Similar to the estimate \eqref{G3.22}, we reach
\begin{align}\label{G3.28}
 &\sum_{i=1}^3\int_{\mathbb{R}^3}\partial^\alpha\partial_i a\partial^\alpha\bigg(-\sum_{j=1}^3\partial_j\Gamma_{ij}(\{\mathbf{I}-\mathbf{P}\}f) 
+(1+\rho)(u_i-b_i)+(1+\rho)u_ia        \bigg)\mathrm{d}x\nonumber\\
\leq\,&\frac{1}{4}\|\nabla a\|_{H^1}^2+C\|\nabla\{\mathbf{I}-\mathbf{P}\}f\|_{L_v^2(H^1)}^2+C(1+\|\rho\|_{H^2})^2\|u-b\|_{H^1}^2\nonumber\\
&+C(1+\|\rho\|_{H^2})^2\|u\|_{H^1}^2\|\nabla a\|_{H^1}^2\nonumber\\
\leq\,&\Big(\frac{1}{4}+C\delta\Big)\|\nabla a\|_{H^1}^2+C\|u-b\|_{H^1}^2+C\|\{\mathbf{I}-\mathbf{P}\}f\|_{L_v^2(H^2)}^2.
\end{align}
Combining \eqref{G3.27}--\eqref{G3.28} with \eqref{G3.26}, and summing them up
over $|\alpha|\leq 1$, we infer that
\begin{align}\label{G3.29}
&-\frac{{\rm d}}{{\rm d}t}\sum_{|\alpha|\leq 1}\int_{\mathbb{R}^3}\partial^\alpha a\partial^\alpha\nabla\cdot b\mathrm{d}x+\Big(\frac{3}{4}-C\delta\Big)\|\nabla a\|_{H^1}^2\nonumber\\
&\quad \leq\|\nabla\cdot b\|_{H^1}^2+C\|u-b\|_{H^1}^2+C\|\{\mathbf{I}-\mathbf{P}\}f\|_{L_v^2(H^2)}^2.
\end{align}
Adding \eqref{G3.25} to \eqref{G3.29} gives
\begin{align*}
&\frac{{\rm d}}{{\rm d}t}\mathcal{E}_0(t)+\|\nabla b\|_{H^1}^2+\Big(\frac{1}{2}-C\delta \Big)\|\nabla a\|_{H^1}^2\leq C\|u-b\|_{H^1}^2+C\|\{\mathbf{I}-\mathbf{P}\}f\|_{L_v^2(H^2)}^2,
\end{align*}
which implies that \eqref{G3.19} holds.  
\end{proof}

\begin{lem}\label{L3.4}
For the strong solution $(f,\rho, u)$ to the Cauchy problem \eqref{I-3}--\eqref{I-6}, there exists a positive constant $\lambda_4$ such that
\begin{align}\label{G3.30}
\frac{\rm d}{{\rm d}t}\sum_{|\alpha|\leq 1}\int_{\mathbb{R}^3}\partial^\alpha u\cdot\partial^\alpha\nabla\rho\mathrm{d}x+\lambda_4\|\nabla\rho\|_{H^{1}}^{2}
\leq&\, C\Big(\|\nabla u\|_{H^2}^2+\|b-u\|_{H^{1}}^{2}\Big), 
\end{align} 
for any $0 \leq t \leq T$.
\end{lem}

\begin{proof}
Applying $\partial^\alpha$ with $|\alpha|\leq 1$ to \eqref{I-5},
then multiplying the result by $\nabla\partial^\alpha\rho$ and  integrating over $\mathbb{R}^3$, one gets   
\begin{align}\label{G3.31}
P^\prime(1)\|\nabla\partial^\alpha\rho\|^2=
\, &-\int_{\mathbb{R}^3}\nabla\partial^\alpha\rho\cdot\partial^\alpha\partial_t u\mathrm{d}x+\int_{\mathbb{R}^3}\nabla\partial^\alpha\rho\cdot\partial^\alpha(b-u)\mathrm{d}x   \nonumber\\
&+\int_{\mathbb{R}^3}\nabla\partial^\alpha\rho\cdot\partial^\alpha(au)\mathrm{d}x-\int_{\mathbb{R}^3}\nabla\partial^\alpha\rho\cdot\partial^\alpha(u\cdot\nabla u)\mathrm{d}x\nonumber\\
&-\int_{\mathbb{R}^3}\nabla\partial^\alpha\rho\cdot\partial^\alpha\Big(\Big(\frac{P^{\prime}(1)}{1+\rho}-P^{\prime}(1)       \Big)\nabla\rho\Big)\mathrm{d}x+\int_{\mathbb{R}^3}\nabla\partial^\alpha\rho\cdot\partial^\alpha\Big(\frac{\mu\Delta u}{1+\rho}        \Big)\mathrm{d}x\nonumber\\
\equiv:\,&\sum_{j=19}^{24}I_j.
\end{align}
First, for the term $I_{19}$, using \eqref{I-4}  and leveraging Lemmas \ref{LA.1}--\ref{LA.2}, it holds that 
\begin{align}\label{G3.32}
I_{19}=&-\frac{{\rm d}}{{\rm d}t}\int_{\mathbb{R}^3}\nabla  \partial^\alpha\rho\cdot\partial^\alpha u\mathrm{d}x+\int_{\mathbb{R}^3}\nabla  \partial^\alpha\partial_t\rho\cdot\partial^\alpha u\mathrm{d}x\nonumber\\
=&-\frac{{\rm d}}{{\rm d}t}\int_{\mathbb{R}^3}\nabla  \partial^\alpha\rho\cdot\partial^\alpha u\mathrm{d}x-\int_{\mathbb{R}^3}\nabla\partial^\alpha\big((\rho+1){\rm div} u+\nabla\rho\cdot u            \big)\cdot \partial^\alpha u\mathrm{d}x\nonumber\\
=&-\frac{{\rm d}}{{\rm d}t}\int_{\mathbb{R}^3}\nabla  \partial^\alpha\rho\cdot\partial^\alpha u\mathrm{d}x+\int_{\mathbb{R}^3}\partial^\alpha\big((\rho+1){\rm div} u+\nabla\rho\cdot u            \big) \partial^\alpha{\rm div} u\mathrm{d}x\nonumber\\
\leq&-\frac{{\rm d}}{{\rm d}t}\int_{\mathbb{R}^3}\nabla  \partial^\alpha\rho\cdot\partial^\alpha u\mathrm{d}x+C\|\partial^\alpha{\rm div}u
\|^2+C(\|\rho\|_{H^1}^2\|\nabla u\|_{H^2}^2+\|\nabla \rho\|_{H^1}^2\|\nabla u\|_{H^1}^2)\nonumber\\
\leq&-\frac{{\rm d}}{{\rm d}t}\int_{\mathbb{R}^3}\nabla  \partial^\alpha\rho\cdot\partial^\alpha u\mathrm{d}x+C\|\nabla u\|_{H^2}^2.
\end{align}
Next, with aid of \eqref{G3.1}  \eqref{rho-1},  H\"{o}lder’s and Young's inequalities, and Lemmas \ref{LA.1}--\ref{LLA.3}, we can estimate the terms $I_{20}$ to $I_{24}$ as  
\begin{align}
I_{20}\leq\,&\frac{1}{6}P^{\prime}(1)\|\nabla\partial^\alpha\rho\|^2+C\|u-b\|_{H^1}^2,\label{G3.33-1}\\
I_{21}\leq\,&\frac{1}{6}P^{\prime}(1)\|\nabla\partial^\alpha\rho\|^2+C\|a\|_{H^1}^2\|\nabla u\|_{H^1}^2\leq\frac{1}{6}P^{\prime}(1)\|\nabla\partial^\alpha\rho\|^2+C\delta\|\nabla u\|_{H^2}^2,\label{G3.33-2}\\
I_{22}\leq\,&\frac{1}{6}P^{\prime}(1)\|\nabla\partial^\alpha\rho\|^2+C\|u\|_{H^1}^2\|\nabla u\|_{H^2}^2\leq\frac{1}{6}P^{\prime}(1)\|\nabla\partial^\alpha\rho\|^2+C\delta\|\nabla u\|_{H^2}^2,\label{G3.33-3}\\
I_{23}\leq\,& \frac{1}{6}P^{\prime}(1)\|\nabla\partial^\alpha\rho\|^2+C\|\rho\|_{H^2}^2\|\nabla\rho\|_{H^1}^2\leq\frac{1}{6}P^{\prime}(1)\|\nabla\partial^\alpha\rho\|^2+C\delta\|\nabla \rho\|_{H^1}^2,\label{G3.33-4}\\
I_{24}\leq\,& \frac{1}{6}P^{\prime}(1)\|\nabla\partial^\alpha\rho\|^2+C\|\Delta
 u\|^2+C\|\nabla\Delta u\|^2+C\Big\|\nabla\Big(\frac{1}{1+\rho}\Big)\Big\|_{H^1}^2\|\Delta u\|_{H^1}^2\nonumber\\
\leq\,&\frac{1}{6}P^{\prime}(1)\|\nabla\partial^\alpha\rho\|^2+C\|\nabla u\|_{H^2}^2+C\|\rho\|_{H^2}^2\|\nabla u\|_{H^2}^2\nonumber\\
\leq\,&\frac{1}{6}P^{\prime}(1)\|\nabla\partial^\alpha\rho\|^2+C\|\nabla u\|_{H^2}^2. \label{G3.33}
\end{align}
Plugging the estimates \eqref{G3.32}--\eqref{G3.33} into \eqref{G3.31}
and {{summing them up}} over $|\alpha|\leq1$, we eventually get \eqref{G3.30}.
\end{proof}

Finally, we give the estimate of the mixed space-velocity derivatives of $f$, i.e., $\partial^\alpha_x\partial^\beta_v f$.
Due to the fact that $\|\partial^\alpha_x \partial^\beta_v \mathbf{P}f\|\leq C\|\partial^\alpha f\|$, we just need to estimate $\|\partial_x^\alpha\partial_v^\beta \{\mathbf{I}-\mathbf{P}\}f\|$.
By  applying the operator $\{\mathbf{I}-\mathbf{P}\}$
to \eqref{I-3} and  using \eqref{G2.1}, we arrive at
\begin{align}\label{G3.34}
&\partial_t\{\mathbf{I}-\mathbf{P}\}f+v\cdot\nabla\{\mathbf{I}-\mathbf{P}\}f
+(1+\rho)u\cdot\nabla_v\{\mathbf{I}-\mathbf{P}\}f-\frac{1}{2}(1+\rho)u\cdot v\{\mathbf{I}-\mathbf{P}\}f\nonumber\\
&\quad =(\rho+1)\mathcal{L}\{\mathbf{I}-\mathbf{P}\}f+\mathbf{P}\Big(v\cdot\nabla\{\mathbf{I}-\mathbf{P}\}f+(1+\rho)u\cdot\nabla_v \{\mathbf{I}-\mathbf{P}\}f      -\frac{1}{2}(1+\rho)u\cdot v\{\mathbf{I}-\mathbf{P}\}f\Big)\nonumber\\
&\qquad -\{\mathbf{I}-\mathbf{P}\}\Big(v\cdot\nabla\mathbf{P}f+(1+\rho)u\cdot\nabla_v \mathbf{P}f-\frac{1}{2}(1+\rho)u\cdot v\mathbf{P}f                   \Big).
\end{align}
Here, we have used the following facts:
\begin{align*}
\{\mathbf{I}-\mathbf{P}\}(u\cdot v\sqrt{M})=0, \quad \{\mathbf{I}-\mathbf{P}\}\big((1+\rho)u\cdot v\sqrt{M}\big)=0.
\end{align*}

\begin{lem}\label{L3.5}
Let $1\leq k\leq 2$.
For the strong solution $(f,\rho, u)$ to the Cauchy problem \eqref{I-3}--\eqref{I-6}, there exists a positive constant $\lambda_5$ such that
 \begin{align}\label{G3.35}
& \frac{\mathrm{d}}{\mathrm{d} t}\sum_{1\leq k\leq2}C_k\sum_{\substack{|\beta|=k \\
|\alpha|+|\beta| \leq 2}}\|\partial^\alpha_x\partial^\beta_v\{\mathbf{I}-\mathbf{P}\} f\|^2+\lambda_5\sum_{\substack{1\leq|\beta|\leq 2 \\
|\alpha|+|\beta| \leq 2}}\|\partial^\alpha_x\partial^\beta_v\{\mathbf{I}-\mathbf{P}\} f\|_\nu^2\nonumber \\
&\quad  \leq  C\Big(\|\nabla (a, b)\|_{H^1}^2+\sum_{|\alpha| \leq 2}\left\|\partial^\alpha\{\mathbf{I}-\mathbf{P}\} f\right\|_\nu^2\Big),
\end{align}
 for any $0 \leq t \leq T$ and any $T>0$, and $C_k$ are some positive constants.
\end{lem}

\begin{proof}

In our proof we shall apply  similar  method used in \cite[Lemma 4.3]{DFT-2010-CMP}.
We first fix $k$ with $1\leq k\leq 2$, 
then choose $\alpha,\beta$ satisfying $|\beta|=k$ and $|\alpha|+|\beta|\leq 2$. 
Multiplying \eqref{G3.34} by $\partial^\alpha_x\partial^\beta_v\{\mathbf{I}-\mathbf{P}\}f$ and taking integration for the resulting identity lead to
\begin{align}\label{G3.36}
&\frac{1}{2}\frac{{\rm d}}{{\rm d}t} \|\partial^\alpha_x\partial^\beta_v \{\mathbf{I}-\mathbf{P}\}f\|^2+\int_{\mathbb{R}^3}\langle -\mathcal{L}\partial^\alpha_x\partial^\beta_v\{\mathbf{I}-\mathbf{P}\}f,\partial^\alpha_x\partial^\beta_v\{\mathbf{I}-\mathbf{P}\}f\rangle\mathrm{d}x\nonumber\\
=\,&\int_{\mathbb{R}^3}\langle\partial_{x}^{\alpha}[\partial_v^\beta,-|v|^2]\{\mathbf{I}-\mathbf{P}\}f,
\partial_{x}^{\alpha}\partial_{v}^{\beta}\{\mathbf{I}-\mathbf{P}\}f\rangle{\rm d}x\nonumber\\
&-\int_{\mathbb{R}^3}\langle\partial_{x}^{\alpha}\partial_{v}^{\beta}(v\cdot\nabla\{\mathbf{I}-\mathbf{P}\}f),
\partial_{x}^{\alpha}\partial_{v}^{\beta}\{\mathbf{I}-\mathbf{P}\}f\rangle{\rm d}x\nonumber\\
&-\int_{\mathbb{R}^3}\big\langle\partial_{x}^{\alpha}\partial_{v}^{\beta}\big((1+\rho)u\cdot\nabla_{v}\{\mathbf{I}-\mathbf{P}\}f\big),
\partial_{x}^{\alpha}\partial_{v}^{\beta}\{\mathbf{I}-\mathbf{P}\}f\big\rangle{\rm d}x \nonumber\\
&+\frac{1}{2}\int_{\mathbb{R}^3}\big\langle\partial_{x}^{\alpha}\partial_{v}^{\beta}\big((1+\rho)u\cdot v\{\mathbf{I}-\mathbf{P}\}f\big),
\partial_{x}^{\alpha}\partial_{v}^{\beta}\{\mathbf{I}-\mathbf{P}\}f\big\rangle{\rm d}x \nonumber\\
&+\int_{\mathbb{R}^3}\big\langle\partial_{x}^{\alpha}\partial_{v}^{\beta}(\rho\mathcal{L}\{\mathbf{I}-\mathbf{P}\}f),
\partial_{x}^{\alpha}\partial_{v}^{\beta}\{\mathbf{I}-\mathbf{P}\}f\big\rangle{\rm d}x \nonumber\\
&+\int_{\mathbb{R}^3}\bigg\langle\partial_{x}^{\alpha}\partial_{v}^{\beta}\mathbf{P}
\Big(v\cdot\nabla\{\mathbf{I}-\mathbf{P}\}f
+(1+\rho)u\cdot\nabla_{v}\{\mathbf{I}-\mathbf{P}\}f
\Big),
\partial_{x}^{\alpha}\partial_{v}^{\beta}\{\mathbf{I}-\mathbf{P}\}f\bigg\rangle{\rm d}x \nonumber\\
&-\frac{1}{2}\int_{\mathbb{R}^3}\bigg\langle\partial_{x}^{\alpha}\partial_{v}^{\beta}\mathbf{P}
\Big((1+\rho)u\cdot v\{\mathbf{I}-\mathbf{P}\}f\Big),
\partial_{x}^{\alpha}\partial_{v}^{\beta}\{\mathbf{I}-\mathbf{P}\}f\bigg\rangle{\rm d}x \nonumber\\
&-\int_{\mathbb{R}^3}\bigg\langle\partial_{x}^{\alpha}\partial_{v}^{\beta}\{\mathbf{I}-\mathbf{P}\}
\Big(v\cdot\nabla \mathbf{P}f+(1+\rho)u\cdot\nabla_{v}\mathbf{P}f-\frac{1}{2}(1+\rho)u\cdot v\mathbf{P}f\Big),
\partial_{x}^{\alpha}\partial_{v}^{\beta}\{\mathbf{I}-\mathbf{P}\}f\bigg\rangle{\rm d}x\nonumber\\
\equiv:\,&\sum_{j=25}^{32}I_j,
\end{align}
where we used the fact $[\partial^\beta_v,\mathcal{L}]=[\partial^\beta_v,-|v|^2]$.

For term $I_{25}$, it follows from Young's inequality and Lemmas
\ref{LA.1}--\ref{LLA.3} that
\begin{align}\label{G3.37}
I_{25}\leq&\, \frac{1}{6}\|\partial^\alpha_x\partial^\beta_v\{\mathbf{I}-\mathbf{P}\}f\|^2+C\|[\partial^\beta_v,-|v|^2]\partial^\alpha_x\{\mathbf{I}-\mathbf{P}\}f\|^2 \nonumber\\
\leq&\, \frac{1}{6}\|\partial^\alpha_x\partial^\beta_v\{\mathbf{I}-\mathbf{P}\}f\|^2
+C\sum_{|\alpha^\prime|\leq 2-k}\|\partial^{\alpha^{\prime}}\{\mathbf{I}-\mathbf{P}\}f\|_{\nu}^2\nonumber\\
&\,+C\chi_{\{k=2\}}\sum_{\substack{1\leq|\beta^{\prime}|\leq k-1\\|\alpha^{\prime}+|\beta^{\prime}|\leq 2}}\|\partial^{\alpha^{\prime}}_x\partial^{\beta^{\prime}}_v\{\mathbf{I}-\mathbf{P}\}f\|_{\nu}^2.
\end{align}
Here and below, $\chi_D$ denotes the characteristic
function of a set $D$.

Similarly, we have
\begin{align}\label{G3.38}
I_{26}=\,&-\int_{\mathbb{R}^3}\langle\partial_{x}^{\alpha}[\partial^\beta_v,v\cdot\nabla]\{\mathbf{I}-\mathbf{P}\}f,
\partial_{x}^{\alpha}\partial_{v}^{\beta}\{\mathbf{I}-\mathbf{P}\}f\rangle{\rm d}x\nonumber\\   
\leq\,& \frac{1}{6}\|\partial^\alpha_x\partial^\beta_v\{\mathbf{I}-\mathbf{P}\}f\|^2
+C\sum_{|\alpha^\prime|\leq 2-k}\|\nabla\partial^{\alpha^{\prime}}\{\mathbf{I}-\mathbf{P}\}f\|^2\nonumber\\
&+C\chi_{\{k=2\}}\sum_{\substack{1\leq|\beta^{\prime}|\leq k-1\\|\alpha^{\prime}+|\beta^{\prime}|\leq 2}}\|\partial^{\alpha^{\prime}}_x\partial^{\beta^{\prime}}_v\{\mathbf{I}-\mathbf{P}\}f\|^2,\nonumber\\
I_{27}\leq\,& C\sum_{|\alpha^\prime|<|\alpha|}\int_{\mathbb{R}^3\times{\mathbb{R}^3}}\big|\partial^{\alpha-\alpha^{\prime}}\big((1+\rho)u\big)\big||\nabla_v\partial_x^{\alpha^{\prime}}\partial^\beta_v\{\mathbf{I}-\mathbf{P}\}f|^2|\partial_x^{\alpha}\partial^\beta_v\{\mathbf{I}-\mathbf{P}\}f|^2\mathrm{d}x\mathrm{d}v\nonumber\\
\leq\,&C(1+\|\rho\|_{H^2})\|u\|_{H^1}\sum_{\substack{1\leq|\beta^{\prime}|\leq 2\\|\alpha^{\prime}+|\beta^{\prime}|\leq 2}}\|\partial^{\alpha^{\prime}}_x\partial^{\beta^{\prime}}_v\{\mathbf{I}-\mathbf{P}\}f\|^2,\nonumber\\
\leq\,&C\delta\sum_{\substack{1\leq|\beta^{\prime}|\leq 2\\|\alpha^{\prime}+|\beta^{\prime}|\leq 2}}\|\partial^{\alpha^{\prime}}_x\partial^{\beta^{\prime}}_v\{\mathbf{I}-\mathbf{P}\}f\|^2,\nonumber\\
I_{28}\leq\,& C\sum_{|\alpha^\prime|<|\alpha|}\int_{\mathbb{R}^3\times{\mathbb{R}^3}}\big|\partial^{\alpha-\alpha^{\prime}}\big((1+\rho)u\big)\big||\partial_x^{\alpha^{\prime}}\partial^\beta_v (v\{\mathbf{I}-\mathbf{P}\}f)|^2|\partial_x^{\alpha}\partial^\beta_v\{\mathbf{I}-\mathbf{P}\}f|^2\mathrm{d}x  {\mathrm{d}v} \nonumber\\
&+C(1+\|\rho\|_{L^{\infty}})\|u\|_{L^{\infty}}\sum_{\substack{1\leq|\beta^{\prime}|\leq 2\\|\alpha^{\prime}+|\beta^{\prime}|\leq 2}}\|\partial^{\alpha^{\prime}}_x\partial^{\beta^{\prime}}_v\{\mathbf{I}-\mathbf{P}\}f\|_{\nu}^2      \nonumber\\
\leq\,&C(1+\|\rho\|_{H^2})\|u\|_{H^2}\sum_{\substack{1\leq|\beta^{\prime}|\leq 2\\|\alpha^{\prime}+|\beta^{\prime}|\leq 2}}\|\partial^{\alpha^{\prime}}_x\partial^{\beta^{\prime}}_v\{\mathbf{I}-\mathbf{P}\}f\|_{\nu}^2\nonumber\\
\leq\,&C\delta\sum_{\substack{1\leq|\beta^{\prime}|\leq 2\\|\alpha^{\prime}+|\beta^{\prime}|\leq 2}}\|\partial^{\alpha^{\prime}}_x\partial^{\beta^{\prime}}_v\{\mathbf{I}-\mathbf{P}\}f\|_{\nu}^2.
\end{align}
For term $I_{29}$, using $[\partial^\beta_v,\mathcal{L}]=[\partial^\beta_v,-|v|^2]$ again, we derive
\begin{align}\label{G3.39}
I_{29}=\,&
\int_{\mathbb{R}^3}\big\langle\partial_{x}^{\alpha}(\rho\mathcal{L}\partial_{v}^{\beta}\{\mathbf{I}-\mathbf{P}\}f),
\partial_{x}^{\alpha}\partial_{v}^{\beta}\{\mathbf{I}-\mathbf{P}\}f\big\rangle{\rm d}x \nonumber\\
&+\int_{\mathbb{R}^3}\big\langle\partial_{x}^{\alpha}(\rho[\partial^\beta_v,-|v|^2]\{\mathbf{I}-\mathbf{P}\}f),
\partial_{x}^{\alpha}\partial_{v}^{\beta}\{\mathbf{I}-\mathbf{P}\}f\big\rangle{\rm d}x \nonumber\\
=\,&\int_{\mathbb{R}^3}\langle \rho\mathcal{L}\partial^\alpha_x\partial^\beta_v\{\mathbf{I}-\mathbf{P}\}f,\partial^\alpha_x\partial^\beta_v\{\mathbf{I}-\mathbf{P}\}f\rangle\mathrm{d}x\nonumber\\
&+\sum_{|\alpha^{\prime}|<|\alpha|}\int_{\mathbb{R}^3}\langle \partial^{\alpha-\alpha^{\prime}}\rho\mathcal{L}\partial^{\alpha^{\prime}}_x\partial^\beta_v\{\mathbf{I}-\mathbf{P}\}f,\partial^\alpha_x\partial^\beta_v\{\mathbf{I}-\mathbf{P}\}f\rangle\mathrm{d}x\nonumber\\
&+\int_{\mathbb{R}^3}\big\langle\partial_{x}^{\alpha}(\rho[\partial^\beta_v,-|v|^2]\{\mathbf{I}-\mathbf{P}\}f),
\partial_{x}^{\alpha}\partial_{v}^{\beta}\{\mathbf{I}-\mathbf{P}\}f\big\rangle{\rm d}x \nonumber\\
\equiv:\,&J_1+J_2+J_3.
\end{align}
By a direct calculation, one obtains
\begin{align}\label{G3.40}
J_2=\,&-\sum_{|\alpha^\prime|<|\alpha|}\int_{\mathbb{R}^3}\bigg\langle\partial^{\alpha-\alpha^{\prime}}\rho \sqrt{M}\nabla_v\Big(\frac{\partial^{\alpha^{\prime}}_x\partial_v^\beta \{\mathbf{I}-\mathbf{P}\}f}{\sqrt{M}}\Big), \nabla_v\Big(\frac{\partial^{\alpha}_x\partial_v^\beta \{\mathbf{I}-\mathbf{P}\}f}{\sqrt{M}}\Big)\bigg\rangle  \mathrm{d}x\nonumber\\
\leq\,&C(\|\rho\|_{L^3}+\|\nabla\rho\|_{L^3})\|\partial^{\alpha^{\prime}+1}_x\partial_v^\beta \{\mathbf{I}-\mathbf{P}\}f\|_{\nu}\|\partial^{\alpha}_x\partial_v^\beta \{\mathbf{I}-\mathbf{P}\}f\|_{\nu}\nonumber\\
\leq\,&C\|\rho\|_{H^2}\sum_{\substack{1\leq|\beta^{\prime}|\leq 2\\|\alpha^{\prime}+|\beta^{\prime}|\leq 2}}\|\partial^{\alpha^{\prime}}_x\partial^{\beta^{\prime}}_v\{\mathbf{I}-\mathbf{P}\}f\|_{\nu}^2\nonumber\\
\leq\,&C\delta\sum_{\substack{1\leq|\beta^{\prime}|\leq 2\\|\alpha^{\prime}+|\beta^{\prime}|\leq 2}}\|\partial^{\alpha^{\prime}}_x\partial^{\beta^{\prime}}_v\{\mathbf{I}-\mathbf{P}\}f\|_{\nu}^2,\nonumber\\
J_3\leq\,& \frac{1}{6}\|\partial^\alpha_x\partial^\beta_v\{\mathbf{I}-\mathbf{P}\}f\|^2+C\|[\partial^\beta_v,-|v|^2]\partial^\alpha_x(\rho\{\mathbf{I}-\mathbf{P}\}f)\|^2 \nonumber\\
\leq\,& \frac{1}{6}\|\partial^\alpha_x\partial^\beta_v\{\mathbf{I}-\mathbf{P}\}f\|^2+C\|\rho\|_{H^2}^2\sum_{|\alpha^{\prime}|\leq 2}\|\partial^{\alpha^{\prime}}\{\mathbf{I}-\mathbf{P}\}f\|^2\nonumber\\
&+C\|\rho\|_{H^2}^2\chi_{\{k=2\}}\sum_{\substack{1\leq|\beta^{\prime}|\leq k-1\\|\alpha^{\prime}+|\beta^{\prime}|\leq 2}}\|\partial^{\alpha^{\prime}}_x\partial^{\beta^{\prime}}_v\{\mathbf{I}-\mathbf{P}\}f\|_{\nu}^2\nonumber\\
\leq\,& \frac{1}{6}\|\partial^\alpha_x\partial^\beta_v\{\mathbf{I}-\mathbf{P}\}f\|^2+C\delta\sum_{|\alpha^{\prime}|\leq 2}\|\partial^{\alpha^{\prime}}\{\mathbf{I}-\mathbf{P}\}f\|^2\nonumber\\
&+C\delta\chi_{\{k=2\}}\sum_{\substack{1\leq|\beta^{\prime}|\leq k-1\\|\alpha^{\prime}+|\beta^{\prime}|\leq 2}}\|\partial^{\alpha^{\prime}}_x\partial^{\beta^{\prime}}_v\{\mathbf{I}-\mathbf{P}\}f\|_{\nu}^2.
\end{align}
Putting \eqref{G3.40} into \eqref{G3.39} ensures
\begin{align}\label{G3.41}
I_{29}\leq\,& J_1+ \frac{1}{6}\|\partial^\alpha_x\partial^\beta_v\{\mathbf{I}-\mathbf{P}\}f\|^2+C\delta\sum_{|\alpha^{\prime}|\leq 2}\|\partial^{\alpha^{\prime}}\{\mathbf{I}-\mathbf{P}\}f\|^2 \nonumber\\
&+C\delta\chi_{\{k=2\}}\sum_{\substack{1\leq|\beta^{\prime}|\leq k-1\\|\alpha^{\prime}+|\beta^{\prime}|\leq 2}}\|\partial^{\alpha^{\prime}}_x\partial^{\beta^{\prime}}_v\{\mathbf{I}-\mathbf{P}\}f\|_{\nu}^2+C\delta\sum_{\substack{1\leq|\beta^{\prime}|\leq 2\\|\alpha^{\prime}+|\beta^{\prime}|\leq 2}}\|\partial^{\alpha^{\prime}}_x\partial^{\beta^{\prime}}_v\{\mathbf{I}-\mathbf{P}\}f\|_{\nu}^2.
\end{align}
For $I_{30}$ to $I_{32}$, from \eqref{G3.1}, Lemmas \ref{LA.1}--\ref{LA.2}
and Young's inequality, it yields 
\begin{align}\label{G3.42}
I_{30}+I_{31}\leq\,& \frac{1}{6}\|\partial^\alpha_x\partial^\beta_v\{\mathbf{I}-\mathbf{P}\}f\|^2+C\|\partial^\alpha_x\partial^\beta_v\mathbf{P}(v\cdot\nabla\{\mathbf{I}-\mathbf{P}\}f)\|^2\nonumber\\
&+C\|\partial^\alpha_x\partial^\beta_v\mathbf{P}\big((1+\rho)u\cdot\nabla_v\{\mathbf{I}-\mathbf{P}\}f\big)\|^2+C\|\partial^\alpha_x\partial^\beta_v\mathbf{P}\big((1+\rho)u\cdot v\{\mathbf{I}-\mathbf{P}\}f\big)\|^2 \nonumber\\
\leq\,& \frac{1}{6}\|\partial^\alpha_x\partial^\beta_v\{\mathbf{I}-\mathbf{P}\}f\|^2+C\sum_{|\alpha^\prime|\leq 2-k}\|\nabla\partial^{\alpha^{\prime}}\{\mathbf{I}-\mathbf{P}\}f\|^2\nonumber\\
&+C(1+\|\rho\|_{H^2})^2\|u\|_{H^2}^2\sum_{|\alpha^\prime|\leq 2-k}\|\partial^{\alpha^\prime}\{\mathbf{I}-\mathbf{P}\}f\|^2 \nonumber\\
\leq\,&\frac{1}{6}\|\partial^\alpha_x\partial^\beta_v\{\mathbf{I}-\mathbf{P}\}f\|^2+C\sum_{|\alpha^{\prime}|\leq 2}\|\partial^{\alpha^{\prime}}\{\mathbf{I}-\mathbf{P}\}f\|^2,\nonumber\\
I_{32}\leq\,&\frac{1}{6}\|\partial^\alpha_x\partial^\beta_v\{\mathbf{I}-\mathbf{P}\}f\|^2+C\|\partial_{x}^{\alpha}\partial_{v}^{\beta}\{\mathbf{I}-\mathbf{P}\} (v\cdot\nabla \mathbf{P}f)\|^2\nonumber\\
&+C\|\partial_{x}^{\alpha}\partial_{v}^{\beta}\{\mathbf{I}-\mathbf{P}\} \big((1+\rho)u\cdot\nabla \mathbf{P}f\big)\|^2+C\|\partial_{x}^{\alpha}\partial_{v}^{\beta}\{\mathbf{I}-\mathbf{P}\} \big((1+\rho)u\cdot v \mathbf{P}f\big)\|^2\nonumber\\
\leq\,&\frac{1}{6}\|\partial^\alpha_x\partial^\beta_v\{\mathbf{I}-\mathbf{P}\}f\|^2
+C\|\nabla (a,b)\|_{H^1}^2+(1+\|\rho\|_{H^2})^2\|u\|_{H^2}^2\|\nabla(a,b)\|_{H^1}^2\nonumber\\
\leq\,& \frac{1}{6}\|\partial^\alpha_x\partial^\beta_v\{\mathbf{I}-\mathbf{P}\}f\|^2
+C\|\nabla (a,b)\|_{H^1}^2.
\end{align}
With {the} help of \eqref{G2.3}, we get
\begin{align}\label{G3.43}
&\int_{\mathbb{R}^3}\langle -\mathcal{L}\partial^\alpha_x\partial^\beta_v\{\mathbf{I}-\mathbf{P}\}f,\partial^\alpha_x\partial^\beta_v\{\mathbf{I}-\mathbf{P}\}f\rangle\mathrm{d}x-J_1 \nonumber\\   
=\,&\int_{\mathbb{R}^3}(1+\rho)\langle -\mathcal{L}\partial^\alpha_x\partial^\beta_v\{\mathbf{I}-\mathbf{P}\}f,\partial^\alpha_x\partial^\beta_v\{\mathbf{I}-\mathbf{P}\}f\rangle\mathrm{d}x\nonumber\\
\geq\,&\lambda_0\|\{\mathbf{I}-\mathbf{P}_0\}\partial^\alpha_x\partial^\beta_v\{\mathbf{I}-\mathbf{P}\}f\|_{\nu}^2\nonumber\\
\geq\,&\frac{\lambda_0}{2}\|\partial^\alpha_x\partial^\beta_v\{\mathbf{I}-\mathbf{P}\}f\|_{\nu}^2-\lambda_0\|\mathbf{P}_0\partial^\alpha_x\partial^\beta_v\{\mathbf{I}-\mathbf{P}\}f\|_{\nu}^2\nonumber\\
\geq\,&\frac{\lambda_0}{2}\|\partial^\alpha_x\partial^\beta_v\{\mathbf{I}-\mathbf{P}\}f\|_{\nu}^2-C\|\partial^\alpha_x\{\mathbf{I}-\mathbf{P}\}f\|^2.
\end{align}
Putting all the above estimates \eqref{G3.36}--\eqref{G3.43} into
\eqref{G3.35} and taking summation over $|\beta|=k$ with $1\leq k\leq 2$, one has
\begin{align}\label{G3.44}
&\frac{1}{2}\frac{\rm d}{{\rm d}t}
\sum_{\substack{ |\beta|=k \\  |\alpha|+|\beta|\leq 2}}
\|\partial^{\alpha}_{x}\partial^{\beta}_{v}\{\mathbf{I}-\mathbf{P}\}f\|^{2}+\lambda_6
\sum\limits_{\substack{|\beta|=k \\  
 |\alpha|+|\beta|\leq 2}}\|\partial^{\alpha}_{x}\partial^{\beta}_{v}\{\mathbf{I}-\mathbf{P}\}f\|_{\nu}^{2}\nonumber\\
\leq\,& C\chi_{\{k=2\}}\sum_{\substack{1\leq|\beta^{\prime}|\leq k-1 \nonumber\\ 
 |\alpha^{\prime}|+|\beta^{\prime
 }|\leq 2}}\|\partial^{\alpha^{\prime}}_{x}\partial^{\beta^{\prime}}_{v}
\{\mathbf{I}-\mathbf{P}\}f\|_{\nu}^{2}
+C\|\nabla(a,b)\|_{H^1}^2\nonumber\\
&+C\sum_{|\alpha| \leq 2}\left\|\partial^\alpha\{\mathbf{I}-\mathbf{P}\} f\right\|_\nu^2+C\delta\sum_{\substack{1\leq|\beta^{\prime}|\leq 2\\|\alpha^{\prime}+|\beta^{\prime}|\leq 2}}\|\partial^{\alpha^{\prime}}_x\partial^{\beta^{\prime}}_v\{\mathbf{I}-\mathbf{P}\}f\|_{\nu}^2,    
\end{align}
for some $\lambda_6>0$.
Then, by choosing some suitable constants $C_k$ in \eqref{G3.44}, we
consequently get \eqref{G3.35}.
\end{proof}

\subsection{Proof of Theorem \ref{T1.1}}
 With Lemmas \ref{L3.1}--\ref{L3.5} in hand, we continue to establish the uniform-in-time a priori estimates of solution $(f,\rho,u)$ to \eqref{I-3}--\eqref{I-6}.

Now we define a temporal energy functional $\mathcal{E} (t)$ and 
the corresponding dissipation rate $\mathcal{D}(t)$ as follows:
\begin{align}\label{G3.45}
\mathcal{E}(t):=\,&\|(f,\rho,u)\|_{\mathcal{H}^2}
+\tau_{1}\mathcal{E}_0(t)+\tau_2\sum_{|\alpha|\leq 1}\int_{\mathbb{R}^3}\partial^\alpha u\cdot\partial^\alpha\nabla\rho\mathrm{d}x
\nonumber\\
&+\tau_{3}\sum_{1\leq k\leq2}C_k\sum_{\substack{|\beta|=k \\
|\alpha|+|\beta| \leq 2}}\|\partial^\alpha_x\partial^\beta_v\{\mathbf{I}-\mathbf{P}\} f\|^2,\nonumber\\
\mathcal{D}(t):=\,&\|b-u\|_{H^2}^2+\|\nabla(a,b,\rho)\|_{H^1}^2+\|\nabla u\|_{H^{2}}^{2}\nonumber\\
&+\sum_{|\alpha|\leq2}\|\partial^\alpha\{\mathbf{I}-\mathbf{P}\}f\|_{\nu}^2
+\sum_{|\alpha|+|\beta|\leq2}\|\partial_{x}^{\alpha}\partial_{v}^{\beta}\{\mathbf{I}-\mathbf{P}\}f\|_{\nu}^{2},
\end{align}
where $0<\tau_1, \tau_2,\tau_3\ll 1$ with $\tau_3\ll\tau_1$ are sufficiently small constants.
Collecting   \eqref{G3.2}, \eqref{G3.10}, $\tau_1\times$\eqref{G3.19},
$\tau_2\times$\eqref{G3.30} and $\tau_3\times$\eqref{G3.35} together gives rise to
\begin{align}\label{G3.46}
\frac{{\rm d}}{{\rm d}t}\mathcal{E}(t)+\lambda_7\mathcal{D}(t)\leq 0,    
\end{align}
for some constants $\lambda_7$ satisfying
$0<\lambda_7\leq\min\{\lambda_1,\lambda_2,\tau_1\lambda_3,\tau_2\lambda_4,\tau_3\lambda_5\}$.
Thus, we have  
\begin{align}\label{G3.47}
\mathcal{E}(t)+\lambda_7\int_0^t\mathcal{D}(s)ds\leq\mathcal{E}(0),    
\end{align}
for any $0\leq t\leq T$.
In addition, \eqref{G3.1} can be justified by selecting 
\begin{align*}
\mathcal{E}(0)\backsim \|f_0\|_{H^2_{x,v}}^2+\|(\rho_0,u_0)\|_{H^2}^2     
\end{align*}
sufficiently small.

The existence and uniqueness of the local solution for the problem \eqref{I-3}--\eqref{I-6} can be established by applying  the  linearization method and the Banach contraction mapping principle (cf. \cite{Gy-CPAM-2004,CDM-krm-2011,Li-Ni-Wu}). 
For the sake of brevity, we omit the  details here.
By combining the continuity argument with the local existence of solutions 
and the uniform a priori estimates \eqref{G3.47}, we can deduce the global existence of strong solutions $(f,\rho,u)$ to \eqref{I-3}--\eqref{I-6}. Additionally, invoking the maximum principle,
we conclude that $F=M+\sqrt{M}f \geq 0$.
This completes the proof of Theorem \ref{T1.1}. \hfill $\square$
 
\smallskip
\section{Time-decay rates of   strong solutions}
In this section, we establish the time decay rates of $(f,\rho,u)$ to the problem
\eqref{I-3}--\eqref{I-6} in whole space $\mathbb{R}^3$. For this purpose,
we first study the linearized equations of \eqref{I-3}--\eqref{I-5} which can be expressed as
\begin{equation}\label{G4.1}
\left\{\begin{aligned}
&\partial_{t}f+v\cdot\nabla_x f-u\cdot v\sqrt{M}=\mathcal{L}f+S_f,  \\
&\partial_{t} \rho+{\rm div}u=0,  \\
&\partial_t u+P^{\prime}(1)\nabla\rho-\mu\Delta u+u-b=0,  
\end{aligned}\right.
\end{equation}
with the initial data:
\begin{equation}\label{G4.2}
(f,\rho,u)|_{t=0}=(f_0,\rho_0,u_0)=(f(0,x,v),\rho(0,x),u(0,x)), \quad (x, v)\in\mathbb{R}^{3}\times \mathbb{R}^3.
\end{equation}
Here, $S_f$ is the nonhomogeneous source given by
\begin{align*}
S_f={\rm div}_vG-\frac{1}{2}v\cdot G+h,
\end{align*}
with $G=G(t,x,v)\in \mathbb{R}^3$ and $h=h(t,x,v)\in \mathbb{R}$ satisfying
\begin{align*}
\mathbf{P}_0 G=0,\quad  \mathbf{P}h=0. 
\end{align*}

For presentation simplicity,  we use $U(t) = (f(t), \rho(t), u(t))$ to denote   the solution of the problem \eqref{G4.1}--\eqref{G4.2}, and $U_0$ for the initial data, i.e.  $U_0=(f_0,\rho_0,u_0)$. 
By Duhamel's principle, we arrive at
\begin{align}\label{G4.4}
U(t)=\mathbb{A}(t)U_0+\int_0^t\mathbb{A}(t-s)(S_f(s),0,0)\mathrm{d}s,
\end{align}
where $\mathbb{A}(t)$ is the solution operator corresponding to
$U(t)$ while $S_f=0$.

Next, we state the main theorem on the  time decay rate of the problem
\eqref{G4.1}--\eqref{G4.2}.
\begin{thm}\label{T4.1}
Let $1\leq q\leq2$. For all $\alpha$, $\alpha^{\prime}$ with 
$\alpha^{\prime}\leq\alpha$ and $m=|\alpha-\alpha^{\prime}|$, we have
 \begin{align}
\|\partial^{\alpha}\mathbb{A}(t)U_{0}\|_{\mathcal{Z}_{2}}
\leq  & C(1+t)^{-\frac{3}{2} (\frac{1}{q}-\frac{1}{2} )-\frac{m}{2}}
\Big(\|\partial^{\alpha^{\prime}}U_{0}\|_{\mathcal{Z}_{q}}+\|\partial^{\alpha}U_{0}\|_{\mathcal{Z}_{2}}\big),
\label{G4.5} \\
\bigg\|\partial^{\alpha}\int_{0}^{t}\mathbb{A}(t-s)(S_f(s),0,0){\rm d}s\bigg\|_{\mathcal{Z}_{2}}^{2}
\leq  &  C\int_0^t(1+t-s)^{-3 (\frac{1}{q}-\frac{1}{2})-m}\nonumber\\
 &   \times \Big(\Big\|\partial^{\alpha^{\prime}}\big(G(s),\nu^{-\frac{1}{2}}h(s)\big)\Big\|_{\mathcal{Z}_{q}}^{2}
+\Big\|\partial^{\alpha}\big(G(s),\nu^{-\frac{1}{2}}h(s)\big)\Big\|_{\mathcal{Z}_{2}}^{2}\Big){\rm d}s,\label{G4.6}
\end{align}
for any $t\geq 0$.
\end{thm}
\begin{proof}
Applying Fourier transform to \eqref{G4.1} in $x$, one has
\begin{equation}\label{G4.7}
\left\{\begin{aligned}
&\partial_{t}\widehat{f}+iv\cdot k \widehat{f}-\widehat{u}\cdot v\sqrt{M}=\mathcal{L}\widehat{f}+\nabla_v\cdot\widehat{G}-\frac{1}{2}v\cdot\widehat{G}+\widehat{h},  \\
&\partial_{t}\widehat{\rho}+ik\cdot \widehat{u}=0,  \\
&\partial_t \widehat{u}+ikP^{\prime}(1)\widehat{\rho}+\mu|k|^2 \widehat{u}+\widehat{u}-\widehat{b}=0.  
\end{aligned}\right.
\end{equation}
Multiplying \eqref{G4.7}$_2$--\eqref{G4.7}$_3$ by $\overline{\widehat{\rho}}$, $\overline{\widehat{u}}$ respectively and taking integration, it yields
\begin{align}\label{G4.8}
\frac{1}{2}\partial_t(P^{\prime}(1)|\widehat{\rho}|^2+|\widehat{u}|^2 )+\mu|k|^2|\widehat{u}|^2       
+|\widehat{u}|^2-\mathbf{Re}(\widehat{u}|\widehat{b})=0.
\end{align}
Similarly, multiplying \eqref{G4.7}$_1$ by $\overline{\widehat{f}}$ and
integrating over $\mathbb{R}^3$, it holds
\begin{align*}
&\frac{1}{2}\partial_t \|\widehat{f}\|_{L_v^2}^2+\mathbf{Re}\int_{\mathbb{R}^3} 
(-\mathcal{L}\{\mathbf{I}-\mathbf{P}\}\widehat{f}|\{\mathbf{I}-\mathbf{P}\}\widehat{f} ) {\mathrm{d}v}+|\widehat{b}|^2-\mathbf{Re}(\widehat{u}|\widehat{b})\nonumber\\
 &\quad =\mathbf{Re}\int_{\mathbb{R}^3}\Big(\nabla_v\cdot\widehat{G}-\frac{1}{2}v\cdot\widehat{G}|
\{\mathbf{I}-\mathbf{P}\}\widehat{f}\Big)\mathrm{d}v
+\mathbf{Re}\int_{\mathbb{R}^3}(\widehat{h}|\{\mathbf{I}-\mathbf{P}\}\widehat{f})\mathrm{d}v.
\end{align*}
Here, we have used the following facts:
\begin{align*}
\mathbf{P}h=0,\quad \nabla_v\cdot G-\frac{1}{2}v\cdot G\perp {\rm Rang}\mathbf{P}.
\end{align*}
By making use of \eqref{G2.3}, $\mathbf{P}_0G=0$, integration by parts, Cauchy–Schwarz inequality and
Young's inequality, we further obtain that
\begin{align*}
&\frac{1}{2}\partial_t \|\widehat{f}\|_{L_v^2}^2+\lambda_0 |\{\mathbf{I}-\mathbf{P}\}\widehat{f}|_{\nu}^2+|\widehat{b}|^2-\mathbf{Re}(\widehat{u}|\widehat{b}) \leq C(\|\widehat{G}\|_{L^2_v}^2+\|\nu^{-\frac{1}{2}}\widehat{h}\|_{L^2_v}^2).
\end{align*}
This, together with \eqref{G4.8}, gives rise to
\begin{align}\label{G4.11}
&\frac{1}{2}\partial_t(\|\widehat{f}\|_{L_v^2}^2+P^{\prime}(1)|\widehat{\rho}|^2+|\widehat{u}|^2 )+ \lambda_0 |\{\mathbf{I}-\mathbf{P}\}\widehat{f}|_{\nu}^2+ \mu|k|^2|\widehat{u}|^2 +|\widehat{u}-\widehat{b}|^2 
\leq C (\|\widehat{G}\|_{L^2_v}^2+\|\nu^{-\frac{1}{2}}\widehat{h}\|_{L^2_v}^2).
\end{align}
Using the same method developed in \cite[Theorem 3.1]{CDM-krm-2011}, we easily
get the following estimate of $\widehat{a}$, $\widehat{b}$:
\begin{align}\label{G4.12}
\partial_t \mathbf{Re}\mathcal{E}_1(\widehat{f})+\frac{|\lambda_8|^2|k|^2}{1+|k|^2} (|\widehat{a
}|^2+|\widehat{b}|^2)\leq C(\|\{\mathbf{I}-\mathbf{P}\}\widehat{f}\|_{L_v^2}^2+|\widehat{u}-\widehat{b}|^2+\|G\|_{L_v^2}^2+\|\nu^\frac{1}{2}\widehat{h}\|_{L_v^2}^2),   
\end{align}
for some $\lambda_8>0$, where 
\begin{align*}
\mathcal{E}_1(\widehat{f}):=\frac{1}{1+|k|^2}\sum_{i,j=1}^3\big(ik_i\widehat{b}_j+   ik_j\widehat{b}_i|\Gamma_{ij}(\{\mathbf{I}-\mathbf{P}\}\widehat{f})     \big)-\frac{1}{1+|k|^2}(\widehat{a}|ik\cdot\widehat{b}).   
\end{align*}
In what follows, we exhibit the remaining estimate of $\widehat{\rho}$. To
do this, by multiplying \eqref{G4.7}$_3$ by $\overline{ik\widehat{\rho}}$
and using \eqref{G4.7}$_2$,
one has
\begin{align*}
P^{\prime}(1)|k|^2|\widehat{\rho}|^2 &=(-\partial_t\widehat{u}-|k|^2\widehat{u}-\widehat{u}+\widehat{b}|ik\widehat{\rho})\nonumber\\
 &=-\partial_t(\widehat{u}|ik\widehat{\rho})+(\widehat{u}|ik\partial_t\widehat{\rho})-\mu(|k|^2\widehat{u}|ik\widehat{\rho})-
(\widehat{u}-\widehat{b}|ik\widehat{\rho})\nonumber\\
 &=-\partial_t(\widehat{u}|ik\widehat{\rho})+|k\cdot \widehat{u}|^2-\mu(|k|^2\widehat{u}|ik\widehat{\rho})-
(\widehat{u}-\widehat{b}|ik\widehat{\rho}),
\end{align*}
which, together with Young's inequality, leads to
\begin{align}\label{G4.14}
\partial_t\frac{\mathbf{Re}(\widehat{u}|ik\widehat{\rho})}{1+|k|^2}+\lambda_9P^{\prime}(1)|k|^2|\widehat{\rho}|^2\leq C (|\widehat{u}|^2+|\widehat{u}-\widehat{b}|^2)+C\frac{|k|^4|\widehat{u}|^2}{1+|k|^2},
\end{align}
for some $\lambda_9>0$.

 {Let us} define a functional $\mathcal{E}_{\mathcal{F}}(\widehat{f},\widehat{\rho},\widehat{u})$ by
\begin{align*}
\mathcal{E}_{\mathcal{F}}(\widehat{f},\widehat{\rho},\widehat{u}):=\|f\|_{L^2_v}^2+P^{\prime}(1)|\widehat{n}|^2+|\widehat{u}|^2+\tau_4\mathbf{Re}\mathcal{E}_1(\widehat{f})+\tau_5\frac{\mathbf{Re}(\widehat{u}|ik\widehat{\rho})}{1+|k|^2},
\end{align*}
where $0<\tau_4,\tau_5\ll1$ are small constants. It is obvious that
\begin{align*}
\mathcal{E}_{\mathcal{F}}(\widehat{f},\widehat{\rho},\widehat{u})\backsim \|\widehat{f}\|_{L_v^2}^2+\|\widehat{\rho}\|^2+\|\widehat{u}\|^2,    
\end{align*}
where we have used the facts that 
$$
|\mathcal{E}_1(\widehat{f})|\leq C(\|\widehat{f}\|_{L_v}^2+\|\widehat{u}\|^2),
\quad \Big|\frac{(\widehat{u}|ik\widehat{\rho})}{1+|k|^2}\Big|\leq C(\|\rho\|^2+\|\widehat{u}\|^2).
$$
Adding up \eqref{G4.11}, $\tau_4\times$\eqref{G4.12} and $\tau_5\times$\eqref{G4.14}, we infer that
\begin{align*}
\partial_t\mathcal{E}_{\mathcal{F}}(\widehat{f},\widehat{\rho},\widehat{u})+\frac{\lambda_{10}|k|^2}{1+|k|^2}\mathcal{E}_{\mathcal{F}}(\widehat{f},\widehat{\rho},\widehat{u})\leq  C(\|\widehat{G}\|_{L^2_v}^2+\|\nu^{-\frac{1}{2}}\widehat{h}\|_{L^2_v}^2),  
\end{align*}
for some $\lambda_{10}>0$,
where we have  used the basic inequality $|\widehat{u}|^2\leq 2|\widehat{u}-\widehat{b}|^2+2|\widehat{b}|^2$.
According to Gronwall's inequality, it comes out
\begin{align}\label{G4.18}
 \mathcal{E}_{\mathcal{F}}(\widehat{f},\widehat{\rho},\widehat{u})\leq e^{-\frac{\lambda_{10}|k|^2}{1+|k|^2}}  \mathcal{E}_{\mathcal{F}}(\widehat{f}_0,\widehat{\rho}_0,\widehat{u}_0)+\int_{0}^te^{-\frac{\lambda_{10}|k|^2}{1+|k|^2}(t-s)}(\|\widehat{G}(s)\|_{L^2_v}^2+\|\nu^{-\frac{1}{2}}\widehat{h}(s)\|_{L^2_v}^2)\mathrm{d}s.
\end{align}
As in 
\cite{Ks-1983}, \eqref{G4.5}--\eqref{G4.6} can be obtained by using \eqref{G4.18} and then
setting non-homogeneous source $S 
_f=0$ and the initial
data $(f_0,\rho_0 ,u_0) = 0$, respectively. 
For simplicity, the detailed proof is omitted here.
\end{proof}

Next, we turn to the proof of the rate of convergence of $(f,\rho,u)$ to the problem \eqref{I-3}--\eqref{I-6}.
Thanks to the definition of $\mathcal{E}_0(t)$ and Young's inequality, we have
\begin{align*}
\mathcal{E}_0(t)\leq C(\|a\|_{H^1}^2+\|\nabla b\|_{H^1}^2+\|\{\mathbf{I}-\mathbf{P}   \}f\|_{H^1}^2).   
\end{align*}
Similarly, 
\begin{align*}
\sum_{|\alpha|\leq 1}\int_{\mathbb{R}^3}\partial^\alpha u\cdot\partial^\alpha\nabla\rho\mathrm{d}x\leq C(\|u\|_{H^1}^2+\|\nabla\rho\|_{H^1}^2).   
\end{align*}
It follows from the definition of $\mathcal{E}(t)$ and $\mathcal{D}(t)$
in \eqref{G3.45} that
\begin{align*}
 \mathcal{E}(t)\leq&\, C(\|\{\mathbf{I}-\mathbf{P}   \}f\|_{H^2_{x,v}}^2+\|a\|_{H^2}^2+\|b\|_{H^2}^2+\|u\|_{H^2}^2+\|\rho\|_{H^2}^2) \nonumber\\
 \leq\,&C(\mathcal{D}(t)+\|(f,\rho,u)\|_{\mathcal{Z}_2}^2),
\end{align*}
which, together with \eqref{G3.46}, ensures that 
\begin{align*}
\frac{{\rm d}}{{\rm d}t}\mathcal{E}(t)+\lambda_{11}\mathcal{E}(t)\leq C\|U\|_{\mathcal{Z}_2}^2,    
\end{align*}
for some $\lambda_{11}>0$. Here, for notation simplicity,  we still use $U(t) = (f(t), \rho(t), u(t))$ to denote   the solution of the problem \eqref{I-3}--\eqref{I-6}, and $U_0$ for the initial data, i.e.  $U_0=(f_0,\rho_0,u_0)$. 

Using Gronwall's inequality, we arrive at
\begin{align}\label{G4.22}
\mathcal{E}(t)\leq e^{-\lambda_{11}t}\mathcal{E}(0)+C\int_{0}^te^{-\lambda_{11}(t-s)} \|U\|_{\mathcal{Z}_2}^2\mathrm{d}s.  
\end{align}
To estimate  $\|U\|_{\mathcal{Z}_2}^2$, we rewrite \eqref{I-3}--\eqref{I-5} as
\begin{align}\label{G4.23}
U(t)=\mathbb{A}(t)U_0+\int_0^t\mathbb{A}(t-s)({S_f,S_\rho,S_u})(s))\mathrm{d}s,
\end{align}
with
\begin{align}\label{G4.24}
{ S_f}:=\,&(1+\rho)\Big(-u\cdot\nabla_vf+\frac{1}{2}u\cdot vf\Big)+\rho(\mathcal{L}f+u
\cdot v\sqrt{M}) \nonumber\\
=\,&-(1+\rho)u\cdot\nabla_v\{\mathbf{I}-\mathbf{P}_0\}f+\frac{1}{2}(1+\rho)u\cdot v\{\mathbf{I}-\mathbf{P}_0\}f-(1+\rho)u\cdot\nabla_v \mathbf{P}_0f\nonumber\\
&+\frac{1}{2}(1+\rho)u\cdot v \mathbf{P}_0 f+\rho\mathcal{L}\{\mathbf{I}-\mathbf{P}\}f+\rho\mathcal{L}\mathbf{P}f
+\rho u\cdot v\sqrt{M}\nonumber\\
=\,&\Big(\nabla_v\cdot G-\frac{1}{2}v\cdot G+(1+\rho)u\cdot a v\sqrt{M}\Big)+h+\rho(u-b)\cdot v\sqrt{M},
\end{align}
where $G:=-(1+\rho)u\{\mathbf{I}-\mathbf{P}_0\}f$ and $h:=\rho\mathcal{L}\{\mathbf{I}-\mathbf{P}\}f$, and
\begin{align}\label{G4.25}
S_\rho:=\,&-\rho{\rm div}u-\nabla\rho\cdot u, \\\label{G4.26} 
S_u:=\,&-u\cdot\nabla u-\Big(\frac{P^{\prime}(1+\rho)}{1+\rho}-P^{\prime}(1)     \Big)\nabla\rho-\frac{\mu\rho}{1+\rho}\Delta u-au.
\end{align}
By inserting \eqref{G4.24}--\eqref{G4.26} into \eqref{G4.23}, we further obtain
\begin{align}\label{G4.27}
U(t)=\,\,&\mathbb{A}(t)U_0+\int_0^t\mathbb{A}(t-s)(\nabla_v\cdot G-\frac{1}{2}v\cdot G,0,0)\mathrm{d}s\nonumber\\
&+\int_0^t\mathbb{A}(t-s)\big((1+\rho)u\cdot av\sqrt{M},0,0\big)\mathrm{d}s+\int_0^t\mathbb{A}(t-s)(\rho (u-b)\cdot v\sqrt{M},0,0\big)\mathrm{d}s\nonumber\\
&+\int_0^t\mathbb{A}(t-s)(h,0,0)\mathrm{d}s+\int_0^t\mathbb{A}(t-s)(0,S_\rho(s),0)\mathrm{d}s+\int_0^t\mathbb{A}(t-s)(0,0,S_u(s))\mathrm{d}s\nonumber\\
=:\,\,&\sum_{j=1}^7K_j(t).
\end{align}
Define 
\begin{align*}
\mathcal{E}_{0,\infty}(t):=\sup_{0\leq s\leq t}(1+s)^\frac{3}{2}\mathcal{E}(s).    
\end{align*}
Taking $\alpha^{\prime}=0$, $\alpha=0$ and $q=1$ in \eqref{G4.5} and
using Theorem \ref{T4.1},
we have
\begin{align}\label{G4.28}
\|K_1(t)\|_{\mathcal{Z}_2}\leq  C(1+t)^{-\frac{3}{4}}(\|U_0\|_{\mathcal{Z}_1}+\|U_0\|_{\mathcal{Z}_2}).   
\end{align}
Similar to \eqref{G4.28}, by using Lemmas \ref{LA.1}--\ref{LA.2} and
Lemma \ref{LA.3}, the terms $K_3(t)$, $K_4(t)$, $K_5(t)$, and $K_6(t)$ can be estimated as \begin{align}\label{G4.29}
\|K_3(t)\|_{\mathcal{Z}_2}+\|K_4(t)\|_{\mathcal{Z}_2}\leq\, & C\int_0^t(1+t-s)^{-\frac{3}{4}}\|(1+\rho)u\cdot a v\sqrt{M}\|_{\mathcal{Z}_1\cap\mathcal{Z}_2}\mathrm{d}s\nonumber\\
&+C\int_0^t(1+t-s)^{-\frac{3}{4}}\|\rho(u-b)\cdot  v\sqrt{M}\|_{\mathcal{Z}_1\cap\mathcal{Z}_2}\mathrm{d}s\nonumber\\
\leq\,& C(1+\delta)\int_0^t(1+t-s)^{-\frac{3}{4}}\mathcal{E}(s)\mathrm{d}s\nonumber\\
\leq\,&C\mathcal{E}_{0,\infty}(t)\int_0^t(1+t-s)^{-\frac{3}{4}}(1+s)^{-\frac{3}{2}}\mathrm{d}s\nonumber\\
\leq\,&C(1+t)^{-\frac{3}{4}}\mathcal{E}_{0,\infty}(t),
\end{align}
and
\begin{align}\label{G4.30}
&\|K_6(t)\|_{\mathcal{Z}_2}+\|K_7(t)\|_{\mathcal{Z}_2}\nonumber\\
\leq \,&
C\int_0^t(1+t-s)^{-\frac{3}{4}}\|S_\rho (s)\|_{\mathcal{Z}_1\cap\mathcal{Z}_2}\mathrm{d}s +C\int_0^t(1+t-s)^{-\frac{3}{4}}\|S_u (s)\|_{\mathcal{Z}_1\cap\mathcal{Z}_2}\mathrm{d}s\nonumber\\
\leq\,&C\int_0^t(1+t-s)^{-\frac{3}{4}}\big(\|\rho\|_{H^1}\|\nabla u\|_{H^1}+\|(\rho,u)\|_{L^{\infty}}\|\nabla (\rho,u)\|\big)\mathrm{d}s\nonumber\\
&+C\int_0^t(1+t-s)^{-\frac{3}{4}}\big(\|(\rho,u)\|_{H^1}\|\nabla (\rho,u)\|+ \|(\rho,u)\|_{L^{\infty}}\|\nabla (\rho,u)\|  \big)\mathrm{d}s\nonumber\\
&+C\int_0^t(1+t-s)^{-\frac{3}{4}}\big(\|a\|(\|u\|+\|u \|_{L^{\infty}})+\|\nabla^2 u\|(\|\rho\|+\|\rho\|_{L^{\infty}})  \big)\mathrm{d}s\nonumber\\
\leq\,&C\int_0^t(1+t-s)^{-\frac{3}{4}}\mathcal{E}(s)\mathrm{d}s\nonumber\\
\leq\,&C\mathcal{E}_{0,\infty}(t)\int_0^t(1+t-s)^{-\frac{3}{4}}(1+s)^{-\frac{3}{2}}\mathrm{d}s\nonumber\\
\leq\,&C(1+t)^{-\frac{3}{4}}\mathcal{E}_{0,\infty}(t).
\end{align}
To handle the term $K_2(t)$, we utilize $\mathbf{P}_0G=0$ and \eqref{G4.5} 
to compute
\begin{align}\label{G4.31}
\|K_2(t)\|_{\mathcal{Z}_2}^2\leq\,& C\int_0^t(1+t-s)^{-\frac{3}{2}}
\|(1+\rho)u\{\mathbf{I}-\mathbf{P}_0\}f\|_{\mathcal{Z}_1\cap\mathcal{Z}_2}^2\mathrm{d}s\nonumber\\
\leq\,& C\int_0^t(1+t-s)^{-\frac{3}{2}}
(1+\|\rho\|_{H^{2}})^2\mathcal{E}^2(s)\mathrm{d}s\nonumber\\
\leq\,&C(1+\delta)\mathcal{E}_{0,\infty}^2(t)\int_0^t(1+t-s)^{-\frac{3}{2}}(1+s)^{-3}\mathrm{d}s\nonumber\\
\leq\,&C(1+t)^{-\frac{3}{2}}\mathcal{E}_{0,\infty}^2(t).
\end{align}
For the last term $K_5(t)$, it follows from \eqref{G3.47}, \eqref{G4.6} and  $\mathbf{P}h=0$
that
\begin{align}\label{G4.32}
\|K_5(t)\|_{\mathcal{Z}_2}^2   \leq\,& C\int_0^t(1+t-s)^{-\frac{3}{2}}
\|\nu^{-\frac{1}{2}}\rho\mathcal{L}\{\mathbf{I}-\mathbf{P}\}f\|_{\mathcal{Z}_1\cap\mathcal{Z}_2}^2\mathrm{d}s\nonumber\\
\leq\,& C\int_0^t(1+t-s)^{-\frac{3}{2}}
\mathcal{E}(s)\mathcal{D}(s)\mathrm{d}s\nonumber\\
\leq\,& C\mathcal{E}_{0,\infty}(t)\int_0^t(1+t-s)^{-\frac{3}{2}}(1+s)^{-\frac{3}{2}}
\mathcal{D}(s)\mathrm{d}s\nonumber\\
\leq\,& C(1+t)^{-\frac{3}{2}}\mathcal{E}_{0,\infty}(t)\int_0^t
\mathcal{D}(s)\mathrm{d}s\nonumber\\
\leq\,& C\mathcal{E}(0)(1+t)^{-\frac{3}{2}}\mathcal{E}_{0,\infty}(t).
\end{align}
By substituting \eqref{G4.28}--\eqref{G4.32} into \eqref{G4.27},
we end up with
\begin{align*}
\|U(t)\|^2_{\mathcal{Z}_2}\leq C(1+t)^{-\frac{3}{2}} \big(\|U_0\|_{\mathcal{Z}_1\cap\mathcal{Z}_2}^2+\mathcal{E}(0)\mathcal{E}_{0,\infty}(t)+\mathcal{E}_{0,\infty}^2(t)\big),  
\end{align*}
which, together with \eqref{G4.22}, ensures that 
\begin{align*}
\mathcal{E}(t)\leq C(1+t)^{-\frac{3}{2}}\big(\|U_0\|_{\mathcal{Z}_1\cap\mathcal{H}^2}^2+\mathcal{E}(0)\mathcal{E}_{0,\infty}(t)+\mathcal{E}_{0,\infty}^2(t)\big).  
\end{align*}
This combines the smallness of $\mathcal{E}(0)$ implies that 
\begin{align*}
\mathcal{E}_{0,\infty}(t)\leq C\big(\|U_0\|_{\mathcal{Z}_1\cap\mathcal{H}^2}^2+\mathcal{E}_{0,\infty}^2(t)\big).  
\end{align*}
Thanks to the smallness of $\|U_0\|_{\mathcal{Z}_1\cap\mathcal{H}^2}^2$ and
Lemma \ref{LLA.4}, we consequently have
\begin{align*}
\mathcal{E}_{0,\infty}(t)\leq C\|U_0\|_{\mathcal{Z}_1\cap\mathcal{H}^2}^2, 
\end{align*}
for any $t>0$. As a result, the following inequality holds
\begin{align*}
\mathcal{E}(t)\leq C(1+t)^{-\frac{3}{2}}\|U_0\|_{\mathcal{Z}_1\cap\mathcal{H}^2}^2, 
\end{align*}
which completes the proof of \eqref{G1.8}. \hfill $\square$

\section{Time-decay rates of the derivatives of strong solutions}

In the previous   section, in order to obtain $\eqref{G1.8}$, we have required    the smallness of $\|(f_0,\rho_0,u_0)\|_{\mathcal{Z}_1}$, by the classical energy method. In what follows, we develop the refined energy estimates to the problem \eqref{I-3}--\eqref{I-6}. In fact, we 
  only need the norm  $\|(f_0,\rho_0,u_0)\|_{\mathcal{Z}_1}$ to be
bounded to  establish the optimal time-decay rates of solutions  $(f,\rho,u)$ (see \eqref{G1.8}) and their gradients $\nabla(f,\rho,u)$ (see \eqref{G1.9}). In addition, the decay rates of solutions 
in $L^p$-norm $(2\leq p\leq \infty)$ (see \eqref{G1.11} and \eqref{G1.12})  can  also be achieved and hence complete the proof of Theorem \ref{T1.2}.

\subsection{Decay-in-time estimates of strong solutions}

First,  we give the  estimates on the  second-order derivatives of $(f,\rho,u)$. 
\begin{lem}\label{L5.1}
For the strong solutions to the Cauchy problem \eqref{I-3}--\eqref{I-6},  there exists a  constant $\lambda_{12}>0$, such that
\begin{align}\label{G5.1}
&\frac{1}{2} \frac{\mathrm{d}}{\mathrm{d} t}\big(\|\nabla^2 f\|_{L_v^2(L^2)}^2+\|\nabla^2(u,\rho)\|^2\big) \nonumber\\
& \qquad + \lambda_{12} \big(\|\{\mathbf{I}-\mathbf{P}\} \nabla^2 f\|_\nu^2 +\left\|\nabla^2(b-u)\right\|^2+\|\nabla^3u\|^2\big)  
  \leq C\delta\|\nabla^2(a,b,\rho)\|_{L^2}^2,
\end{align}
for any $0 \leq t \leq T$. 
\end{lem}    

\begin{rem}
In fact, using the previous results stated in Lemma \ref{G3.2}, we only need to consider the case  $|\alpha|$ = 2 in \eqref{G3.12}. By utilizing \eqref{G3.1}, Lemmas \ref{LA.1}--\ref{LLA.3}, the basic inequality $\|\nabla^2u\|\leq (\|\nabla^2(u-b)\|+\|\nabla^2 b\|)$ and Hölder’s and Young's inequalities, we easily get
the desired \eqref{G5.1}.
\end{rem}

Next, we apply the low- and  high-frequency decomposition method (see \cite{Ww-CMS-2024,Li-Ni-Wu})
to obtain more refined energy estimates of the strong solutions to  \eqref{I-3}--\eqref{I-6}.
We start to estimate the dissipation of $(a^{H},b^{H})$. Applying
the operator $\phi_1(D_x)$ (see the definition in the Appendix) on \eqref{G3.17}, we have
\begin{equation}\label{G5.2}
\left\{\begin{aligned}
&\partial_{t}a^{H}+\nabla\cdot b^{H}=0,\\
&\partial_{t} b_i^{H}+\partial_{i} a^{H}+\sum_{j=1}^3\partial_j\Gamma_{ij}(\{\mathbf{I}-\mathbf{P}\}f)^{H}=\big((1+\rho)(u_i-b_i)\big)^{H}+\big((1+\rho)u_ia\big)^{H},  \\
&\partial_{i}b_j^{H}+\partial_j b_i^{H}-\big((1+\rho)(u_ib_j+u_jb_i)\big)^{H}=-\partial_t \Gamma_{ij}(\{\mathbf{I}-\mathbf{P}\}f)^{H}+\Gamma_{ij}(\mathfrak{l}+\mathfrak{r}+\mathfrak{s})^{H},  
 \end{aligned}
 \right.
\end{equation}
for $1\leq i,j\leq 3$, where $\mathfrak{l},\mathfrak{r},\mathfrak{s}$, and $\Gamma_{ij}(\cdot)$ are defined 
by \eqref{lrs-1}--\eqref{gammaa}. Therefore, we define the following temporal functional
$\mathcal{E}_0^H(t)$, which only contains the high frequency part of $f(t,x,v)$, as
\begin{align*}
\mathcal{E}_0^H(t) :=  \sum_{i, j=1}^3 \int_{\mathbb{R}^3} \nabla(\partial_i b_j^H+\partial_j b_i^H) \nabla \Gamma_{ij}(\{\mathbf{I}-\mathbf{P}\} f)^H \mathrm{d}x-\int_{\mathbb{R}^3} \nabla a^H \nabla\nabla\cdot b^H \mathrm{d} x.
\end{align*} 

We have  
\begin{lem}\label{L5.2}
For the strong solutions to the Cauchy problem \eqref{I-3}-\eqref{I-6}, there exists a positive constant $\lambda_{13}>0$, such that
\begin{align}\label{G5.4}
& \frac{\mathrm{d}}{\mathrm{d} t} \mathcal{E}_0^H(t)+\lambda_{13}(\|\nabla^2 a^H\|^2+\|\nabla^2 b^H\|^2)\nonumber\\
\, &\quad \leq C\|\nabla^2\{\mathbf{I}-\mathbf{P}\} f\|^2+C\|\nabla^2(u-b)\|^2 +C\delta\|\nabla^2(a,b,\rho)\|^2,
\end{align}
for any $0 \leq t\leq T$. 
\end{lem}   

\begin{proof}
By applying $\nabla$ to \eqref{G5.2}$_3$, and then multiplying by $\nabla(\partial_ib_j^H+\partial_jb_i^H)$ and taking integration, we have 
\begin{align}\label{G5.5}
&\sum_{i,j=1}^3\|\nabla(\partial_ib_j^H+\partial_jb_i^H)\|^2\nonumber\\
=\,&\sum_{i,j=1}^3\int_{\mathbb{R}^3}\nabla(\partial_ib_j^H+\partial_jb_i^H)
\nabla\Big(\big((1+\rho)(u_ib_j+u_jb_i)\big)^H+\Gamma_{ij}(\mathfrak{l}+\mathfrak{r}+\mathfrak{s})^H\Big)\mathrm{d}x        \nonumber\\
&-\sum_{i,j=1}^3\int_{\mathbb{R}^3}\nabla(\partial_ib_j^H+\partial_jb_i^H)\nabla\partial_t \Gamma_{ij}(\{\mathbf{I}-\mathbf{P}\}f)^H  {\rm d}x  \nonumber\\
=\,&-\frac{{\rm d}}{{\rm d}t}\sum_{i,j=1}^3\int_{\mathbb{R}^3}\nabla(\partial_ib_j^H+\partial_jb_i^H)\nabla\Gamma_{ij}(\{\mathbf{I}-\mathbf{P}\}f)^H\mathrm{d}x\nonumber\\
&-2\sum_{i,j=1}^3\int_{\mathbb{R}^3}\nabla \partial_t b_i^H\nabla\partial_j\Gamma_{ij}(\{\mathbf{I}-\mathbf{P}\}f)^H\mathrm{d}x\nonumber\\
&+\sum_{i,j=1}^3\int_{\mathbb{R}^3}\nabla(\partial_ib_j^H
+\partial_jb_i^H)\nabla\Big(\big((1+\rho)(u_ib_j+u_jb_i)\big)^H+\Gamma_{ij}(\mathfrak{l}+\mathfrak{r}+\mathfrak{s})^H      \Big)\mathrm{d}x.
\end{align}
From \eqref{G5.2}$_2$, Young's inequality, Lemmas \ref{LA.1}--\ref{LLA.3} and Lemma \ref{LA.7}, the second 
term on the right-hand side of \eqref{G5.5} can be estimated as 
\begin{align}\label{G5.6}
&-2\sum_{i,j=1}^3\int_{\mathbb{R}^3}\nabla \partial_t b_i^H\nabla\partial_j\Gamma_{ij}(\{\mathbf{I}-\mathbf{P}\}f)^H\mathrm{d}x\nonumber\\  
=\,&2 \sum_{i,j=1}^3\int_{\mathbb{R}^3}\nabla \bigg(\partial_i a^H+\sum_{m=1}^3\partial_m\Gamma_{im}(\{\mathbf{I}-\mathbf{P}\}f)^H\bigg)\nabla\partial_j\Gamma_{ij}(\{\mathbf{I}-\mathbf{P}\}f)^H\mathrm{d}x\nonumber\\ 
&-2\sum_{i,j=1}^3\int_{\mathbb{R}^3}\nabla\Big(\big((1+\rho)(u_i-b_i)\big)^H+\big((1+\rho)u_i a\big)^H  \Big)\nabla\partial_j\Gamma_{ij}(\{\mathbf{I}-\mathbf{P}\}f)^H\mathrm{d}x\nonumber\\
\leq\,& \frac{1}{4}\|\nabla^2 a^H\|^2+C\|\nabla^2 \{\mathbf{I}-\mathbf{P}\}f \|^2+C\|\nabla(u-b)^H\|^2\nonumber\\
&+C\big(\|\nabla\rho\|_{L^6}^2\|u-b\|_{L^3}^2+\|\rho\|_{L^3}^2\|\nabla (u-b)\|_{L^6}^2\big)+C\|\nabla \rho\|_{L^6}^2\|a\|_{L^6}^2\|u\|_{L^6}^2\nonumber\\
&+C(1+\|\rho\|_{H^2})^2(\|\nabla u\|_{L^6}^2\|a\|_{L^3}^2+\|\nabla a\|_{L^6}^2\|u\|_{L^3}^2)\nonumber\\
\leq\,&\frac{1}{4}\|\nabla^2 a^H\|^2+C\|\nabla^2 \{\mathbf{I}-\mathbf{P}\}f \|^2+C\|\nabla^2(u-b)\|^2+C\delta\|\nabla^2(a,b,\rho)\|^2.
\end{align}
For the remaining term on the right hand-side of \eqref{G5.5}, by
  Young's inequality, we obtain
\begin{align}\label{NJK5.6}
&\sum_{i,j=1}^3\int_{\mathbb{R}^3}\nabla(\partial_ib_j^H+\partial_jb_i^H)
\nabla\Big(\big((1+\rho)(u_ib_j+u_jb_i)\big)^H+\Gamma_{ij}(\mathfrak{l}+\mathfrak{r}+\mathfrak{s})^H     \Big)\mathrm{d}x\nonumber\\
\,&\quad \leq\frac{1}{2}\sum_{i,j=1}^3\|\nabla(\partial_ib_j^H+\partial_jb_i^H)\|^2+C\sum_{i,j=1}^3
\big\|\nabla\big((1+\rho)(u_ib_j+u_jb_i)\big)^H\big\|^2\nonumber\\
&\qquad +C\sum_{i,j=1}^3\big(\|\nabla\Gamma_{ij}(\mathfrak{l})^H\|^2+\|\nabla\Gamma_{ij}(\mathfrak{r})^H\|^2
+\|\nabla\Gamma_{ij}(\mathfrak{s})^H\|^2   \big).
\end{align}
Similar to \eqref{G3.24}, by a direct calculation, the last two terms on the above inequality can  be estimated as 
\begin{align}\label{G5.8}
&\sum_{i,j=1}^3
\big\|\nabla\big((1+\rho)(u_ib_j+u_jb_i)\big)^H\big\|^2+\sum_{i,j=1}^3\Big(\|\nabla\Gamma_{ij}(\mathfrak{l})^H\|^2
+\|\nabla\Gamma_{ij}(\mathfrak{r})^H\|^2+\|\nabla\Gamma_{ij}(\mathfrak{s})^H\|^2                  \Big)\nonumber\\
\leq\,&C(1+\|\rho\|_{H^2})^2(\|b\|_{L^3}^2\|\nabla u\|_{L^6}^2+\|u\|_{L^3}^2\|\nabla b\|_{L^6}^2)\nonumber\\
&+C\|\nabla \rho\|_{L^6}^2\|u\|_{L^6}^2\|b\|_{L^6}^2+C\|\nabla \{\mathbf{I}-\mathbf{P}\}f^H\|^2+ C\|\nabla^2 \{\mathbf{I}-\mathbf{P}\}f^H\|^2   \nonumber\\
&+C\|u\|_{L^3}^2\|\nabla  \{\mathbf{I}-\mathbf{P}\}f \|_{L_v^2(L^6)}^2+C\|\nabla u\|_{L^6}^2\|  \{\mathbf{I}-\mathbf{P}\}f \|_{L_v^2(L^3)}^2\nonumber\\
&+C\|\nabla\rho\|_{L^6}^2\Big(\|\{\mathbf{I}-\mathbf{P}\}f \|_{L_v^2(L^3)}^2+\|u\|_{L^6}^2\|\{\mathbf{I}-\mathbf{P}\}f \|_{L_v^2(L^6)}^2\Big)   \nonumber\\
\leq\,&C\delta\big(\|\nabla^2(\rho,b)\|^2+\|\nabla^2(u-b)\|^2\big)+C\| \nabla^2\{\mathbf{I}-\mathbf{P}\}f \|^2.
\end{align}
Substituting \eqref{G5.8} into \eqref{NJK5.6}, we further obtain
\begin{align}\label{NJK5.8}
&\sum_{i,j=1}^3\int_{\mathbb{R}^3}\nabla(\partial_ib_j^H+
\partial_jb_i^H)\nabla\Big(\big((1+\rho)(u_ib_j+u_jb_i)\big)^H+\Gamma_{ij}(\mathfrak{l}+\mathfrak{r}+\mathfrak{s})^H     \Big)\mathrm{d}x\nonumber\\  
&\quad \leq\frac{1}{2}\sum_{i,j=1}^3\|\nabla(\partial_ib_j^H+\partial_jb_i^H)\|^2+C\delta\big(\|\nabla^2(\rho,b)\|^2+\|\nabla^2(u-b)\|^2\big)+C\| \nabla^2\{\mathbf{I}-\mathbf{P}\}f \|^2.
\end{align}

It is easy to observe that
\begin{align}\label{G5.9}
\sum_{i,j=1}^3\|\nabla(\partial_ib_j^H+\partial_jb_i^H)\|^2=2\|\nabla^2 b^H\|^2+2\|\nabla\cdot\nabla b^H\|^2.   
\end{align}
Then, putting \eqref{G5.6} and \eqref{NJK5.8}--\eqref{G5.9} into \eqref{G5.5},
one gets
\begin{align}\label{G5.10}
&\frac{{\rm d}}{{\rm d}t}\sum_{i,j=1}^3\int_{\mathbb{R}^3}\nabla(\partial_ib_j^H+\partial_jb_i^H)\nabla\Gamma_{ij}(\{\mathbf{I}-\mathbf{P}\}f)^H\mathrm{d}x+\|\nabla^2b^H\|^2+\|\nabla\cdot\nabla b^H\|^2\nonumber\\
\,&\quad \leq\frac{1}{4}\|\nabla^2 a^H\|^2+C\|\nabla^2(u-b)\|^2+C\|\nabla^2 \{\mathbf{I}-\mathbf{P}\}f \|^2+C\delta\|\nabla^2(a,b,\rho)\|^2.
\end{align}
In addition, from \eqref{G5.2}$_2$, it holds
\begin{align}\label{G5.11}
\| \nabla^2 a^H\|^2 
=\,& \sum_{i=1}^3\int_{\mathbb{R}^3}\nabla\partial_i a^H\nabla\partial_i a^H\mathrm{d}x\nonumber\\
=\,&\sum_{i=1}^3\int_{\mathbb{R}^3}\nabla\partial_i a^H\nabla\Big(\big((1+\rho)(u_i-b_i)\big)^H+\big((1+\rho)u_ia\big)^H        \Big)\mathrm{d}x\nonumber\\
&-\sum_{i=1}^3\int_{\mathbb{R}^3}\nabla\partial_i a^H\nabla\bigg(\partial_{t} b_i^H+\sum_{j=1}^3\partial_j\Gamma_{ij}(\{\mathbf{I}-\mathbf{P}\}f)^H\bigg)\mathrm{d}x\nonumber\\
=\,&-\frac{{\rm d}}{{\rm d}t}\sum_{i=1}^3\int_{\mathbb{R}^3}\nabla\partial_i a^H\nabla b_i^H\mathrm{d}x+\sum_{i=1}^3\int_{\mathbb{R}^3}\nabla\partial_i\partial_t a^H
\nabla b_i^H\mathrm{d}x\nonumber\\
&+\sum_{i=1}^3\int_{\mathbb{R}^3}\nabla\partial_i a^H\nabla\bigg(-\sum_{j=1}^3\partial_j\Gamma_{ij}(\{\mathbf{I}-\mathbf{P}\}f)^H 
+\big((1+\rho)(u_i-b_i)\big)^H+\big((1+\rho)u_ia\big)^H        \bigg)\mathrm{d}x.
\end{align}  
Owing to \eqref{G5.2}$_1$, one has
\begin{align}\label{G5.12}
\sum_{i=1}^3\int_{\mathbb{R}^3}\nabla\partial_i\partial_t a^H
\nabla b_i^H\mathrm{d}x=-\int_{\mathbb{R}^3}\nabla\partial_t a^H
\nabla\nabla\cdot b^H\mathrm{d}x =\|\nabla\nabla\cdot b^H\|^2.  
\end{align}
By using Lemmas \ref{LA.1}--\ref{LLA.3} and Lemma \ref{LA.7}, Hölder’s and Young's inequalities, we arrive at 
\begin{align}\label{G5.13}
 &\sum_{i=1}^3\int_{\mathbb{R}^3}\nabla\partial_i a^H\nabla\bigg(-\sum_{j=1}^3\partial_j\Gamma_{ij}(\{\mathbf{I}-\mathbf{P}\}f)^H 
+\big((1+\rho)(u_i-b_i)\big)^H+\big((1+\rho)u_ia\big)^H        \bigg)\mathrm{d}x\nonumber\\
\leq\,&\frac{1}{4}\|\nabla^2 a^H\|^2+C\|\nabla^2\{\mathbf{I}-\mathbf{P}\}f^H\|^2+C\|\nabla(u-b)^H\|^2+C\|\rho\|_{L^3}^2\|\nabla(u-b)\|_{L^6}^2\nonumber\\
&+C\|\nabla\rho\|_{L^6}^2\|u-b\|_{L^3}^2+C(1+\|\rho\|_{H^2})^2(\|u\|_{L^3}^2\|\nabla a\|_{L^6}^2+\|a\|_{L^3}\|\nabla u\|_{L^6}^2)\nonumber\\
&+C\|\nabla\rho\|_{L^6}^2\|u\|_{L^6}^2\|a\|_{L^6}^2      \nonumber\\
\leq\,&\frac{1}{4}\|\nabla^2 a^H\|^2+C\|\nabla^2(u-b)\|^2+C\|\nabla^2\{\mathbf{I}-\mathbf{P}\}f\|^2+C\delta\|\nabla^2(a,b,\rho)\|^2.
\end{align}
Inserting \eqref{G5.12}--\eqref{G5.13} into \eqref{G5.11}, we deduce that
\begin{align}\label{G5.14}
&-\frac{{\rm d}}{{\rm d}t}\int_{\mathbb{R}^3}\nabla a^H\nabla \nabla\cdot b^H\mathrm{d}x+\frac{1}{2}\|\nabla^2 a^H\|^2\nonumber\\
\,&\quad\leq\|\nabla \nabla\cdot b^H\|^2+C\|\nabla^2(u-b)\|^2+C\|\nabla^2\{\mathbf{I}-\mathbf{P}\}f\|^2+C\delta\|\nabla^2(a,b,\rho)\|^2.
\end{align}
Adding \eqref{G5.10} and \eqref{G5.14} up, one gets
\begin{align*}
&\frac{{\rm d}}{{\rm d}t}\mathcal{E}^H_0(t)+\|\nabla^2 b^H\|^2+\frac{1}{4}\|\nabla^2 a^H\|^2\leq C\|\nabla^2(u-b)\|^2+C\|\nabla^2\{\mathbf{I}-\mathbf{P}\}f\|^2+C\delta\|\nabla^2(a,b,\rho)\|^2.
\end{align*}
Thus,  \eqref{G5.4} holds.
\end{proof}
\begin{lem}
For the strong solutions to the Cauchy problem \eqref{I-3}-\eqref{I-6}, there exists a positive constant $\lambda_{14}>0$, such that    
\begin{align}\label{G5.15}
&-\frac{\rm d}{{\rm d}t}\int_{\mathbb{R}^3}\nabla{\rm div} u\nabla\rho^H\mathrm{d}x+\lambda_{14}\|\nabla^2 \rho^H\|^{2}\nonumber\\
\,&\quad \leq C\Big(\|\nabla^2(b-u) \|^2+\|\nabla^2\rho^L\|^2+\|\nabla^2 u\|_{H^1}^2\Big)+C\delta\|\nabla^2(a,b,\rho)\|^2, 
\end{align} 
for any $0 \leq t \leq T$. 
\end{lem}
\begin{proof}
Applying $\nabla$ to \eqref{I-5},
then multiplying the resulting identity by $\nabla^2\rho^H$ and
taking integration, one has   
\begin{align}\label{G5.16}
P^\prime(1)\|\nabla^2\rho^H\|^2=\,&-\int_{\mathbb{R}^3}\nabla^2\rho^H\cdot\nabla\partial_t u\mathrm{d}x+\int_{\mathbb{R}^3}\nabla^2\rho^H\cdot\nabla(b-u)\mathrm{d}x+ \int_{\mathbb{R}^3}\nabla^2\rho^H\cdot\nabla^2\rho^L\mathrm{d}x \nonumber\\
&+\int_{\mathbb{R}^3}\nabla^2\rho^H\cdot\nabla(au)\mathrm{d}x-\int_{\mathbb{R}^3}\nabla^2\rho^H\cdot\nabla(u\cdot\nabla u)\mathrm{d}x\nonumber\\
&-\int_{\mathbb{R}^3}\nabla^2\rho^H\cdot\nabla\bigg(\Big(\frac{P^{\prime}(1)}{1+\rho}-P^{\prime}(1)       \Big)\nabla\rho\bigg)\mathrm{d}x+\int_{\mathbb{R}^3}\nabla^2\rho^H\cdot\nabla\Big(\frac{\mu\Delta u}{1+\rho}        \Big)\mathrm{d}x\nonumber\\
\equiv:\,&\sum_{j=33}^{39}I_j.
\end{align}
By making use of Lemmas \ref{LA.1}--\ref{LLA.3} and Young's inequality, the term $I_{33}$ can be estimated as 
\begin{align}\label{G5.17}
I_{33}=\,&-\frac{{\rm d}}{{\rm d}t}\int_{\mathbb{R}^3}\nabla^2  \rho^H\cdot\nabla u\mathrm{d}x+\int_{\mathbb{R}^3}\nabla^2  \partial_t\rho^H\cdot\nabla u\mathrm{d}x\nonumber\\
=\,&-\frac{{\rm d}}{{\rm d}t}\int_{\mathbb{R}^3}\nabla^2  \rho^H\cdot\nabla u\mathrm{d}x+\int_{\mathbb{R}^3}\nabla\big((\rho+1){\rm div} u+\nabla\rho\cdot u            \big)^H \nabla{\rm div} u\mathrm{d}x\nonumber\\
\leq\,&\frac{{\rm d}}{{\rm d}t}\int_{\mathbb{R}^3}\nabla  \rho^H\nabla{\rm div}u\mathrm{d}x+C\|\nabla{\rm div}u
\|^2+C(\|\nabla\rho\|_{L^3}^2\|\nabla u\|_{L^6}^2+\|\nabla^2 \rho\|^2\| u\|_{L^{\infty}}^2)\nonumber\\
&+C(1+\|\rho\|_{H^2})^2\|\nabla{\rm div}u\|^2+C\|\nabla\rho\|_{L^3}^2\|\nabla u\|_{L^6}^2    \nonumber\\
\leq\,&\frac{{\rm d}}{{\rm d}t}\int_{\mathbb{R}^3}\nabla  \rho^H\nabla{\rm div}u\mathrm{d}x+C\|\nabla^2 u\|^2+C\delta\|\nabla^2\rho\|^2.
\end{align}
Thanks to  Young's inequality, Lemmas \ref{LA.1}--\ref{LLA.3},   Lemma \ref{LA.7}, and \eqref{rho-1}, we  obtain 
that 
\begin{align}
I_{34}=\,&-\int_{\mathbb{R}^3}\nabla\rho^H\nabla{\rm div}(b-u)\mathrm{d}x \leq \frac{1}{7}P^{\prime}(1)\|\nabla^2\rho^H\|^2+C\|\nabla^2(u-b)\|^2,\label{G5.18a}\\
I_{35}\leq\,& \frac{1}{7}P^{\prime}(1)\|\nabla^2\rho^H\|^2+C\|\nabla^2 \rho^L\|^2,\label{G5.18b}\\
I_{36}\leq\,& \frac{1}{7}P^{\prime}(1)\|\nabla^2\rho^H\|^2+C(\|a\|_{L^3}^2\|\nabla u\|_{L^6}^2+\|u\|_{L^3}^2\|\nabla a\|_{L^6}^2)\nonumber\\
\leq\,& \frac{1}{7}P^{\prime}(1)\|\nabla^2\rho^H\|^2+C\delta\|\nabla^2(a,u)\|^2,\label{G5.18c}\\
I_{37}\leq\,& \frac{1}{7}P^{\prime}(1)\|\nabla^2\rho^H\|^2+C\|u\|_{L^{\infty}}^2\|\nabla^2 u\|^2+C\|\nabla u\|_{L^3}^2\|\nabla u\|_{L^6}^2\nonumber\\
\leq\,&\frac{1}{7}P^{\prime}(1)\|\nabla^2\rho^H\|^2+C\delta\|\nabla^2 u\|^2,\label{G5.18d}\\
I_{38}\leq\,& \frac{1}{7}P^{\prime}(1)\|\nabla^2\rho^H\|^2+C\|\rho\|_{L^{\infty}}^2\|\nabla^2 \rho\|^2+C\|\nabla \rho\|_{L^3}^2\|\nabla \rho\|_{L^6}^2\nonumber\\
\leq\,&\frac{1}{7}P^{\prime}(1)\|\nabla^2\rho^H\|^2+C\delta\|\nabla^2 \rho\|^2,\label{G5.18e}\\
I_{39}\leq\,& \frac{1}{7}P^{\prime}(1)\|\nabla^2\rho^H\|^2+C\Big\|\frac{1}{1+\rho}\Big\|_{L^{\infty}}^2\|\nabla^3 u\|^2+C\Big\|\nabla \frac{1}{1+\rho}\Big\|_{L^3}^2\|\nabla^2 u\|_{L^6}^2\nonumber\\
\leq\,&\frac{1}{7}P^{\prime}(1)\|\nabla^2\rho^H\|^2+C\|\nabla^3 u\|^2.\label{G5.18}
\end{align}
Inserting \eqref{G5.17}--\eqref{G5.18} into \eqref{G5.16}, we eventually get
the desired \eqref{G5.15}.
\end{proof}
In the what follows, we define the new energy functional $\mathcal{E}_1(t)$ and the corresponding dissipation $\mathcal{D}_1(t)$ as 
\begin{align*}
\mathcal{E}_1(t):=\,&\|\nabla^2 f\|_{L_v^2(L^2)}^2+\|\nabla^2(u,\rho)\|^2
+\tau_{6}\mathcal{E}^H_0(t)-\tau_7\int_{\mathbb{R}^3}\nabla{\rm div} u\nabla\rho^H\mathrm{d}x
\nonumber\\
\mathcal{D}_1(t):=\,&\|\nabla^2(b-u)\|_{H^2}^2+\|\nabla^2(a^H,b^H,\rho^H)\|^2+\|\nabla^3 u\|^{2}+\,\|\nabla^2\{\mathbf{I}-\mathbf{P}\}f\|_{\nu}^2,
\end{align*}
where $0<\tau_6, \tau_7\ll 1$  are sufficiently small constants.

Adding \eqref{G5.1} to $\tau_6\times$\eqref{G5.4} and $\tau_7\times$\eqref{G5.15}, we deduce that
\begin{align}\label{G5.20}
\frac{\mathrm{d}}{\mathrm{d}t}\mathcal{E}_1(t)+\lambda_{16}\mathcal{D}_1(t)\leq C
\|\nabla^2(a^L,b^L,\rho^L,u^L)\|^2,
\end{align}
for some $\lambda_{16}>0$.
Here, Lemma \ref{LA.7} and following facts have been used:
\begin{align*}
\|\nabla^2 u\|^2\leq\,& C(\|\nabla^2 u^L\|^2+\|\nabla^2 u^H\|^2)\leq C(\|\nabla^2 u^L\|^2+\|\nabla^3 u\|^2),\nonumber\\
\|\nabla^2(a,b,\rho)\|\leq\,& C\big(\|\nabla^2(a^L,b^L,\rho^L)\|^2+\|\nabla^2(a^H,b^H,\rho^H)\|^2\big).
\end{align*}
According to Young's inequality and Lemma \ref{LA.7}, we have
\begin{align*}
\int_{\mathbb{R}^3}\nabla{\rm div}u\nabla\rho^H\mathrm{d}x\leq&\, \frac{1}{2}\|
\nabla^2 u\|^2+\frac{1}{2}\|\nabla\rho^H\|^2\leq \frac{1}{2}\|
\nabla^2 u\|^2+\frac{1}{2}\|\nabla^2\rho\|^2,\nonumber\\ 
\mathcal{E}_0^H(t)\leq\,& \|\nabla^2 b\|^2+\frac{1}{2}\|\nabla a^H\|^2+\frac{1}{2}\|\nabla\{\mathbf{I}-\mathbf{P}\}f^H\|^2\nonumber\\
\leq\,& \|\nabla^2 b\|^2+\frac{1}{2}\|\nabla^2 a\|^2+\frac{1}{2}\|\nabla^2\{\mathbf{I}-\mathbf{P}\}f\|^2,
\end{align*}
which implies that
\begin{align*}
\mathcal{E}_1(t)\backsim \|\nabla^2 f\|_{L_v^2(L^2)}^2+\|\nabla^2(u,\rho)\|^2.
\end{align*}
Applying $\phi_0(D_x)$ to \eqref{G4.4} gives
\begin{align*}
U^L(t)=\mathbb{A}^L(t)U_0+\int_0^t\mathbb{A}^L(t-s)(S_f(s),0,0)\mathrm{d}s,
\end{align*}
with $U^L(t) := (f^L(t), \rho^L(t), u^L(t))$.

\begin{thm}\label{T5.4}
Let $1\leq q\leq2$. For all $\alpha$, $\alpha^{\prime}$ with 
$\alpha^{\prime}\leq\alpha$ and $m=|\alpha-\alpha^{\prime}|$, we have
\begin{align}
& \|\partial^{\alpha}\mathbb{A}^L(t)U_{0}\|_{\mathcal{Z}_{2}}
\leq C(1+t)^{-\frac{3}{2}\Big(\frac{1}{q}-\frac{1}{2}\Big)-\frac{m}{2}}
\|\partial^{\alpha^{\prime}}U_{0}\|_{\mathcal{Z}_{q}},\label{G5.23}\\
&\bigg\|\partial^{\alpha}\int_{0}^{t}\mathbb{A}^L(t-s)(S_f(s),0,0){\rm d}s\bigg\|_{\mathcal{Z}_{2}}^{2}\nonumber\\
&\quad \leq C\int_0^t(1+t-s)^{-3\Big(\frac{1}{q}-\frac{1}{2}\Big)-m}
\Big\|\partial^{\alpha^{\prime}}\big(G(s),\nu^{-\frac{1}{2}}h(s)\big)\Big\|_{\mathcal{Z}_{q}}^{2}
{\rm d}s,\label{G5.24}
\end{align}
for any $t\geq 0$.
\end{thm}

\begin{rem}
Through the same process used in Theorem \ref{T4.1}, we can easily get \eqref{G5.23} and \eqref{G5.24}. 
We point out that  in Theorem \ref{T5.4},  only the low-frequency parts of solutions  are included, which
is different from Theorem \ref{T4.1}.
\end{rem}

With Theorem \ref{T5.4} in hand, we shall  establish the estimate of 
$ \|\nabla^2 f\|_{L_v^2(L^2)}^2+\|\nabla^2(u,\rho)\|^2$.
Adding $\lambda_{16}\|\nabla^2(a^L,b^L,\rho^L,u^L)\|_{L^2}$ to both sides of \eqref{G5.20} yields
\begin{align}\label{NJKKK5.24}
\frac{\mathrm{d}}{\mathrm{d}t}\mathcal{E}_1(t)+\lambda_{16}\mathcal{E}_1(t)\leq C
\|\nabla^2(a^L,b^L,\rho^L,u^L)\|^2.
\end{align}
Applying Gr\"{o}nwall’s inequality to \eqref{NJKKK5.24}, we derive that 
\begin{align}\label{G5.25}
\mathcal{E}_1(t)\leq e^{-\lambda_{16}t}\mathcal{E}_1(0)+C\int_{0}^te^{-\lambda_{16}(t-s)} \|\nabla^2U^L(s)\|_{\mathcal{Z}_2}^2\mathrm{d}s,  
\end{align}
for some $\lambda_{16}>0$.  From Lemma \ref{LA.7}, it holds
\begin{align*}
 \|\nabla^2 U^L(t)\|_{\mathcal{Z}_2}^2 \leq C\|\nabla^n U^L(t)\|_{\mathcal{Z}_2}^2,  
\end{align*}
for $n=0,1$.
This, together with \eqref{G5.25}, leads to
\begin{align}\label{G5.27}
\mathcal{E}_1(t)\leq e^{-\lambda_{16}t}\mathcal{E}_1(0)+C\int_{0}^te^{-\lambda_{16}(t-s)} \|\nabla^n U^L(s)\|_{\mathcal{Z}_2}^2\mathrm{d}s,  
\end{align}
for $n=0,1$.
Therefore, it is crucial to estimate $\|\nabla^n U^L(t)\|_{\mathcal{Z}_2}^2$. To
do this, we apply $\nabla^n\phi_0(D_x)$ to \eqref{G4.4} to obtain
\begin{align}\label{G5.28}
\nabla^n U^L(t)=\sum_{j=1}^7\nabla^n K^L_j(t), 
\end{align}
for $n=0,1$.
Here, $K_j(j=1,2,\dots,7)$ are defined in \eqref{G4.27}. 

Below we give some  estimates about the decay rates of solutions and their 
spatial derivatives. Define
\begin{align}\label{G5.29}
\mathcal{E}_{1,\infty}(t):=\sup_{0\leq s\leq t}\,&\bigg\{\sum_{n=0}^1(1+s)^{\frac{3}{4}+n}\Big(\|\nabla^n f\|_{L_v^2(L^2)}^2+\|\nabla^n(\rho,u)\|^2 \Big)\nonumber\\
&+(1+s)^{\frac{7}{4}}\Big(\|\nabla^2 f\|_{L_v^2(L^2)}^2+\|\nabla^2(\rho,u)\|^2 \Big)   \bigg\}.
\end{align}

Now we estimate $\nabla^nK^L_j(j=1,2, \dots, 7)$ in \eqref{G5.29} one by one. 
Taking $q= {4}/{3}$ in \eqref{G5.23} and utilizing Lemma \ref{LLA.4}, 
one has
\begin{align}\label{G5.30}
\|\nabla^n K^L_1(t)\|_{\mathcal{Z}_2}\leq C(1+t)^{-\frac{3}{8}-\frac{n}{2}}
\|U_0\|_{\mathcal{Z}_{\frac{4}{3}}}\leq C(1+t)^{-\frac{3}{8}-\frac{n}{2}}\|U_0\|_{\mathcal{Z}_1}^\frac{1}{2}\|U_0\|_{\mathcal{Z}_2}^\frac{1}{2},
\end{align}
for $n=0,1$.

Similarly, taking $q=1,m=0$  and $q={3}/{2}, m=n$ in \eqref{G5.23} respectively,
and then using Lemmas \ref{LA.1}--\ref{LA.2}, we obtain, for  $\nabla^nK_3^L$ and $\nabla^nK_4^L$, that 
\begin{align}\label{G5.31}
&\|\nabla^n K^L_3(t)\|_{\mathcal{Z}_2}+\|\nabla^n K^L_4(t)\|_{\mathcal{Z}_2}\nonumber\\
\leq\,&  C\int_0^{\frac{t}{2}}(1+t-s)^{-\frac{3}{4}-\frac{n}{2}}\|(1+\rho)u\cdot a v\sqrt{M}\|_{\mathcal{Z}_1}\mathrm{d}s\nonumber\\
&+ C\int_{\frac{t}{2}}^{t}(1+t-s)^{-\frac{1}{4}}\big\|\nabla^n\big((1+\rho)u\cdot a v\sqrt{M}\big)\big\|_{\mathcal{Z}_{\frac{3}{2}}}\mathrm{d}s\nonumber\\
&+C\int_{0}^{\frac{t}{2}}(1+t-s)^{-\frac{3}{4}-\frac{n}{2}}\|\rho(u-b)\cdot  v\sqrt{M}\|_{\mathcal{Z}_1}\mathrm{d}s\nonumber\\
&+C\int_{\frac{t}{2}}^{t}(1+t-s)^{-\frac{1}{4}}\big\|\nabla^n\big(\rho(u-b)\cdot  v\sqrt{M}\big)\big\|_{\mathcal{Z}_{\frac{3}{2}}}\mathrm{d}s\nonumber\\
\leq\,& C\int_0^{\frac{t}{2}}(1+t-s)^{-\frac{3}{4}-\frac{n}{2}}\big((1+\|\rho\|_{L^{\infty}})\|u\|\|a\|+\|\rho\|\|(u,b)\|\big) 
  \mathrm{d}s\nonumber\\
&+C\int_{\frac{t}{2}}^t(1+t-s)^{-\frac{1}{4}}(1+\|\rho\|_{L^{\infty}})(\|\nabla^n u\|\|a\|_{L^6}+\|u\|_{L^6}\|\nabla^n a\|)\mathrm{d}s\nonumber\\
&+C\int_{\frac{t}{2}}^t(1+t-s)^{-\frac{1}{4}}\|\nabla^n \rho\|\big(\|a\|_{L^6}\|u\|_{L^{\infty}}+\|(u,b)\|_{L^6}\big)\mathrm{d}s\nonumber\\
&+C\int_{\frac{t}{2}}^t(1+t-s)^{-\frac{1}{4}}\|\nabla^n(u-b)\|\|\rho\|_{L^6}\mathrm{d}s\nonumber\\
\leq\,& C(1+\delta)\int_0^{\frac{t}{2}}(1+t-s)^{-\frac{3}{4}-\frac{n}{2}}\|(a,b,\rho,u)\|^2 
  \mathrm{d}s\nonumber\\
&+C(1+\delta)\int_{\frac{t}{2}}^t(1+t-s)^{-\frac{1}{4}}\|\nabla^n(a,b,\rho,u)\|\|\nabla(a,b,\rho,u)\|_{L^2}\mathrm{d}s\nonumber\\
\leq\,& C\mathcal{E}_{1,\infty}(t)\int_0^{\frac{t}{2}}(1+t-s)^{-\frac{3}{4}-\frac{n}{2}}(1+s)^{-\frac{3}{4}}
  \mathrm{d}s\nonumber\\
&+ C\mathcal{E}_{1,\infty}(t)\int_{\frac{t}{2}}^{t}(1+t-s)^{-\frac{1}{4}}(1+s)^{-\frac{5}{4}-\frac{n}{2}}
  \mathrm{d}s\nonumber\\
\leq\,& C(1+t)^{-\frac{1}{2}-\frac{n}{2}}\mathcal{E}_{1,\infty}(t),
\end{align}
for $n=0,1$. 

For $\nabla^n K^L_6$ and $\nabla^n K^L_7$,  taking $q=1,m=0$ and $q={3}/{2}, m=0$  in \eqref{G5.23} respectively yields
\begin{align}\label{G5.32}
&\|\nabla^n K^L_6(t)\|_{\mathcal{Z}_2}+\|\nabla^n K^L_7(t)\|_{\mathcal{Z}_2}\nonumber\\
\leq \,&
C\int_0^{\frac{t}{2}}(1+t-s)^{-\frac{3}{4}-\frac{n}{2
}}\|S_\rho (s)\|_{\mathcal{Z}_1}\mathrm{d}s\nonumber\\
&+C\int_{\frac{t}{2}}^{t}(1+t-s)^{-\frac{1}{4}-\frac{n}{2}}\|S_\rho (s)\|_{\mathcal{Z}_{\frac{3}{2}}}\mathrm{d}s\nonumber\\
&+C\int_0^{\frac{t}{2}}(1+t-s)^{-\frac{3}{4}-\frac{n}{2}}\|S_u (s)\|_{\mathcal{Z}_1}\mathrm{d}s\nonumber\\
&+C\int_{\frac{t}{2}}^{t}(1+t-s)^{-\frac{1}{4}-\frac{n}{2}}\| S_u (s)\|_{\mathcal{Z}_{\frac{3}{2}}}\mathrm{d}s\nonumber\\
\leq\,&C\int_0^{\frac{t}{2}}(1+t-s)^{-\frac{3}{4}-\frac{n}{2}}\big(\|(\rho,u)\|\|\nabla (\rho,u)\|+\|\rho\|\|\nabla^2u\|+\|a\|\|u\|\big)\mathrm{d}s\nonumber\\
&+C\int_{\frac{t}{2}}^t(1+t-s)^{-\frac{1}{4}-\frac{n}{2}}\big(\|(\rho,u)\|_{L^6}\|\nabla(\rho,u)\|+\|\rho\|_{L^6}\|\nabla^2 u\|+\|a\|\|\nabla u\|\big)\mathrm{d}s\nonumber\\
\leq\,&C\mathcal{E}_{1,\infty}(t)\int_0^{\frac{t}{2}}(1+t-s)^{-\frac{3}{4}-\frac{n}{2}}(1+s)^{-\frac{3}{4}}\mathrm{d}s\nonumber\\
&+C\mathcal{E}_{1,\infty}(t)\int_{\frac{t}{2}}^t(1+t-s)^{-\frac{1}{4}-\frac{n}{2}}(1+s)^{-\frac{5}{4}}\mathrm{d}s\nonumber\\
\leq\,& C(1+t)^{-\frac{1}{2}-\frac{n}{2}}\mathcal{E}_{1,\infty}(t).
\end{align}
for $n=0,1$. 

Next, for the term $\nabla^n K^L_2(t)$, using $\mathbf{P}_0G=0$ and 
taking $q=1, m=0$ and $q= {3}/{2}, m=n$ in \eqref{G5.24} respectively, we have 
to compute
\begin{align}\label{G5.33}
\|\nabla^nK^L_2(t)\|_{\mathcal{Z}_2}^2\leq\,& C\int_0^{\frac{t}{2}}(1+t-s)^{-\frac{3}{2}-n}
\|(1+\rho)u\{\mathbf{I}-\mathbf{P}_0\}f\|_{\mathcal{Z}_1}^2\mathrm{d}s\nonumber\\
&+ C\int_{\frac{t}{2}}^t(1+t-s)^{-\frac{1}{2}}
\big\|\nabla^n\big((1+\rho)u\{\mathbf{I}-\mathbf{P}_0\}f\big)\big\|_{\mathcal{Z}_{\frac{3}{2}}}^2\mathrm{d}s\nonumber\\
\leq\,& C(1+\delta)\mathcal{E}_{1,\infty}^2(t)\int_0^{\frac{t}{2}}(1+t-s)^{-\frac{3}{2}-n}
(1+s)^{-\frac{3}{2}}\mathrm{d}s\nonumber\\
&+C(1+\delta)\mathcal{E}_{1,\infty}^2(t)\int_{\frac{t}{2}}^t(1+t-s)^{-\frac{1}{2}}(1+s)^{-\frac{5}{2}-n}\mathrm{d}s\nonumber\\
\leq\,&C(1+t)^{-\frac{3}{2}-n}\mathcal{E}_{1,\infty}^2(t),
\end{align}
for $n=0,1$. 

Finally, for the remaining term $\nabla^nK^L_5(t)$, it follows from \eqref{G3.47}, \eqref{G5.24} and  $\mathbf{P}h=0$
that
\begin{align}\label{G5.34}
\|\nabla^nK^L_5(t)\|_{\mathcal{Z}_2}^2   \leq\,& C\int_0^{\frac{t}{2}}(1+t-s)^{-\frac{3}{2}-n}
\|\nu^{-\frac{1}{2}}\rho\mathcal{L}\{\mathbf{I}-\mathbf{P}\}f\|_{\mathcal{Z}_1}^2\mathrm{d}s\nonumber\\
&+ C\int_{\frac{t}{2}}^{t}(1+t-s)^{-\frac{1}{2}-n}
\|\nu^{-\frac{1}{2}}\rho\mathcal{L}\{\mathbf{I}-\mathbf{P}\}f\|_{\mathcal{Z}_{\frac{3}{2}}}^2\mathrm{d}s\nonumber\\
\leq\,& C(1+t)^{-\frac{3}{2}-n}\mathcal{E}_{1,\infty}(t)\int_0^{\frac{t}{2}}
\mathcal{D}(s)\mathrm{d}s\nonumber\\
&+C(1+t)^{-\frac{7}{4}}\mathcal{E}_{1,\infty}(t)\int_{\frac{t}{2}}^t\mathcal{D}(s)\mathrm{d}s\nonumber\\
\leq\,& C\mathcal{E}(0)(1+t)^{-\frac{7}{4}}\mathcal{E}_{1,\infty}(t),
\end{align}
for $n=0,1$.

Putting \eqref{G5.30}--\eqref{G5.34} into \eqref{G5.28}, we consequently obtain that 
\begin{align*}
\|\nabla^n U^L(t)\|_{\mathcal{Z}_2}\leq C(1+t)^{-\frac{3}{8}-\frac{n}{2}}\Big(   \|U_0\|_{\mathcal{Z}_1}^{\frac{1}{2}}\|U_0\|_{\mathcal{Z}_2}^{\frac{1}{2}}
+\mathcal{E}^{\frac{1}{2}}(0)\mathcal{E}^{\frac{1}{2}}_{1,\infty}(t)+\mathcal{E}_{1,\infty}(t)\Big),  
\end{align*}
for $n=0,1$.
This,  together with \eqref{G5.27},  implies that
\begin{align}\label{G5.36}
\|\nabla^2 U\|_{\mathcal{Z}_2}^2\leq C(1+t)^{-\frac{7}{4}}\big(\|U_0\|_{\mathcal{H}^2}^2+\|U_0\|_{\mathcal{Z}_1}\|U_0\|_{\mathcal{Z}_2}
+\mathcal{E}(0)\mathcal{E}_{1,\infty}(t)+\mathcal{E}_{1,\infty}^2(t)    \big).
\end{align}
Leveraging Lemma \ref{LA.7} and \eqref{G5.36}, we get
\begin{align*}
\|\nabla^n U(t)\|_{\mathcal{Z}_2}^2\leq &\,C(\|\nabla^n U^L(t)\|_{\mathcal{Z}_2}^2+\|\nabla^n U^H(t)\|_{\mathcal{Z}_2}^2)\nonumber\\
\leq\,& C(\|\nabla^n U^L(t)\|_{\mathcal{Z}_2}^2+\|\nabla^2 U(t)\|_{\mathcal{Z}_2}^2)\nonumber\\
\leq\,&C(1+t)^{-\frac{3}{4}-n}\big(\|U_0\|_{\mathcal{H}^2}^2
+\|U_0\|_{\mathcal{Z}_1}\|U_0\|_{\mathcal{Z}_2}+\mathcal{E}(0)\mathcal{E}_{1,\infty}(t)
+\mathcal{E}_{1,\infty}^2(t)    \big),
\end{align*}
for $n=0,1$.
Combining \eqref{G5.36} with the definition of $\mathcal{E}_{1,\infty}$ in
\eqref{G5.29} leads to
\begin{align*}
\mathcal{E}_{1,\infty}(t)\leq C \big(\|U_0\|_{\mathcal{H}^2}^2+\|U_0\|_{\mathcal{Z}_1}\|U_0\|_{\mathcal{Z}_2}+\mathcal{E}(0)\mathcal{E}_{1,\infty}(t)
+\mathcal{E}_{1,\infty}^2(t)    \big).   
\end{align*}
Using Lemma \ref{LA.4}, the smallness of $\mathcal{E}(0)$ and $\|U_0\|_{\mathcal{H}^2}^2$, one arrives at
\begin{align*}
\mathcal{E}_{1,\infty}(t)\leq C \big(\|U_0\|_{\mathcal{H}^2}^2+\|U_0\|_{\mathcal{Z}_1}\|U_0\|_{\mathcal{Z}_2}   \big),
\end{align*}
which means that 
\begin{align}
\|f\|_{L_v^2(L^2)}+\|(\rho,u)\| \leq\,& C(1+t)^{-\frac{3}{8}}\Big(\|U_0\|_{\mathcal{H}^2}+\|U_0\|_{\mathcal{Z}_1}^{\frac{1}{2}}\|U_0\|_{\mathcal{Z}_2}^{\frac{1}{2}}   \Big),  \label{decay-a}\\
\|\nabla (f,\rho,u)\|_{\mathcal{H}^1}\leq\,& C(1+t)^{-\frac{7}{8}}\Big(\|U_0\|_{\mathcal{H}^2}+\|U_0\|_{\mathcal{Z}_1}^{\frac{1}{2}}\|U_0\|_{\mathcal{Z}_2}^{\frac{1}{2}}   \Big).\label{decay-b}
\end{align}

\subsection{Optimal time-decay rates of 
solutions and their gradients }

Notice that the decay rates given in \eqref{decay-a} and \eqref{decay-b} 
are slower than the  optimal decay rates of the linearized equations. 
In order to obtain better results stated in Theorem \ref{T1.2}, below we give more refined energy estimates  of $(f,\rho,u)$.
Define
\begin{align}\label{G5.41}
\mathcal{E}_{\infty}(t):=\sup_{0\leq s\leq t}&\bigg\{\sum_{n=0}^1(1+s)^{\frac{3}{2}+n}\big(\|\nabla^n f\|_{L_v^2(L^2)}^2+\|\nabla^n(\rho,u)\|^2 \big)\nonumber\\
&\,+(1+s)^{\frac{5}{2}}\big(\|\nabla^2 f\|_{L_v^2(L^2)}^2+\|\nabla^2(\rho,u)\|^2 \big)   \bigg\}.
\end{align}

Now we give the proof of Theorem \ref{T1.2}.  
    Taking $q=1,m=0$ in \eqref{G5.23}, we have
\begin{align*}
\|\nabla^n K^L_1(t)\|_{\mathcal{Z}_2}\leq C(1+t)^{-\frac{3}{4}-\frac{n}{2}}
\|U_0\|_{\mathcal{Z}_{1}},
\end{align*}
for $n=0,1,2$.
Taking  a similar way to \eqref{G5.31}, it is obvious, for $\nabla^n K^L_3$ and $\nabla^n K^L_4$, that 
\begin{align}\label{G5.43}
&\|\nabla^n K^L_3(t)\|_{\mathcal{Z}_2}+\|\nabla^n K^L_4(t)\|_{\mathcal{Z}_2}\nonumber\\
\leq\,& C(1+\delta)\int_0^{\frac{t}{2}}(1+t-s)^{-\frac{3}{4}-\frac{n}{2}}\|(a,b,\rho,u)\|^2 
  \mathrm{d}s\nonumber\\
&+C(1+\delta)\int_{\frac{t}{2}}^t(1+t-s)^{-\frac{1}{4}}\|\nabla^n(a,b,\rho,u)\|\|\nabla(a,b,\rho,u)\|_{L^2}\mathrm{d}s\nonumber\\
\leq\,& C\Big(\|U_0\|_{\mathcal{H}^2}+\|U_0\|_{\mathcal{Z}_1}^{\frac{1}{2}}\|U_0\|_{\mathcal{Z}_2}^{\frac{1}{2}}   \Big)\mathcal{E}^{\frac{1}{2}}_{\infty}(t)\int_0^{\frac{t}{2}}(1+t-s)^{-\frac{3}{4}-\frac{n}{2}}(1+s)^{-\frac{9}{8}}
  \mathrm{d}s\nonumber\\
&+ C\Big(\|U_0\|_{\mathcal{H}^2}+\|U_0\|_{\mathcal{Z}_1}^{\frac{1}{2}}\|U_0\|_{\mathcal{Z}_2}^{\frac{1}{2}}   \Big)\mathcal{E}^{\frac{1}{2}}_{\infty}(t)\int_{\frac{t}{2}}^{t}(1+t-s)^{-\frac{1}{4}}(1+s)^{-\frac{13}{8}-\frac{n}{2}}
  \mathrm{d}s\nonumber\\
\leq\,& C(1+t)^{-\frac{3}{4}-\frac{n}{2}}\Big(\|U_0\|_{\mathcal{H}^2}+\|U_0\|_{\mathcal{Z}_1}^{\frac{1}{2}}\|U_0\|_{\mathcal{Z}_2}^{\frac{1}{2}}   \Big)\mathcal{E}^{\frac{1}{2}}_{\infty}(t),
\end{align}
for $n=0,1,2$.

Similarly, for the terms $\nabla^n K^L_6$ and $\nabla^n K^L_7$ with   $n=0,1$, we have 
\begin{align}\label{G5.44}
&\|\nabla^n K^L_6(t)\|_{\mathcal{Z}_2}+\|\nabla^n K^L_7(t)\|_{\mathcal{Z}_2}\nonumber\\
\leq \,&
C\int_0^{\frac{t}{2}}(1+t-s)^{-\frac{3}{4}-\frac{n}{2
}}\|(S_\rho,S_u) (s)\|_{\mathcal{Z}_1}\mathrm{d}s\nonumber\\
&+C\int_{\frac{t}{2}}^{t}(1+t-s)^{-\frac{1}{4}-\frac{n}{2}}\|(S_\rho,S_u-au )(s)\|_{\mathcal{Z}_{\frac{3}{2}}}\mathrm{d}s\nonumber\\
&+C\int_{\frac{t}{2}}^{t}(1+t-s)^{-\frac{1}{4}}\|\nabla^n(au )(s)\|_{\mathcal{Z}_{\frac{3}{2}}}\mathrm{d}s\nonumber\\
\leq\,&C\int_0^{\frac{t}{2}}(1+t-s)^{-\frac{3}{4}-\frac{n}{2}}\big(\|(\rho,u)\|\|\nabla (\rho,u)\|+\|\rho\|\|\nabla^2u\|+\|a\|\|u\|\big)\mathrm{d}s\nonumber\\
&+C\int_{\frac{t}{2}}^t(1+t-s)^{-\frac{1}{4}-\frac{n}{2}}\big(\|(\rho,u)\|_{L^6}\|\nabla(\rho,u)\|+\|\rho\|_{L^6}\|\nabla^2 u\|\big)\mathrm{d}s\nonumber\\
&+C\int_{\frac{t}{2}}^t(1+t-s)^{-\frac{1}{4}}\big(\|\nabla^n a\|\|u\|_{L^6}+\|\nabla^n u\|\|a\|_{L^6}     \big)\mathrm{d}s\nonumber\\
\leq\,&C\Big(\|U_0\|_{\mathcal{H}^2}+\|U_0\|_{\mathcal{Z}_1}^{\frac{1}{2}}\|U_0\|_{\mathcal{Z}_2}^{\frac{1}{2}}   \Big)\mathcal{E}^{\frac{1}{2}}_{\infty}(t)\int_0^{\frac{t}{2}}(1+t-s)^{-\frac{3}{4}-\frac{n}{2}}(1+s)^{-\frac{9}{8}}
  \mathrm{d}s\nonumber\\
&+C\Big(\|U_0\|_{\mathcal{H}^2}+\|U_0\|_{\mathcal{Z}_1}^{\frac{1}{2}}\|U_0\|_{\mathcal{Z}_2}^{\frac{1}{2}}   \Big)\mathcal{E}^{\frac{1}{2}}_{\infty}(t)\int_{\frac{t}{2}}^t(1+t-s)^{-\frac{1}{4}-\frac{n}{2}}(1+s)^{-\frac{17}{8}}
  \mathrm{d}s\nonumber\\
&+C\Big(\|U_0\|_{\mathcal{H}^2}+\|U_0\|_{\mathcal{Z}_1}^{\frac{1}{2}}\|U_0\|_{\mathcal{Z}_2}^{\frac{1}{2}}   \Big)\mathcal{E}^{\frac{1}{2}}_{\infty}(t)\int_{\frac{t}{2}}^t(1+t-s)^{-\frac{1}{4}}(1+s)^{-\frac{13}{8}-\frac{n}{2}}
  \mathrm{d}s\nonumber\\
\leq\,& C(1+t)^{-\frac{3}{4}-\frac{n}{2}}\Big(\|U_0\|_{\mathcal{H}^2}+\|U_0\|_{\mathcal{Z}_1}^{\frac{1}{2}}\|U_0\|_{\mathcal{Z}_2}^{\frac{1}{2}}   \Big)\mathcal{E}^{\frac{1}{2}}_{\infty}(t). 
\end{align}
For the case  $n=2$,  $ \nabla^n K^L_6$ and $\nabla^n K^L_7$ can be estimated as 
and 
\begin{align}\label{G5.45}
&\|\nabla^2 K^L_6(t)\|_{\mathcal{Z}_2}+\|\nabla^2 K^L_7(t)\|_{\mathcal{Z}_2}\nonumber\\
\leq \,&
C\int_0^{\frac{t}{2}}(1+t-s)^{-\frac{3}{4}-1}\|(S_\rho,S_u) (s)\|_{\mathcal{Z}_1}\mathrm{d}s\nonumber\\
&+C\int_{\frac{t}{2}}^{t}(1+t-s)^{-\frac{1}{4}-1}\|(S_\rho,S_u-au )(s)\|_{\mathcal{Z}_{\frac{3}{2}}}\mathrm{d}s\nonumber\\
&+C\int_{\frac{t}{2}}^{t}(1+t-s)^{-\frac{1}{4}}\|\nabla (au )(s)\|_{\mathcal{Z}_{\frac{3}{2}}}\mathrm{d}s\nonumber\\
\leq\,&C\int_0^{\frac{t}{2}}(1+t-s)^{-\frac{3}{4}-1}\big(\|(\rho,u)\|\|\nabla (\rho,u)\|+\|\rho\|\|\nabla^2u\|+\|a\|\|u\|\big)\mathrm{d}s\nonumber\\
&+C\int_{\frac{t}{2}}^t(1+t-s)^{-\frac{1}{4}-1}\big(\|(\rho,u)\|_{L^6}\|\nabla(\rho,u)\|+\|\rho\|_{L^6}\|\nabla^2 u\|\big)\mathrm{d}s\nonumber\\
&+C\int_{\frac{t}{2}}^t(1+t-s)^{-\frac{3}{4}}\big(\|\nabla  a\|\|u\|_{L^6}+\|\nabla^{n-1} u\|\|a\|_{L^6}     \big)\mathrm{d}s\nonumber\\
\leq\,&C\Big(\|U_0\|_{\mathcal{H}^2}+\|U_0\|_{\mathcal{Z}_1}^{\frac{1}{2}}\|U_0\|_{\mathcal{Z}_2}^{\frac{1}{2}}   \Big)\mathcal{E}^{\frac{1}{2}}_{\infty}(t)\int_0^{\frac{t}{2}}(1+t-s)^{-\frac{3}{4}-1}(1+s)^{-\frac{9}{8}}
  \mathrm{d}s\nonumber\\
&+C\Big(\|U_0\|_{\mathcal{H}^2}+\|U_0\|_{\mathcal{Z}_1}^{\frac{1}{2}}\|U_0\|_{\mathcal{Z}_2}^{\frac{1}{2}}   \Big)\mathcal{E}^{\frac{1}{2}}_{\infty}(t)\int_{\frac{t}{2}}^t(1+t-s)^{-\frac{1}{4}-1}(1+s)^{-\frac{17}{8}}
  \mathrm{d}s\nonumber\\
&+C\Big(\|U_0\|_{\mathcal{H}^2}+\|U_0\|_{\mathcal{Z}_1}^{\frac{1}{2}}\|U_0\|_{\mathcal{Z}_2}^{\frac{1}{2}}   \Big)\mathcal{E}^{\frac{1}{2}}_{\infty}(t)\int_{\frac{t}{2}}^t(1+t-s)^{-\frac{1}{4}}(1+s)^{-\frac{9}{8}-1}
  \mathrm{d}s\nonumber\\
\leq\,& C(1+t)^{-\frac{3}{4}-1}\Big(\|U_0\|_{\mathcal{H}^2}+\|U_0\|_{\mathcal{Z}_1}^{\frac{1}{2}}\|U_0\|_{\mathcal{Z}_2}^{\frac{1}{2}}   \Big)\mathcal{E}^{\frac{1}{2}}_{\infty}(t).
\end{align} 

Thus, from \eqref{G5.43}--\eqref{G5.45}, it holds
\begin{align*}
&\|\nabla^n K^L_6(t)\|_{\mathcal{Z}_2}+\|\nabla^n K^L_7(t)\|_{\mathcal{Z}_2} \leq C(1+t)^{-\frac{3}{4}-\frac{n}{2}}\Big(\|U_0\|_{\mathcal{H}^2}+\|U_0\|_{\mathcal{Z}_1}^{\frac{1}{2}}\|U_0\|_{\mathcal{Z}_2}^{\frac{1}{2}}   \Big)\mathcal{E}^{\frac{1}{2}}_{\infty}(t),   
\end{align*}
for $n=0,1,2$.

Similar to \eqref{G5.33} and \eqref{G5.34}, we have
\begin{align*}
\|\nabla^nK^L_2(t)\|_{\mathcal{Z}_2}^2\leq\,& C\int_0^{\frac{t}{2}}(1+t-s)^{-\frac{3}{2}-n}
\|(1+\rho)u\{\mathbf{I}-\mathbf{P}_0\}f\|_{\mathcal{Z}_1}^2\mathrm{d}s\nonumber\\
&+ C\int_{\frac{t}{2}}^t(1+t-s)^{-\frac{1}{2}}
\big\|\nabla^n\big((1+\rho)u\{\mathbf{I}-\mathbf{P}_0\}f\big)\big\|_{\mathcal{Z}_{\frac{3}{2}}}^2\mathrm{d}s\nonumber\\
\leq\,&C(1+t)^{-\frac{3}{2}-n}\Big(\|U_0\|_{\mathcal{H}^2}^2+\|U_0\|_{\mathcal{Z}_1}\|U_0\|_{\mathcal{Z}_2}   \Big)\mathcal{E}_{\infty}(t),  
\end{align*}
for $n=0,1,2$, and
\begin{align*}
\|\nabla^nK^L_5(t)\|_{\mathcal{Z}_2}^2   \leq\,& C\int_0^{\frac{t}{2}}(1+t-s)^{-\frac{3}{2}-n}
\|\nu^{-\frac{1}{2}}\rho\mathcal{L}\{\mathbf{I}-\mathbf{P}\}f\|_{\mathcal{Z}_1}^2\mathrm{d}s\nonumber\\
&+ C\int_{\frac{t}{2}}^{t}(1+t-s)^{-\frac{1}{2}-n}
\|\nu^{-\frac{1}{2}}\rho\mathcal{L}\{\mathbf{I}-\mathbf{P}\}f\|_{\mathcal{Z}_{\frac{3}{2}}}^2\mathrm{d}s\nonumber\\
\leq\,& C(1+t)^{-\frac{3}{2}-n}\Big(\|U_0\|_{\mathcal{H}^2}+\|U_0\|_{\mathcal{Z}_1}^{\frac{1}{2}}\|U_0\|_{\mathcal{Z}_2}^{\frac{1}{2}}   \Big)\mathcal{E}^{\frac{1}{2}}_{\infty}(t)\int_0^{\frac{t}{2}}
\mathcal{D}(s)\mathrm{d}s\nonumber\\
&+C(1+t)^{-\frac{5}{2}}\mathcal{E}_{\infty}(t)\int_{\frac{t}{2}}^t\mathcal{D}(s)\mathrm{d}s\nonumber\\
\leq\,& C\mathcal{E}(0)(1+t)^{-\frac{5}{2}}\mathcal{E}_{\infty}(t),
\end{align*}
for $n=0,1,2$.
As a result, we obtain
\begin{align}\label{G5.49}
\|\nabla^n U^L(t)\|_{\mathcal{Z}_2}\leq C(1+t)^{-\frac{3}{4}-\frac{n}{2}} \Big(\|U_0\|_{\mathcal{H}^2}+\|U_0\|_{\mathcal{Z}_1}^{\frac{1}{2}}\|U_0\|_{\mathcal{Z}_2}^{\frac{1}{2}}+\mathcal{E}^{\frac{1}{2}}(0)   \Big)\mathcal{E}^{\frac{1}{2}}_{\infty}(t),   
\end{align}
for $n=0,1$, and
\begin{align}\label{G5.50}
\|\nabla^n U^L(t)\|_{\mathcal{Z}_2}\leq C(1+t)^{-\frac{5}{4}} \Big(\|U_0\|_{\mathcal{H}^2}+\|U_0\|_{\mathcal{Z}_1}^{\frac{1}{2}}\|U_0\|_{\mathcal{Z}_2}^{\frac{1}{2}}+\mathcal{E}^{\frac{1}{2}}(0)   \Big)\mathcal{E}^{\frac{1}{2}}_{\infty}(t),   
\end{align}
for $n=2$.
Combining \eqref{G5.49} and \eqref{G5.50} with \eqref{G5.27}, we can infer that
\begin{align}\label{G5.51}
\|\nabla^2U(t)\|_{\mathcal{Z}_2}^2\leq C(1+t)^{-\frac{5}{2}} \Big\{\Big(\|U_0\|_{\mathcal{H}^2}^2+\|U_0\|_{\mathcal{Z}_1}\|U_0\|_{\mathcal{Z}_2}+\mathcal{E}(0)    \Big)\mathcal{E}_{\infty}(t)+\|U_0\|_{\mathcal{H}^2\cap{\mathcal{Z}_1}}^2  \Big\}.
\end{align}
Making use of Lemma \ref{LA.7} and \eqref{G5.51} yields
\begin{align*}
\|\nabla^n U(t)\|_{\mathcal{Z}_2}^2\leq \,&C(\|\nabla^n U^L(t)\|_{\mathcal{Z}_2}^2+\|\nabla^n U^H(t)\|_{\mathcal{Z}_2}^2)\nonumber\\
\leq\,& C(\|\nabla^n U^L(t)\|_{\mathcal{Z}_2}^2+\|\nabla^2 U(t)\|_{\mathcal{Z}_2}^2)\nonumber\\
\leq\,&C(1+t)^{-\frac{3}{2}-n}\Big\{\big(\|U_0\|_{\mathcal{H}^2}^2+\|U_0\|_{\mathcal{Z}_1}\|U_0\|_{\mathcal{Z}_2}
+\mathcal{E}(0)    \big)\mathcal{E}_{\infty}(t)+\|U_0\|_{\mathcal{H}^2\cap{\mathcal{Z}_1}}^2  \Big\},
\end{align*}
for $n=0,1$.
This, together with the definition of $\mathcal{E}_{\infty}(t)$ in \eqref{G5.41}, leads to
\begin{align*}
\mathcal{E}_{\infty}(t)\leq C \Big\{\big(\|U_0\|_{\mathcal{H}^2}^2+\|U_0\|_{\mathcal{Z}_1}\|U_0\|_{\mathcal{Z}_2}+\mathcal{E}(0)    \big)\mathcal{E}_{\infty}(t)+\|U_0\|_{\mathcal{H}^2\cap{\mathcal{Z}_1}}^2  \Big\}.   
\end{align*}
Thanks to the smallness of $ \|U_0\|_{\mathcal{H}^2}^2+\|U_0\|_{\mathcal{Z}_1}\|U_0\|_{\mathcal{Z}_2}+\mathcal{E}(0)  $, we have
\begin{align*}
\mathcal{E}_{\infty}(t)\leq C \|U_0\|_{\mathcal{H}^2\cap{\mathcal{Z}_1}}^2,   
\end{align*}
which ensures that \eqref{G1.8}--\eqref{G1.10} hold.

Based on the decay rates obtained in \eqref{G1.8}--\eqref{G1.10}, we can further investigate  
the  time-decay rates of strong solutions in $L^p$ space with  $p\in [2,+\infty]$. 
In fact, by \eqref{G1.8}--\eqref{G1.9}, we notice that  
\begin{align}
 \|(\rho,u)\|_{L^6}\leq \,& C\|\nabla(\rho,u)\|_{L^2}\leq C(1+t)^{-\frac{5}{4}},\label{A1}  \\
 \| f\|_{L_v^2(L_x^{6})}\leq \,& C \|\nabla f\|_{L_v^2(L_x^{2})}\leq C(1+t)^{-\frac{5}{4}},  \label{A3}   \\
 \|(\rho,u)\|_{L^2}\leq \,& C(1+t)^{-\frac{3}{4}},  \label{A4}\\
 \|f\|_{L_v^2(L_x^{2})}\leq \,& C(1+t)^{-\frac{3}{4}}.\label{A2}
\end{align}

For $p\in [2,6]$, using the interpolation inequality, it follows from \eqref{A1}--\eqref{A2} that
\begin{align*}
 \|(\rho,u)\|_{L^p}\leq \,& C\|(\rho,u)\|_{L^2}^{\zeta}\|(\rho,u)\|_{L^6}^{{1-\zeta}}\leq C(1+t)^{-\frac{3}{2}(1-\frac{1}{p})} , \\
 \|f\|_{L_v^2(L_x^{p})}\leq \,& C\|f\|_{L_v^2(L_x^{2})}^{\zeta}\|f\|_{L_v^2(L_x^{6})}^{{1-\zeta}}\leq C(1+t)^{-\frac{3}{2}(1-\frac{1}{p})},
\end{align*}
where $\zeta= ({6-p})/ {2p} \in [0,1]$.

By  Gagliardo-Nierenberg's inequality, we get
\begin{align*}
&\|(\rho,u)\|_{L^{\infty}}\leq C\|\nabla(\rho,u)\|_{H^{1}}\leq C(1+t)^{-\frac{5}{4}} , \\
&\|f\|_{L_v^2(L_x^{\infty})}\leq C\|\nabla f\|_{L_v^2(H_x^{1})}\leq C(1+t)^{-\frac{5}{4}}.
\end{align*}
For $p\in [6,\infty]$, using Lemma \ref{LLA.4}, it holds
\begin{align*}
\|(\rho,u)\|_{L^p}&\,\leq C\|(\rho,u)\|_{L^6}^{\zeta^{\prime}}\|(\rho,u)\|_{L^{\infty}}^{{1-\zeta^{\prime}}}\leq C(1+t)^{-\frac{5}{4}} , \\
\|f\|_{L_v^2(L_x^{p})}&\,\leq C\|f\|_{L_v^2(L_x^{6})}^{\zeta^{\prime}}\|f\|_{L_v^2(L_x^{\infty})}^{{1-\zeta^{\prime}}}\leq C(1+t)^{-\frac{5}{4}},
\end{align*}
where $\zeta^{\prime}= {6}/{p} \in [0,1]$.
This ensures \eqref{G1.11} and \eqref{G1.12} hold. Hence, we complete the proof of Theorem \ref{T1.2}.
\hfill $\square$

\begin{rem}
The term  $\rho\mathcal{L}$ appearing in $K_5$ in \eqref{G4.27}
brings us some  difficulties in the above 
  estimates. The reason is that this term $K_5$ causes energy losses, which 
  prevent us from obtaining   the optimal decay rates for
the second  order derivative of $(f,\rho,u)$. 
This is quite different from the results in \cite{Ww-CMS-2024,Li-Ni-Wu}
for the density-independent friction force case \eqref{fo-aaaa}.
\end{rem}

\section{The periodic domain case}

In this section, we deal with the Cauchy problem \eqref{I-3}--\eqref{I-6} in the spatial periodic domain $\mathbb{T}^3$.  We have

\begin{thm}\label{T1.3} 
Assume that  the initial data  $(f_0, \rho_0, u_0)$
satisfy   that $F_0=M+$ $\sqrt{M} f_0 \geq 0$, 
  $\|f_0\|_{H_{x, v}^2}+\|(\rho_0, u_0)\|_{H^2}$ is sufficient small, and
\begin{gather*}
  \int_{\mathbb{T}^3} a_0 \, \mathrm{d} x=0, \quad \int_{\mathbb{T}^3} \rho_0 \, \mathrm{d} x=0, \quad 
  \int_{\mathbb{T}^3}\big(b_0+(1+\rho_0) u_0\big) \mathrm{d} x=0, 
\end{gather*}
where
\begin{align*}
a_0=\int_{\mathbb{T}^3} \sqrt{M} f_0(x, v) \mathrm{d} v,\ \  b_0=\int_{\mathbb{T}^3} v \sqrt{M} f_0(x, v) \mathrm{d} v.
\end{align*}
Then, the Cauchy problem \eqref{I-3}--\eqref{I-6}  admits a unique global strong solution $(f, \rho, u)$ satisfying $F=M+\sqrt{M} f \geq 0$, and
\begin{gather*}
  f \in C([0, \infty) ; H_{x,v}^2), \quad\rho, u \in C([0, \infty) ; H^2 ), \\
  \|f(t)\|_{H_{x, v}^2}+\|(\rho, u)(t)\|_{H^2} \leq C\big(\|f_0\|_{H_{x, v}^2}+\|(\rho_0, u_0)\|_{H^2}\big) e^{-\zeta_0 t}, 
\end{gather*}
with some constant $\zeta_0>0$, for any $t \geq 0$.
\end{thm}

\begin{proof}

First, by a direct calculation,  
the following conservation laws can be obtained from \eqref{I-1}:
\begin{gather*}
 \frac{\mathrm{d}}{\mathrm{d} t} \iint_{\mathbb{R}^3\times\mathbb{T}^3} F \mathrm{d} x \mathrm{d} v=0,\quad 
 \frac{\mathrm{d}}{\mathrm{d} t} \int_{\mathbb{T}^3} n \mathrm{d} x=0,\\
  \frac{\mathrm{d}}{\mathrm{d} t}\left(\int_{\mathbb{T}^3} n u \mathrm{d} x+\iint_{\mathbb{R}^3\times\mathbb{T}^3} v F \mathrm{d} x \mathrm{d} v\right)=0.
\end{gather*}
Then, according to the assumption in Theorem \ref{T1.3}, we derive that 
\begin{align}\label{G6.1}
   \int_{\mathbb{T}^3} a \mathrm{d} x=0, \quad 
   \int_{\mathbb{T}^3} \rho \mathrm{d} x=0, \quad 
   \int_{\mathbb{T}^3}\big(b+(1+\rho) u\big) \mathrm{d} x=0, 
\end{align}
for any $t\geq 0$.


Due to the fact that the proof of local existence and uniqueness of $(f,\rho,u)$ in 
  $ \mathbb{T}^3$ is similar to that in the
whole space $ \mathbb{R}^3$  case, we only show the proof of  the global a priori 
estimates here. From the Poincaré's inequality and \eqref{G6.1}, one gets
\begin{align}\label{G6.2}
\|a\|_{L^2}\leq\,& C\|\nabla a\|_{L^2},  \nonumber\\
\|\rho\|_{L^2}\leq\,& C\|\nabla \rho\|_{L^2},\nonumber\\
\|u+b\|_{L^2}\leq\,&C\|b+\rho(1+u)\|_{L^2}+\|\rho u\|_{L^2}\nonumber\\
\leq\,&C\big\|\nabla\big(  b+\rho(1+u)          \big)\big\|_{L^2}+\|\rho\|_{L^3}\|u\|_{L^{6}}\nonumber\\
\leq\,&C\|\nabla(b,u)\|_{L^2}+C\|\nabla u\|_{L^3}\|\rho\|_{L^6}+C\|u\|_{L^{\infty}}\|\nabla\rho\|_{L^2}+\|\rho\|_{H^1}\|\nabla u\|_{L^2}\nonumber\\
\leq\,& C\|\nabla(b,u,\rho)\|_{L^2}.
\end{align}
Thus, we deduce from \eqref{G6.2} and basic triangle inequality that
\begin{align}
\|u\|_{L^2}\leq\,& C\|u-b\|_{L^2}+C\|u+b\|_{L^2}\leq C\|u-b\|_{L^2}+C\|\nabla(b,u,\rho)\|_{L^2},\label{G6.22}\\
\|b\|_{L^2}\leq\,& C\|u-b\|_{L^2}+\|u\|_{L^2}\leq C\|u-b\|_{L^2}+C\|\nabla(b,u,\rho)\|_{L^2}.\label{G6.3}
\end{align}
Now, {let us} define the energy functional $\mathcal{E}(t)$ and the corresponding dissipation rate functional $\mathcal{D}(t)$ on $\Omega = \mathbb{T}^3$, exactly as in the case $\Omega = \mathbb{R}^3$ (see \eqref{G3.45}). And we also define 
\begin{align}\label{G6.4}
\mathcal{D}_{\mathbb{T}}(t):=\mathcal{D}(t)+\|(a,b,\rho,u)\|_{L^2}.    
\end{align}
Applying 
the similar process on  getting  \eqref{G3.46} to the $\mathbb{T}^3$   case, then combining
\eqref{G6.2}--\eqref{G6.4}, and using the fact that $\mathcal{E}(t)\leq C\mathcal{D}_{\mathbb{T}}(t)$, we eventually obtain
\begin{align}\label{G6.5}
\frac{{\rm d}}{{\rm d}t}\mathcal{E}(t)+\lambda \mathcal{E}(t)\leq 0,
\end{align}
for some $\lambda>0$.
Then, from Gronwall's inequality, it holds
\begin{align*}
\mathcal{E}(t)\leq e^{-\lambda t}\mathcal{E}(0),
\end{align*}
which leads to the exponential decay rate of $\mathcal{E}(t)$. Noticing that $\mathcal{E}(t)  \backsim\|(f
,\rho,u)\|_{\mathcal{H}^2}^2$ in time, hence, we complete the proof 
of Theorem \ref{T1.3}.
\end{proof}

\appendix

\section{some useful lemmas}
Below we list  some useful lemmas which will be used throughout in this paper.

\begin{lem} [{See \cite[Lemma 2.1]{CDM-krm-2011}}]\label{LA.1}   
There exists $a$ positive constant C 
such that for any $g,h\in H^2(\mathbb{R}^3)$ and any multi-index $\alpha$  with $1\leq|\alpha|\leq2$, it holds
\begin{align*}
\|g\|_{L^{\infty}(\mathbb{R}^{3})}& \leq C\|\nabla_x g\|_{L^{2}(\mathbb{R}^{3})}^{1/2}
\|\nabla_x^{2}g\|_{L^{2}(\mathbb{R}^{3})}^{1/2}, \\
\|gh\|_{H^{1}(\mathbb{R}^{3})}& \leq C\|g\|_{H^{1}(\mathbb{R}^{3})}\|\nabla_x h\|_{H^{1}(\mathbb{R}^{3})}, \\
\|\partial_x^{\alpha}(gh)\|_{L^{2}(\mathbb{R}^{3})}
& \leq C\|\nabla_x g\|_{H^{1}(\mathbb{R}^{3})}\|\nabla_x h\|_{H^{1}(\mathbb{R}^{3})}. 
\end{align*}
\end{lem} 

\begin{lem}[{See \cite[Lemmas 2.1--2.2]{Dk-MZ-1992}}]\label{LA.2}
There exist positive constants C  such that  for any $g \in H^1(\mathbb{R}^3)$, one has
\begin{align*}
\|g\|_{L^6(\mathbb{R}^3)} \leq\,& C\|\nabla_x g\|_{L^2(\mathbb{R}^3)}\leq C\|g\|_{H^1(\mathbb{R}^3)},
\end{align*}
and
\begin{align*}
  \|g\|_{L^q(\mathbb{R}^3)} \leq\,& C\|g\|_{H^1(\mathbb{R}^3)}, 
\end{align*}
for $2 \leq q \leq 6$.
\end{lem}

\begin{lem}[see \cite{AF-Pa-2003,MB-book-2002}]\label{LLA.3}
Let $k\geq 1$ be an integer, then we have
\begin{align*}
\|\nabla^{k}(gh) \|_{L^r\left(\mathbb{R}^3\right)} & \leq C\|g\|_{L^{r_1}\left(\mathbb{R}^3\right)}\|\nabla^{k}h\|_{L^{r_2}\left(\mathbb{R}^3\right)}+C\|h\|_{L^{r_3}\left(\mathbb{R}^3\right)}\|\nabla^{k}g\|_{L^{r_4}\left(\mathbb{R}^3\right)},
\end{align*}
and
\begin{align*}
\|\nabla^{k}(gh)-g\nabla^k h \|_{L^r\left(\mathbb{R}^3\right)} & \leq C\|\nabla g\|_{L^{r_1}\left(\mathbb{R}^3\right)}\|\nabla^{k-1}h\|_{L^{r_2}\left(\mathbb{R}^3\right)}+C\|h\|_{L^{r_3}\left(\mathbb{R}^3\right)}\|\nabla^{k}g\|_{L^{r_4}\left(\mathbb{R}^3\right)},    
\end{align*}
with $1<r,r_2,r_4<\infty$ and $r_i(1\leq i\leq 4)$ satisfy the following identity:
\begin{align*}
\frac{1}{r_1}+\frac{1}{r_2}=\frac{1}{r_3}+\frac{1}{r_4}=\frac{1}{r}. 
\end{align*}
\end{lem}

\begin{lem}[see\cite{AF-Pa-2003}]\label{LLA.4}
Suppose that $1\leq r\leq s\leq q\leq \infty$, and
\begin{align*}
\frac{1}{s}=\frac{\zeta}{r}+\frac{1-\zeta}{q},
\end{align*}
where $0\leq \zeta\leq 1$.
Assume also $g\in L^r(\mathbb{R}^3)\cap L^q(\mathbb{R}^3)$. Then $g\in L^s(\mathbb{R}^3)$, and
\begin{align*}
\|g\|_{L^s}\leq \|g\|_{L^r}^\zeta \|g\|_{L^q}^{1-\zeta}.    
\end{align*}
\end{lem}

\begin{lem} [{See \cite[Lemma 3.2]{CDM-krm-2011}}]\label{LA.3}   
Given any $0<\beta_1\ne 1$ and $\beta_2>1$, then
\begin{align*}
\int_0^t (1+t-s)^{-\beta_1}(1+s)^{-\beta_2} {\rm d}s \leq C(1+t)^{-\min\{\beta_1,\beta_2\}}, 
\end{align*}
for all $t\geq 0$.
\end{lem} 

\begin{lem} [{See \cite[Lemma 3.3]{CDM-krm-2011}}]\label{LA.4}   
Let $\gamma>1$ and $g_1,g_2\in C(\mathbb{R}_+,\mathbb{R}_+)$ with
$g_1(0)=0$. For $A\in \mathbb{R}_+$, define $\mathcal{B}_{A}:=\{y\in C(\mathbb{R}_+,(\mathbb{R}_+)|\, \,y\leq A+g_1(A)y+g_2(A)y^{\gamma},\,\, y(0)\leq A\}$.
Then, there exists a constant $A_0\in (0,\,\min\{A_1,A_2\})$ such that for
any $0<A\leq A_0$, 
\begin{align*}
y\in \mathcal{B}_A \Rightarrow \sup_{t\geq 0}y(t)\leq 2A.
\end{align*}
\end{lem} 

\medskip
Finally, 
for a  function $g(x)\in L^2(\mathbb{R}^3)$, we define the   \emph{low and high frequencies decompositions
$g^L$ and $g^H$ } of 
it as 
\begin{align*}
g^{L}(x)=\phi_0(D_x)g(x),\quad g^{H}(x)=\phi_1(D_x)g(x).
\end{align*}
Here, $D_x=\frac{1}{\sqrt{-1}}{(\partial_{x_1},\partial_{x_2},\partial_{x_3})}$, and  $\phi_0(D_x)$ 
and $\phi_1(D_x)$ are two pseudo-differential operators
corresponding to the smooth cut-off functions $\phi_0(k)$ and $\phi_1(k)$ satisfying
$0\leq \phi_0(k)\leq 1$, $\phi_1(k)=1-\phi_0(k)$, where
\begin{align*}
\phi_0(k)=\begin{cases}1, \quad|k|\leq \frac{r_0}{2},\\0, \quad|k|>{r_0},\end{cases}
\end{align*}
for a fixed constant $r_0>0$. It is obvious that
\begin{align*}
 g(x)=g^{L}(x)+g^{H}(x),   
\end{align*}
which together with Plancherel Theorem gives rise to the following lemma.
\begin{lem}[see \cite{WW-SCM-2022}]\label{LA.7}
Let $g\in H^2(\mathbb{R}^3)$. Then, we have
\begin{gather*}
 \|g^H\|_{L^2}\leq C\|\nabla g\|_{L^2},\quad \|g^H\|_{L^2}\leq C\|\nabla^2 g\|_{L^2}, \\
 \|\nabla^2 g^L\|_{L^2}\leq C\|\nabla g^L\|_{L^2}, 
\end{gather*}
for some constant $C>0$.
\end{lem}

 \bigskip 
	\noindent
{\bf Acknowledgements:} 
The authors are very grateful to the reviewer  for his/her constructive comments and helpful suggestions, which considerably improved the earlier version of this paper.
Li and   Ni  are supported by NSFC (Grant No.  12331007).  
 Li is also supported by the “333 Project" of Jiangsu Province.
Wu is supported by the Basic Science (Natural Science) Research Project of Colleges and Universities of Jiangsu Province (Grant No. 24KJD110004) and Shandong Provincial Natural Science Foundation (Grant No. ZR2024MA078).

\vspace{2mm}
		\noindent
	 \noindent \textbf{Conflict of interest.} The authors have no conflicts of interest.

	\vspace{2mm}

	\noindent \textbf{Data availability statement.}
	 Data sharing is not applicable to this article as no data sets were generated or analyzed during the current study.

\bibliographystyle{plain}

\end{document}